\documentclass[11pt]{article}
\usepackage{mathrsfs}  
\usepackage{amsthm}
\usepackage{thmtools}
\usepackage{amsmath}
\usepackage{amssymb}
\usepackage{bm}
\usepackage{color}
\usepackage{natbib}
\usepackage{mathtools}
\usepackage{xcolor}
\mathtoolsset{showonlyrefs}

\usepackage{hyperref}
\usepackage{cleveref}
\usepackage{algorithm}
\usepackage{algpseudocode}

\newcommand{\1}{\text{\usefont{U}{bbold}{m}{n}1}}
\usepackage[margin=1.25in]{geometry}

\crefname{assumption}{assumption}{assumptions}
\Crefname{assumption}{Assumption}{Assumptions}

\theoremstyle{plain}

\newtheorem{lemma}{Lemma}

\theoremstyle{definition}
\newtheorem{definition}{Definition}

\newtheorem{example}{Example}

\usepackage{notation}

\theoremstyle{plain}
\newtheorem{assumption}{Assumption}

\theoremstyle{plain}
\newtheorem{remark}{Remark}

\title{Semiparametric inference on identification sets in choice modeling}

\usepackage{authblk}

\begin{document}

\author[* 1, 2]{Antoine Scheid}
\author[* 3]{Jia Wan}

\author[4]{Guy Aridor}
\author[2,5]{Nathan Kallus}
\author[2]{Aur\'elien Bibaut}

\affil[1]{Ecole polytechnique}
\affil[2]{Netflix}
\affil[3]{MIT}
\affil[4]{Northwestern University, Kellogg School of Management}
\affil[5]{Cornell University and Cornell Tech}
\affil[*]{Equal contribution, alphabetical order}

\maketitle

\begin{abstract}
    In a discrete choice model, choice probabilities observed for a finite collection of choice sets may not identify a counterfactual choice probability under an unobserved choice set. We represent this counterfactual probability as a linear functional of a mixing distribution. Because the target is a functional of a distribution whose support is not restricted to a finite set, the parameter space is infinite-dimensional, while the data impose only finitely many moment restrictions. Therefore, observed choice probabilities need not point identify such a target. The identified set is defined as the set of target values compatible with observed choice probabilities. Rather than imposing conditions to ensure point identification, we characterize the identified set, and conduct inference on its lower and upper endpoints. We represent each endpoint as the value of a linear program over probability measures, and give conditions to obtain pathwise differentiability of the identification bounds. As a consequence, we are able to prove asymptotic normality of plug-in endpoint estimators. Finally, we provide an Expectation-Maximization-like algorithm for certifying membership of candidate values in the identified set and establish local convergence guarantees.
\end{abstract}

\section{Introduction}

Discrete choice models quantify how demand changes when available alternatives
change. Transportation agencies use choice models to evaluate new modes of
travel \citep{ben1985discrete}, retailers use them to evaluate product
assortments \citep{guadagni1983logit}, and media platforms use them to evaluate
catalogs and recommendations \citep{chou2024estimating}. Each application can
require choice probabilities under a choice set that is absent from the data
used to fit the model. We refer to such a set as a \emph{counterfactual choice
set}. A random utility model (RUM) imposes a rationality restriction on the collection of choice probabilities across choice sets by requiring the choice probabilities to be generated by the maximization of latent utilities. In statistical terms, a RUM is a semiparametric model, since rationality restricts the choice probabilities, while the latent utility distribution remains infinite-dimensional. The multinomial logit model (MNL) is a simple RUM in which each alternative’s utility is the sum of a systematic component and an independent type-I extreme-value shock. Its pairwise choice and ranking counterparts are commonly known as the Bradley–Terry \citep{bradley1952rank} and Plackett–Luce models \citep{plackett1975analysis,luce1959individual}, respectively. Beyond its modeling simplicity, a celebrated empirical application of MNLs is the extremely accurate transportation demand substitution prediction induced by the introduction of the BART in the Bay Area \citep{mcfadden1974measurement}. Yet, the MNL model implies independence of irrelevant alternatives (IIA)--type substitution patterns that may be implausible in some settings \citep{arrow2012social, ray1973independence}. In the well-known red bus/blue bus example \citep{train2009discrete}, a consumer initially chooses equally often between a car and a blue bus. Introducing a red bus that is otherwise identical to the blue bus should primarily split the original bus demand. Instead, multinomial logit assigns equal probability to all three alternatives, reducing the car share from one half to one third and increasing total bus demand from one half to two thirds. 

A mixed MNL model represents choice probabilities as mixtures of MNL choice
probabilities \citep{revelt1998mixed,mcfadden2000mixed}. Each component MNL
model is indexed by what we call a \emph{preference type}, which determines the
systematic utility assigned to each alternative and hence the component's MNL
choice probabilities. For any choice set, the mixed MNL probability of choosing
an alternative is the weighted average, across preference types, of the
corresponding component MNL probabilities, with weights specified by a mixing
measure. Under a population interpretation, the mixture describes a population
of MNL decision makers with heterogeneous utilities. For repeated choices by a
single decision maker, the components may instead index latent states, such as
intent or mood, under which the same alternatives receive different systematic
utilities. Mixing across components can therefore generate substitution
patterns beyond those of a single MNL model. Under regularity conditions, for any RUM, and any approximation tolerance, \citet{mcfadden2000mixed} construct a mixed MNL model whose choice probabilities approximate those of the RUM uniformly over the domain of choices.

A task of interest in discrete choice modeling is to infer population level choice probabilities for agents facing \emph{counterfactual choice sets}. For example, in the context of an online streaming platform, it may be of interest to perform inference on the probability that a user watches a given title if the platform adds a specified set of titles with known characteristics that are not currently available on the service \citep[see,][]{zielnicki2025value}. If the analyst imposes no restrictions on substitution patterns beyond those implied by rationality, the observed data generally do not point-identify choice probabilities for counterfactual choice sets. One way to understand this is to consider a nonparametric mixed MNL \citep[see][]{mcfadden2000mixed}, in which population heterogeneity is described by an unknown distribution of preferences. In this representation, rationality alone leaves the mixing distribution unrestricted and potentially with an infinite support, while choice probabilities from finitely many observed choice sets impose only a finite number of identifying restrictions on that distribution. Consequently, the observed data distribution is compatible with multiple mixing distributions that satisfy the identifying restrictions but imply different choice probabilities under a counterfactual choice set. An \emph{identified set} contains all counterfactual choice probabilities compatible with the observed data generating process. Counterfactual choice probabilities are therefore set-identified rather than point-identified. In this paper, we conduct inference on the lower and upper bounds of these identified sets of choice probabilities under counterfactual choice sets.

Inference for identification bounds is a well-studied problem in statistics. In particular, \citep{manski1989anatomy,manski2003partial} formalize it as an inferential goal. Using asymptotically normal estimators of the endpoints, \citet{imbens2004confidence} construct Wald-type confidence intervals for the identified set, rather than seeking simultaneous coverage of the entire identified interval. Such problems arise naturally for a variety of targets in a variety of settings. For example, in instrumental-variable models in which the instrument, treatment, and outcome are binary, the average treatment effect is generally an interval rather than point-identified \citep{balke1997bounds}. In our setting, the observed choice probabilities are finitely many linear observations of an unrestricted distribution of preferences, while a counterfactual choice probability is another linear functional of that distribution. Previous works on linear inverse problems give conditions under which a linear functional of an infinite-dimensional parameter is point-identified, even when the underlying infinite-dimensional parameter is not \citep{evans2002inverse,severini2006some,severini2012efficiency}. We do not impose these point identification conditions, and instead conduct inference on the identification bounds of the counterfactual functional. These bounds are values of linear programs over probability distributions. This linear-programming formulation parallels \citet{benmichael2025partial}, who studies finite-dimensional linear programs conditional on covariates. In contrast, we consider optimization problems over an unrestricted probability distribution on a potentially continuous preference space subject to finitely many choice-probability restrictions. It is of interest to note that in some linear inverse problems, the identified set for a linear functional can be unbounded; for example, \citet{freyberger2015identification} show that, in many cases, an unidentified linear functional in a nonparametric instrumental variable model can take any real value. In our setting though, it is in general not the case.

For a fixed finite collection of observed choice sets, the lower and upper identification bounds are functions of the observed choice probabilities and take values in [0,1]. Each bound is the optimal value of a linear program over preference distributions that reproduce these probabilities. We establish conditions under which strong duality holds for the optimization problem. Under strong duality, the upper bound is the infimum and the lower bound is the supremum of affine functions of observed choice probabilities. \citet{hirano2012impossibility} show that functionals defined by minima or maxima of smooth functionals of the data-generating process are non-differentiable \citep{pfanzagl1982contributions} at ties and therefore admit no regular estimator. To tackle this issue, \citet{benmichael2025partial} imposes a margin condition on the difference between the optimal value and the closest nonoptimal basic feasible value and, for a first-order expansion independent of the chosen basis, requires either a unique optimal basis or a nondegenerate primal optimum that yields a unique dual solution. \citet{jordan2026data} assume strong duality and pathwise differentiability of the target parameter, 
and require a unique nondegenerate optimum with a locally stable basis. In this work we isolate a set of sufficient conditions for regularity of the targets. Under our conditions, the optimum of the dual problem is unique and constant in a neighborhood (for the topology from \Cref{definition:neighborhood_mnpn}) of the optimum. Local linearity yields an asymptotically normal plug-in estimator and a consistent variance estimator.

Approximating the preference space by a finite grid produces a finite-dimensional optimization program, but can narrow the identified interval by lowering its upper endpoint or raising its lower endpoint. \citet{benmichael2025partial} computes bounds from finite-dimensional conditional linear programs using solver output rather than enumerating all vertices of each feasible set, whereas \citet{kalouptsidi2026counterfactual} compute counterfactual endpoints by constrained nonlinear optimization, and construct confidence sets by inverting tests based on the distance between estimated moments and the set of model-implied moments. We show that every value in a nonempty identified interval is attained by a preference distribution supported on at most $r+2$ preference values, where $r$ is the number of restrictions imposed by the observed choice probabilities. We define a membership criterion that determines whether a candidate counterfactual choice probability can be generated by a preference distribution that also reproduces the observed choice probabilities, and express it as the minimum KL divergence between two subsets of a finite-dimensional probability simplex. The first set contains probability vectors whose marginals equal the observed choice probabilities and a candidate counterfactual choice probability. The second set contains probability vectors generated by preference distributions with a fixed number of support points. Under realizability and compactness conditions, the identified interval consists exactly of the candidate values for which this minimum equals zero. For each candidate value, we propose an expectation-maximization (EM)-inspired algorithm \citep{dempster1977maximum} that alternates KL projections between these two sets. When a candidate value admits a representation satisfying our local regularity conditions, the resulting KL divergence converges to zero from a set of initializations with positive probability.

\subsection*{Related literature}

Our paper contributes to several literatures in economics and statistics.

\paragraph{Identification in Multinomial Choice Models.} Our paper contributes to a growing literature in industrial organization that studies identification problems in the context of a widely used choice model: the mixed multinomial logit model. The application of this model in economics dates back to \cite{mcfadden1972conditional}, largely since it can capture arbitrary substitution patterns implied by random utility maximization \citep{mcfadden2000mixed}. Rather than seek point identification of the full mixing distribution, we leave it unrestricted and derive sharp identification bounds on its linear functionals under a generic measurable choice kernel.

Within this literature, the closest papers are \cite{fox2011identifying, fox2012random, fox2016nonparametric} who study identification and estimation of the distribution of consumer heterogeneity in multinomial choice models. Their work focuses on settings where there is variation in the choice sets faced by different consumers, and establishes point identification of the heterogeneity distribution under relatively strong assumptions, including that the distribution of heterogeneity can be represented by a finite set of consumer types. Classical results in statistics establish identifiability of finite mixtures for specific families \citep{teicher1963identifiability}. Our framework does not
impose finite support or component-family conditions sufficient for such
identification. Relatedly \citet{tebaldi2023nonparametric} exploit quasilinear utility and finitely many price vectors to reduce the identification problem to a finite-dimensional linear program over finitely many types. More broadly, \cite{berry2014identification, berry2016identification, berry2024nonparametric} provide conditions for nonparametric point identification of differentiated-products demand systems. Building on the point-identification conditions of \citet{berry2014identification}, \citet{compiani2022market} develops a nonparametric estimator of structural demand functions and associated market counterfactuals. Similarly, \cite{briesch2010nonparametric, raval2017semiparametric} propose estimators that impose structure on heterogeneity in situations where the target may not be nonparametrically point identified. In an assortment choice setting, \citet{kallus2016revealed} estimate heterogeneous customer preferences from the assortment offered to each customer and the single alternative selected by imposing a low-rank structure on the utility matrix \citep[see][for a general treatment of low-rank models]{udell2016generalized}. In our context, the conditions for point identification are not satisfied since we observe choices from a finite set of choice sets and allow the heterogeneity distribution to be infinite-dimensional. Finally, \citet{pakes2024moment} sharply characterize the identified set for the covariate index in a two-period panel multinomial choice model using conditional moment inequalities, and \citet{athey2025identification} identify the average treatment effect on the treated in a nonparametric panel model without identifying its latent unit and time components.

\paragraph{Partial identification and inference.} Finitely many observed choice probabilities need not identify an
unrestricted distribution of preferences and, consequently, need not identify
a linear functional of this distribution. Following the partial-identification
framework introduced in \citet{manski1989anatomy}, we conduct inference on the identified set. Subsequent work ranges from bootstrap confidence intervals for individual bound
endpoints \citep{manski1992alternative}, and joint asymptotic inference for
plug-in endpoint estimators \citep{horowitz2000nonparametric} to confidence intervals for interval-identified scalar parameters \citep{imbens2004confidence}, and confidence regions for identified sets
characterized by criterion functions and moment restrictions \citep{chernozhukov2007estimation}. More recently, \citet{kaido2016dual} develops Wald-type inference for compact convex
identified sets by representing each set through maxima of linear
functionals, and \citet{mbakop2023identification} characterizes the identified sets for common parameters in semiparametric panel choice models by generating conditional moment inequalities. \citet{manski2003partial} provides a systematic treatment of
partial identification of probability distributions, and \citet{molinari2020microeconometrics} reviews identification and inference in
partially identified microeconometric models.

\paragraph{Dynamic discrete choice and inverse reinforcement learning.} Our formulation can accommodate likelihoods induced by sequential decision problems. In the canonical bus engine replacement problem, \citet{rust1987optimal} represents replacement decisions
as the optimal policy of a controlled Markov process and estimates the underlying cost and transition parameters as the solution of a Bellman equation. Structural dynamic discrete-choice models are closely connected to maximum-entropy inverse reinforcement learning \citep{ziebart2010modeling}. Under additive payoff shocks that are i.i.d. type I extreme value across actions and time, optimal conditional choice probabilities have the same softmax representation as entropy-regularized policies \citep[see][the paper recovers the reward function by sequentially estimating the policy, Q-function, and reward with deep learning]{geng2020deep}. Recent work by \citet{vanderlaan2025inverse} exploits this connection to recover a normalized reward using classification followed by Bellman regressions.

\paragraph{Measuring Substitution Patterns.} Finally, we contribute to the growing literature on measuring substitution patterns in the ``attention economy``, where the scarce resource is user attention and goods are typically costless to consume \citep{brynjolfsson2019using, calvano2021market, yuan2025competing}.  Existing work relies on generated or exogenous good unavailability variation \citep{conlon2013demand, raval2022using, aridor2025measuring, zielnicki2025value}, second-choice data \citep{conlon2023estimating}, or using survey-based methods to elicit choices under hypothetical scenarios \citep{dertwinkel2024defining, bursztyn2025measuring}
to estimate second-choice diversion ratios. In this paper we show how to conduct inference using this type of variation and optimize it to be maximally informative about specific counterfactuals of interest.

\subsection*{Notation}

Throughout this paper, $\NN$ denotes the set of positive integers, $\RR$ denotes the set of real numbers, and $\RR_+$ denotes the set of nonnegative real numbers. For any $m\in\NN$, we write $[m]= \{1, \ldots, m\}$.

For any set $S$, $2^S$ denotes its power set and $\Card(S)$ denotes its cardinality whenever $S$ is finite. If $S$ is a subset of a finite-dimensional Euclidean space, $\overline{S}$ denotes its closure with respect to the Euclidean norm. For any measurable space $(\Xcal,\mathscr{X})$, $\Pcal(\Xcal, \mathscr{X})$ denotes the set of probability measures on $(\Xcal,\mathscr{X})$. For a distribution $P$, we write $\PP_P$, $\EE_P$, and $\Var_P$ for probability, expectation, and variance under $P$, respectively. We write $\Ncal(0,1)$ for the normal distribution with mean zero and variance one. For any $K \in \NN$, we write $\Delta_K= \{ q \in[0,1]^K \colon \sum_{k=1}^K q_k=1 \}$. We interpret each $q \in \Delta_K$ as a probability mass function on $[K]$, represented by a vector. For any measurable space $(E,\Ecal)$ and any $x\in E$, let $\delta_x \in \Pcal(E,\Ecal)$ denote the Dirac measure at $x$.

For a vector or matrix $u$, $u^\top$ denotes its transpose, and $\| u \|$ denotes its Euclidean norm. For a collection $A$ of vectors or functions, $\Span(A)$ denotes its linear span.

By convention, the supremum of the empty set is $-\infty$ and the infimum of the empty set is $+\infty$. We also use the convention $0\log 0=0$. For probability measures $r$ and $q$ on a common finite set $\Xcal$, $\KL(r\| q)$ denotes the Kullback--Leibler divergence from $r$ to $q$, $\KL(r \| q)=\sum_{x \in \Xcal} r(x)\log\frac{r(x)}{q(x)}$, with the conventions $0\log (0 / q) =0$, and $r \log (r / 0) =+\infty$ for $r > 0$.

\section{Setting}

\paragraph{Design, data domain, and nonparametric model.} Let $J, T\geq 1$ be positive integers, $\Acal_T=\{A_1\times \cdots \times A_T \colon
A_t\subseteq [J], A_t\ne \varnothing\}$. For any positive integer $N$, let $a^N_1,\ldots,a^N_N$ be a sequence of tuples such that for any $i \in [N], a^N_i \in \Acal_T$. Let $\Ycal^N = a^N_1 \times \ldots \times a^N_N$. We refer to $[J]$ as the set of \textit{alternatives}, to any subset thereof as a \textit{choice set}, to $(a^N_1,\ldots,a^N_N)$ as the \textit{design}, to $\Ycal^N$ as the \textit{data domain}, and to $N$ as a sample size. We define
\begin{align}
    \Mcal^{np,N} = \Pcal(\Ycal^N, 2^{\Ycal^N}) .
\end{align}
The model $\Mcal^{np,N}$ contains every probability measure on $\Ycal^N$ and therefore imposes no restrictions on the joint distribution of observed choices. We interpret $P^N \in \Mcal^{np,N}$ as a joint distribution of choice vectors under the fixed design $(a_1^N,\ldots,a_N^N)$.

\paragraph{Preferences, choice kernel, and choice model.} Let $(\Bcal, \mathscr{B})$ be a measurable space, $\Qcal_\beta = \Pcal(\Bcal, \mathscr{B})$. We will refer to $\Bcal$ as the \textit{preference set}. Let $\Zcal_i = \Bcal \times a^N_i, \mathscr{Z}_i = \mathscr{B}\times 2^{a^N_i}$, and $(\Zcal^N, \mathscr{Z}^N) = (\Zcal_1\times \ldots \times \Zcal_N, \mathscr{Z}_1\times \ldots \times \mathscr{Z}_N)$.

\begin{definition}\label{definition:kernel_s}[Choice kernel]
    We say that a map $s \colon \Bcal \times \Acal_T \times [J]^T \to [0,1]$ is a choice kernel if it satisfies the following three conditions.
    \begin{enumerate}
        \item $s$ is $(\mathscr{B} \times 2^{\Acal_T} \times 2^{[J]^T})$-measurable.
        \item For any $(\beta, a, y) \in  \Bcal \times \Acal_T  \times [J]^T$, $s(\beta, a, y) = 0$ if $y \not \in a$.
        \item For any $(\beta, a) \in \Bcal \times \Acal_T, \sum_{y \in a} s(\beta, a, y) = 1$. 
    \end{enumerate}
\end{definition}

For any $(\beta,a) \in \Bcal \times\Acal_T$, the map $s(\beta, a, \cdot)$ is a probability mass function supported on $a$.
We interpret $\beta$ as a preference vector and $s(\beta,a,y)$ as the conditional probability of choice $y$ under choice set tuple $a$ and preference value $\beta$. \Cref{example:low_rank_mnl,example:mixture_transformers} provide two specifications of the choice kernel $s$.

For any $Q_\beta \in \Pcal(\Bcal, \mathscr{B})$, let $P^{F,s,N}(Q_\beta)$ be the probability measure over $(\Zcal^N, \mathscr{Z}^N)$ defined for any $Z^N =((B_1,Y^N_1),\ldots,(B_N, Y^N_N)) \in \mathscr{Z}^N$, as
\begin{align}
    P^{F,s,N}(Q_\beta)(Z^N) =  \prod_{i=1}^N \int_{B_i} dQ_\beta(\beta_i)\sum_{y\in Y^N_i} s(\beta_i, a^N_i, y).
\end{align}
Let $\Mcal^{F,s,N} = \{P^{F,s,N}(Q_\beta) \colon Q_\beta \in \Qcal_\beta\}$. We refer to $\Mcal^{F,s,N}$ as the \textit{complete data choice model} induced by the preference space $(\Bcal, \mathscr{B})$, choice kernel $s$, and design $a^N_1,\ldots,a^N_N$. Under $P^{F,s,N}(Q_\beta)$, the preference values
$\beta_1,\ldots,\beta_N$ are independent draws from $Q_\beta$, and the
choices are conditionally independent given these preference values.
We interpret $P^{F,s,N}(Q_\beta)$ as the joint distribution of latent
preferences and observed choices.

\paragraph{Marginalization operator.} Let $M_{\Ycal^N} \colon \Pcal(\Zcal^N, \mathscr{Z}^N) \to \Mcal^{np,N}$ be such that for any $P^{F,N} \in \Pcal(\Zcal^N, \mathscr{Z}^N), (y^N_1,\ldots, y^N_N)\in \Ycal^N$
\begin{align}
    M_{\Ycal^N}(P^{F,N})(y^N_1,\ldots, y^N_N) = \int_{\beta_1 \in \Bcal} \ldots \int_{\beta_N \in \Bcal} d P^{F,N}(\beta_1,y^N_1,\ldots, \beta_N, y^N_N) .
\end{align}
The marginalization operator maps a probability measure on the complete-data space $\Zcal^N$, containing latent preferences and observed choices, to its marginal distribution on the observed-choice space $\Mcal^{np,N}$.

\paragraph{Target Operator.} Let $g \colon \Bcal \to [0,1], \beta \mapsto g(\beta)$ be a function measurable with respect to $\mathscr{B}$. Let $\tilde{\Psi}$ be the operator defined as
\begin{align}\label{equation:target_operator}
    \tilde{\Psi} \colon \Qcal_\beta & \to \RR , \\
    Q_\beta & \mapsto \int_\Bcal g(\beta) d Q_\beta(\beta) .
\end{align}
The value $\tilde{\Psi}(Q_\beta)$ is the expectation of $g(\beta)$ under $Q_\beta$. When $Q_\beta$ represents a population distribution of preferences, we interpret this expectation as a population average. Note that, if $g(\beta) = s(\beta, a^\star, y^\star)$ for some $a^\star, y^\star \in \Acal_T \times [J]^T, a^\star \notin \{a^N_1, \ldots, a^N_N\}\}$, then $\tilde{\Psi}(Q_\beta)$ is the probability of choice vector $y^\star$ under the counterfactual choice set $a^\star$, averaged over $Q_\beta$. Let $\Psi^N$ be defined as
\begin{align}
    \Psi^N \colon \Mcal^{np,N} & \to 2^{\RR} \\
    P^N & \mapsto \left\{\tilde{\Psi}(Q_\beta) \colon Q_\beta \in \Qcal_\beta,  M_{\Ycal^N}(P^{F,s,N}(Q_\beta))= P^N \right\} .
\end{align}
$\Psi^N(P^N)$ contains all the target values compatible with observed-choice distribution $P^N$ under choice kernel $s$. We define the target functional $\Psi^{+,N}$ as
\begin{align}
    \Psi^{+,N} \colon \Mcal^{np, N} & \to \bar{\RR} \\
    P^N & \mapsto \sup \Psi^N(P^N) ,
\end{align}
and note that $\Psi^{+,N}(P^N)$ is the upper endpoint of the \emph{identification set} $\Psi^N(P^N)$.

The mixed multinomial logit model represents choice probabilities as averages of conditional logit probabilities over a latent distribution of preferences \citep{mcfadden2000mixed}. We study this latent-mixture structure, but replace the multinomial logit kernel with any choice kernel $s$ satisfying the conditions in \Cref{definition:kernel_s}, and shift the object of inference. Rather than seeking point identification of the full mixing law as the objective, we leave the latent preference distribution $Q_\beta$ unrestricted and study identification bounds on any linear functional as defined in \eqref{equation:target_operator}. The finitely many observed choice probabilities impose moment restrictions on $Q_\beta$, but need not identify either the full distribution or a counterfactual functional of it.

\begin{figure}[htbp]
  \centering
\includegraphics[width=1.0\textwidth]{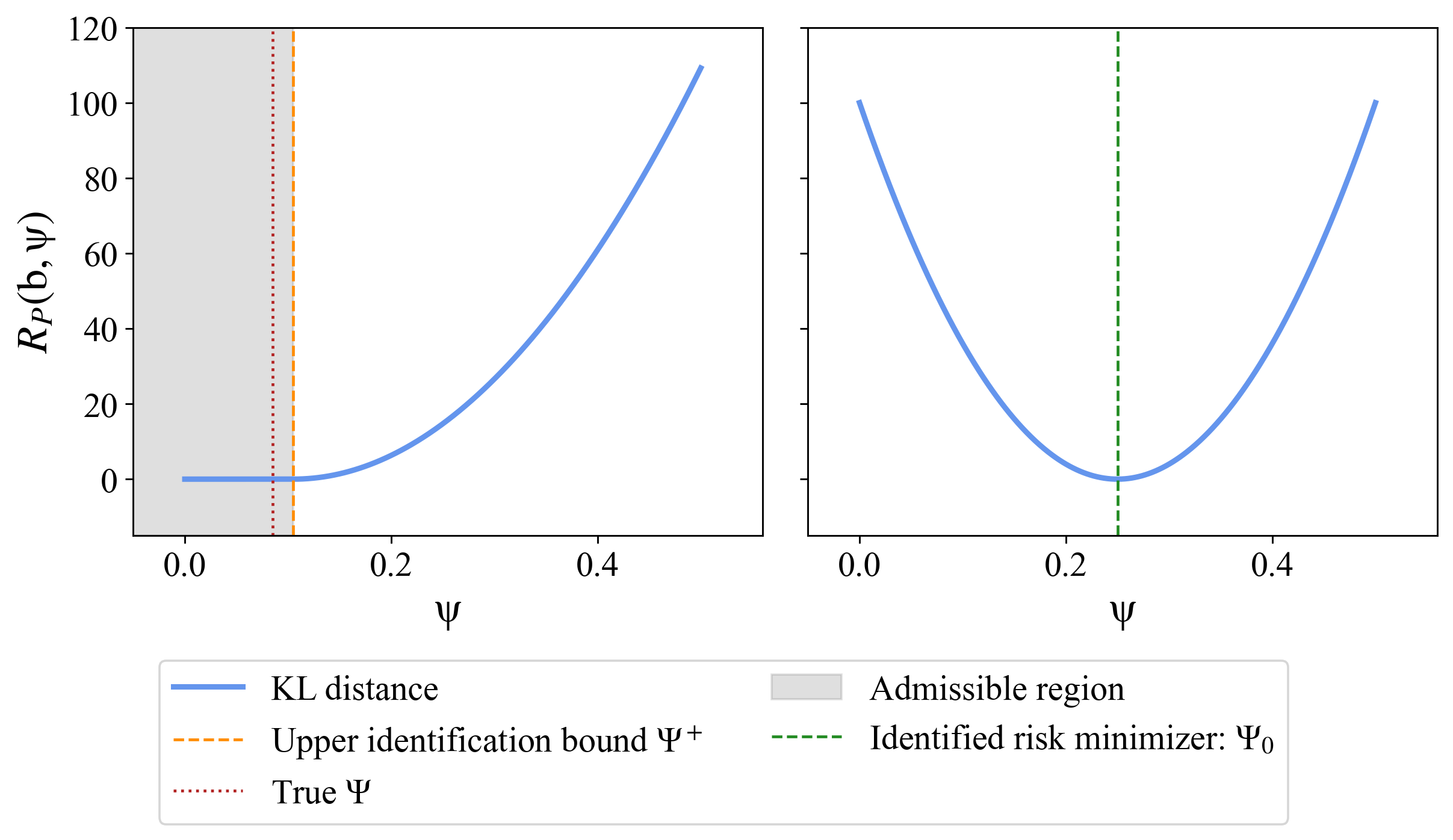}
  \caption{Upper identification bound for a partially identified target. Right: identified functional.}
  \label{figure:illustration1}
\end{figure}

\begin{example}\label{example:low_rank_mnl}[Low-rank multinomial logit]
Let $I, J, T, r \in \NN$ satisfy
$r \leq \min\{I,J\}$. Define $\Mcal_J = \{\tilde{a} \subseteq [J] \colon \tilde{a} \neq \varnothing \}$ and $\Acal_T = \Mcal_J^T$. An element $\tilde{a} = (\tilde{a}_1, \ldots, \tilde{a}_T) \in \Acal_T$ is a $T$-tuple of nonempty subsets of $[J]$. Let $U \in \RR^{I \times J}$ satisfy
$\text{rank} (U) \leq r$, and define the mapping
\begin{align}
s \colon[I] \times \Acal_T \times[J]^T \to [0,1] , \; (i,a,y) \mapsto
\begin{cases}
\displaystyle \prod_{t=1}^T \frac{\exp(U_{i y_t})}
{\sum_{j\in \tilde{a}_t}\exp(U_{ij})}
\quad \text{ if } y \in a , \\
0 \quad \text{otherwise} .
\end{cases}
\end{align}
For any $(i,a) \in [I] \times \Acal_T$, note that
\begin{align}
\sum_{y\in[J]^T} s(i,a,y) = \sum_{y_1 \in \tilde{a}_1} \cdots \sum_{y_T \in \tilde{a}_T} \prod_{t=1}^T \frac{\exp(U_{i y_t})}{\sum_{j\in \tilde{a}_t}\exp(U_{ij})} = \prod_{t=1}^T \sum_{y_t\in \tilde{a}_t} \frac{\exp(U_{i y_t})}{\sum_{j\in \tilde{a}_t} \exp(U_{ij})} = 1 .
\end{align}
Thus, $s(i,a,\cdot)$ is a probability mass function on $[J]^T$ for any $(i,a) \in [I] \times \Acal_T$. Since $\text{rank}(U) \leq r$, there exist $X \in \RR^{I\times r}$, and $Z \in \RR^{J\times r}$ such that $U = X Z^\top$. Let $x_i^\top$ denote row $i$ of $X$, and let $z_j^\top$ denote row $j$ of $Z$. For any $(i,a,y) \in [I] \times \Acal_T\times[J]^T$ such that $y\in a$,
\begin{align}
s(i, a, y) = \prod_{t=1}^T \frac{\exp(x_i^\top z_{y_t})}{\sum_{j\in \tilde{a}_t} \exp(x_i^\top z_j)} .
\end{align}
Let $q=(q_1,\ldots,q_I)\in[0,1]^I$ satisfy
$\sum_{i=1}^I q_i=1$. Define
\begin{align}
p_q \colon\mathcal{A}_T\times[J]^T\to[0,1] , \; (a,y) \mapsto \sum_{i=1}^I q_i s(i,a,y) .
\end{align}
Consequently, for a sequence $(A_t)_{t \in [T]}$ such that for any $t \in [T], A_t \subseteq [J]$, we have that
\begin{align}
p_q(a,y) = \begin{cases}
\displaystyle
\sum_{i=1}^I q_i \prod_{t=1}^T \frac{\exp(x_i^\top z_{y_t})}
{\sum_{j\in A_t}\exp(x_i^\top z_j)},
\; \text{ if } y \in a, \\ 
0 \; \text{ otherwise} .
\end{cases}
\end{align}
We interpret $i$ as a consumer preference, $j$ as an alternative, $A_t$ as a choice set at time $t$, and $U_{ij}$ as systematic utility for preference $i$ and alternative $j$. Conditional on type $i$, choices at distinct times $t \in [T]$ are independent. The vector $q$ gives the probabilities of preference types. Hence, $p_q(a,y)$ is the probability of choice sequence $y$ under choice sets sequence $a$. We refer to
$\text{rank}(U)\leq r$ as the low-rank restriction.
\end{example}

\begin{example}\label{example:mixture_transformers}[Mixture of transformer choice models, see \citet{zielnicki2025value}]
Let $\Ycal = \{0, \ldots, J\}$, where $0$ is the outside option. A choice set is
\begin{align}
a = (A_1, \ldots, A_T, (C_t)_{t \in [T]}),
\end{align}
where $A_t \subseteq \Ycal$, $0 \in A_t$ (independently of the choice set, the outside option is always available), and $C_t \subseteq A_t \setminus \{0\}$ for any $t \in [T]$. We interpret $C_t$ as the set of alternatives recommended at decision point $t$. A preference $\beta \in \Bcal$ consists of alternative representations $B_{\beta,j} \in \RR^d$, recommendation effects $\gamma_{\beta,j} \in \RR$, and mappings
\begin{align}
F_{\beta,t} \colon \Ycal^{t-1} \to \RR^d, \qquad t \in [T].
\end{align}
The vector $F_{\beta,t}(y_1, \ldots, y_{t-1})$ represents the decision maker's preference vector at decision point $t$ as a function of the previous choices. It is modeled by a transformer that is masked so that it does not use future choices. For a feasible choice sequence $y \in A_1 \times \cdots \times A_T$, define for any $j \in A_t \setminus \{0\}$
\begin{align}
u_{\beta,t}(j \mid y,a) = F_{\beta,t}(y_1, \ldots, y_{t-1})^\top B_{\beta,j} + \gamma_{\beta,j}\ones_{\{j \in C_t\}} ,
\end{align}
and $u_{\beta,t}(0 \mid y,a) = 0$ can be used as a normalization of the logits. In that case, the choice kernel is
\begin{align}
s(\beta,a,y) = \begin{cases}
\displaystyle \prod_{t=1}^T \frac{\exp\{u_{\beta,t}(y_t \mid y,a)\}}{\sum_{j \in A_t}\exp\{u_{\beta,t}(j \mid y,a)\}} , \text{ if } y \in A_1 \times \cdots \times A_T , \\
0 , \text{ otherwise}.
\end{cases}
\end{align}
Given a probability measure $Q_\beta$ on $\Bcal$, the population probability of $y$ under $a$ is
\begin{align}
p_{Q_\beta}(a,y) = \int_{\Bcal}s(\beta,a,y) dQ_\beta(\beta) .
\end{align}
In the setting of \citet{zielnicki2025value}, the decision maker is a Netflix user, an alternative is a title, and each decision point corresponds to one day. The outside option means that the user consumes no title that day. The preference $\beta$ is drawn once for a user and remains fixed across the $T$ decision points, while $Q_\beta$ describes preference heterogeneity across users.
\end{example}

\section{Von Mises expansion of the target functional}\label{section:vonmises}

\paragraph{Feasible choices and marginalization operator.} Let $\pairay^N = \{(a,y) \colon a \in \{a^N_1,\ldots,a^N_N\},\, y \in a\}$ denote the set of feasible choice set tuples. Otherwise stated, the elements of $\pairay^N$ are pairs $(a,y)$, where $a$ is a choice set tuple appearing in the design and $y$ is a feasible choice vector under $a$.

For any $a \subseteq [J]^T$, let $\couNA = \{i \in [N] \colon a^N_i =a\}$, $n_a = \Card(\couNA)$, and
\begin{align}
    \Qcal = \left\{Q \colon \pairay^N \to [0,1] \colon Q(a,\cdot) \in \Pcal(a,2^a) \text{ for any } a \in \{a^N_1,..,a^N_N\} \right\} .
\end{align}
Let $Q^{np,N}_i \colon \Mcal^{np, N} \to \Pcal(a^N_i, 2^{a^N_i})$ which maps any $P^N \in \Mcal^{np,N}$ to $Q^{np,N}_i(P^N)$, defined as
\begin{align}
Q^{np,N}_i(P^N) \colon a^N_i & \to [0,1] \\
y & \mapsto \sum_{y^N_1 \in a^N_1} .. \sum_{y^N_{i-1} \in a^N_{i-1}} \sum_{y^N_{i+1} \in a^N_{i+1} } \ldots \sum_{y^N_N \in a^N_N} P^N(y^N_1, \ldots , y^N_{i-1}, y, y^N_{i+1}, \ldots , y^N_N) .
\end{align}
In words, $Q_i^{np,N}(P^N)$ is the $i$-th marginal
of $P^N$. Let $Q^{np,N}_{Y \mid a} \colon \Mcal^{np, N} \to \Qcal$ which to any $P^N \in \Mcal^{np,N}$ maps $Q^{np,N}_{Y \mid a}(P^N)$ defined as
\begin{align}
Q^{np,N}_{Y \mid a}(P^N) \colon \pairay^N & \to [0,1] \\
(a,y) & \mapsto
\frac{1}{n_a} \sum_{i \in \couNA} Q^{np,N}_i(P^N)(y) .
\end{align}
Note that $Q^{np,N}_{Y\mid a}(P^N)(a,\cdot)$ averages the marginal probability measures $Q_i^{np,N}(P^N)$ over indices $i$ satisfying $a_i^N=a$. We interpret this average as a distribution of choices in choice set $a$, under $P^N$.

\begin{definition}\label{definition:correct_fit}
For any $P^N \in \Mcal^{np,N}$, we say that $P^N$ is realizable under $(\Qcal_\beta, s)$ if there exists $Q_\beta \in \Qcal_\beta$ such that $P^N = M_{\Ycal^N}(P^{F,s,N}(Q_\beta))$.
\end{definition}

If $P^N$ is realizable, and $P^N = M_{\Ycal^N}(P^{F,s,N}(Q_\beta))$, we interpret $Q_\beta$ as a probability distribution of preference values whose implied distribution of observed choices equals $P^N$. In particular, it implies that $\Psi^N(P^N)$ is nonempty, and it imposes both common preference distribution $Q_\beta$ across indices and independence of observed choices across indices under the induced marginal probability measure.

When $P^N$ is realizable, any $Q_\beta \in \Qcal_\beta$ satisfying $P^N = M_{\Ycal^N}(P^{F,s,N}(Q_\beta))$ may be interpreted as a distribution of preferences that generates $P^N$ under the choice kernel $s$ and the design $\{a^N_1, \ldots, a^N_N\}$. In particular, for any $(y_1^N, \ldots, y_N^N) \in \Ycal^N$, we have that
\begin{align}
P^N(y_1^N, \ldots, y_N^N) = \prod_{i=1}^N
\int_{\Bcal} s (\beta, a^N_i, y^N_i) dQ_\beta (\beta) ,
\end{align}
which shows that observed choices are independent across any $i \in [N]$ under $P^N$. These marginal distributions may differ because the choice set tuples $a_i^N$ may differ across indices. Moreover, every such $Q_\beta$ satisfies $\tilde{\Psi}(Q_\beta)\in\Psi^N(P^N)$, so realizability implies that $\Psi^N(P^N)$ is nonempty.

\paragraph{Moments.} Let $L^N = \Card(\pairay^N)$ be the number of distinct choice tuples and compatible choices, and let $\{(a(l),y(l))\}_{l \in [L^N]}$ be an arbitrary ordering of $\{(a,y) \colon a \in \{a^N_1,\ldots, a^N_N\}, y \in a\}$. For any $l \in [L^N]$, let
\begin{align}\label{equation:rjerlkj}
    t^N_l \colon \Bcal & \to \RR \\
    \beta & \mapsto s(\beta, a(l), y(l)) .
\end{align}
By definition of $s, t^N_l$ is measurable with respect to $\mathscr{B}$ for any $l \in [L^N]$. We define the mapping
\begin{align}
    \Bar{b}^N \colon \Mcal^{np,N} & \to \RR^{L^N} \\
    P^N & \mapsto (Q^{np,N}_{Y\mid a}(P^N)(a(1),y(1)),\ldots, Q^{np,N}_{Y\mid a}(P^N)(a(L^N),y(L^N)))^\top .
\end{align}
Let $\tilde{L}^N+1$ be the maximal cardinality of a linearly independent subset of $\{1, t^N_1,\dots, t^N_{L^N}\}$. Without loss of generality, we suppose that $1, t^N_1, \dots, t^N_{\tilde{L}^N}$ is such a subset. Since $t^N_{\tilde{L}^N+1}, \ldots$ $,t^N_{L^N}$ lie in the linear span of $1, t^N_1,\ldots,t^N_{\tilde{L}^N}$, the corresponding moment constraints are redundant for every $P^N \in \Mcal^{np,N}$ realizable under $(\Qcal_\beta, s)$; throughout the rest of the paper, we write $\Bar{t}^N(\beta)=(t^N_1(\beta),\ldots,t^N_{L^N}(\beta))^\top$, $t^N(\beta)=(t^N_1(\beta),\ldots,t^N_{\tilde{L}^N}(\beta))^\top$, and $b^N(P^N)=(\Bar{b}^N_1(P^N),\ldots, \Bar{b}^N_{\tilde{L}^N}(P^N))^\top$.

For any $P^N \in \Mcal^{np,N}$, let $\textsc{Primal}(P^N)$ be the linear program
\begin{align}
\sup_{Q_\beta \in \Qcal_\beta} \int g(\beta)d Q_\beta(\beta) \text{ subject to } \int t^N_l(\beta)d Q_\beta (\beta)=b^N_l(P^N) \text{ for any } l =1,\ldots,\tilde{L}^N ,
\end{align}
and define $\tilde{\Psi}^{+,N}(P^N)$ as the value of the linear program $\textsc{Primal}(P^N)$. As in the general conditional linear program framework of \citet{benmichael2025partial}, the identification bound is characterized as the value of a linear optimization problem. Here, $\textsc{Primal}(P^N)$ is a supremum over probability measures on the preference space subject to the choice probability moment restrictions.

\begin{restatable}{theorem}{theoremrepresentation}\label{theorem:representation1}
Let $P^N \in \Mcal^{np,N}$ such that $P^N$ is realizable under $(\Qcal_\beta, s)$. Then, we have
\begin{align}
     \Psi^{+,N}(P^N) = \tilde{\Psi}^{+,N}(P^N) .
\end{align}
\end{restatable}

\begin{remark}
    \Cref{theorem:representation1} replaces the equality constraint on the observed-choice probability measures with $\tilde{L}^N$ scalar moment equalities. Hence, for any realizable $P^N, \Psi^{+,N}(P^N)$ depends on $P^N$ solely through $b^N(P^N)$.
\end{remark}

\begin{definition}
[Nondegeneracy]\label{definition:nondegeneracy} For any $P^N \in \Mcal^{np,N}$, we say that $P^N$ is nondegenerate if $\textsc{Primal}(P^N)$ admits a maximizer of the form 
$Q = \sum_{l=1}^{\tilde{L}^N+1}Q_l \delta_{\beta_l}$ with $Q_l > 0$ for all $l = 1,\dots, \tilde{L}^N+1$, $\sum_{l=1}^{\tilde{L}^N +1 } Q_l = 1$, $\beta_1,\dots, \beta_{\tilde{L}^N+1} \in \Bcal$ such that
\begin{align}
    M(Q) = \begin{bmatrix}
        \tilde M(Q) \\
        \bm{1}_{1\times (\tilde{L}^N+1)}
    \end{bmatrix} \in \RR^{(\tilde{L}^N + 1) \times (\tilde{L}^N + 1)} ,
\end{align}
is full rank, where $\tilde{M}(Q) \in \RR^{\tilde{L}^N \times (\tilde{L}^N + 1)}$ collects at every row $l=1,\ldots,\tilde{L}^N$ the constraint function $t^N_l$ evaluated at the atoms $\beta_k$, $k=1,\ldots,\tilde{L}^N+1$, that is
\begin{align}
\tilde{M}(Q) = [t^N_l(\beta_k)]_{l=1,\ldots, \tilde{L}^N, k=1,\ldots \tilde{L}^N+1} .
\end{align}
\end{definition}

\begin{remark}
    We note that positivity of $Q_1,\ldots,Q_{\tilde L^N+1}$ and invertibility of $M(Q)$ imply that, for every $b$ in a neighborhood of $b^N(P^N)$, there exists a positive weight vector $w(b)\in\Delta_{\tilde L^N+1}$ such that $b = \sum_{k=1}^{\tilde{L}^N+1} w_k(b) t^N(\beta_k)$. Thus, nearby moment vectors can be represented by changing the weights while keeping the atom locations fixed. The condition in \Cref{definition:nondegeneracy} is similar to the finite-dimensional nondegeneracy condition in \citet{benmichael2025partial}, under which an optimal solution is associated with an invertible constraint submatrix and has strictly positive coordinates.
\end{remark}

\begin{assumption}\label{assumption:g_not_in_span}
$g \notin \mathrm{Span}\{1, t^N_1,\dots, t^N_{\tilde{L}^N}\}$.
\end{assumption}

\begin{remark}
We note that, if \Cref{assumption:g_not_in_span} fails, then there exist
$c_0,c_1,\ldots,c_{\tilde L^N}\in\RR$ such that for any $\beta \in \Bcal$, we have $g(\beta) = c_0 + \sum_{l=1}^{\tilde{L}^N}c_l t^N_l (\beta)$. Consequently, any $Q_\beta$ feasible for $\textsc{Primal}(P^N)$ satisfies
\begin{align}
\tilde{\Psi}(Q_\beta) = \int_\Bcal g(\beta) dQ_\beta(\beta) = \int_\Bcal \{c_0 + \sum_{l=1}^{\tilde{L}^N}c_l t^N_l (\beta)\} dQ_\beta(\beta) = c_0 + \sum_{l=1}^{\tilde L^N} c_l b_l^N(P^N).
\end{align}
Therefore, whenever $P^N$ is realizable, the identification set is the singleton
\begin{align}
\Psi^N(P^N) = \left\{c_0+\sum_{l=1}^{\tilde L^N}c_l b_l^N(P^N) \right\}.
\end{align}
Thus, \Cref{assumption:g_not_in_span} excludes the case in which the
moment restrictions point identify the target, even though they may not
identify $Q_\beta$.
\end{remark}

For any $P^N \in \Mcal^{np,N}$, let $\textsc{Dual}(P^N)$ be the linear program
\begin{align}
\inf_{\lambda\in \RR^{\tilde{L}^N}, \nu\in \RR} \lambda^\top b^N(P^N)+\nu \text{ subject to } \lambda^\top t^N(\beta)+\nu \geq g(\beta) \text{ for any } \beta \in \Bcal .
\end{align}

\begin{restatable}{proposition}{propositionstrongduality}\label{proposition:strong_duality}
Suppose that \Cref{assumption:g_not_in_span} holds and that $P^N$ is nondegenerate in the sense of \Cref{definition:nondegeneracy}. Then $\textsc{Primal}(P^N)$ and $\textsc{Dual}(P^N)$ have the same value, and $\textsc{Dual}(P^N)$ admits a unique minimizer.
\end{restatable}

\begin{definition}\label{definition:neighborhood_mnpn}[Neighborhood in $\Mcal^{np,N}$.]
We say that $\Vcal \subset \Mcal^{np,N}$ is a neighborhood of a given $P^N \in \Mcal^{np,N}$ if there exists a neighborhood $v$ of $b^N(P^N) \in \RR^{\tilde{L}^N}$ such that for any $\Bar{P}^{N} \in \Mcal^{np,N}$ such that $b^N(\Bar{P}^{N}) \in v$, we have $\Bar{P}^{N} \in \Vcal$. 
\end{definition}

We note that $\textsc{Primal}(P^N)$ and $\textsc{Dual}(P^N)$ depend on $P^N$ through $b^N(P^N)$. This is why we define neighborhoods of $P^N \in \Mcal^{np,N}$ through neighborhoods of $b^N(P^N) \in \RR^{\tilde{L}^N}$.

\begin{restatable}{theorem}{theorempiecewiselinearity}\label{theorem:piecewise_linearity}
Suppose that \Cref{assumption:g_not_in_span} holds, and that $P^N$ is nondegenerate in the sense of \Cref{definition:nondegeneracy}. Then there exists a neighborhood $\Vcal$ of $P^N$ in $\Mcal^{np,N}$ such that for any $P \in \Vcal$,
\begin{align}
\tilde{\Psi}^{+,N}(P) - \tilde{\Psi}^{+,N}(P^N) = \lambda(P)^\top\{b^N(P)-b^N(P^N)\} ,
\end{align}
where $(\lambda(P),\nu(P))$ is the unique minimizer of $\textsc{Dual}(P)$. Moreover, we have that $\lambda(P) = \lambda(P^N)$.
\end{restatable}

\begin{remark}
We note that the upper identification bound $\tilde{\Psi}^{+,N}$ is affine in the observed moment vector throughout a neighborhood around $P^N$, with constant slope $\lambda(P^N)$. Each coordinate of $\lambda(P^N)$ measures the sensitivity of the upper bound to the corresponding moment.
\end{remark}

\section{Inference}\label{section:semiparametric_inference}

Let $(P^N)_{N \geq 1}$ be a sequence of probability distributions such that for any positive integer $N, P^N \in \Mcal^{np,N}$. For any $N \geq 1$, let $(Y^N_1, \ldots, Y^N_N) \sim P^N, P_N = \delta_{(Y^N_1, \ldots, Y^N_N)}$. In this section, we study $\tilde{\Psi}^{+,N}(P_N)$ as an estimator of $\Psi^{+,N}(P^N)$. Under \Cref{assumption:g_not_in_span,assumption:stabilizing_design,assumption:Pconverges,assumption:stability_dual_N_grows,assumption:nondegenerate_limiting_variance}, realizability of $P^N$ under $(\Qcal_\beta, s)$ and nondegeneracy of $P^N$, we prove that a rescaled estimation error converges in distribution to a standard normal random variable. For any $a \in \{a^N_1,
\ldots, a^N_N\}$, define
\begin{align}\label{equation:definitionpropensity}
    e^N(a) = n_a / N .
\end{align}
Note that the quantity $e^N(a)$ is the fraction of indices $i \in [N]$ for which the design $a^N_i$ is equal to $a$.

\begin{assumption}\label{assumption:stabilizing_design} There exist a positive integer $N_0$, and $a_1,\ldots,a_M$, such that, for any $N \geq N_0$, $\{a_1^N,\ldots,a_N^N\}=\{a_1,\ldots,a_M\}$. Moreover, for any $a\in\{a_1,\ldots,a_M\}$, there exists $e(a) >0$ such that
\begin{align}
    e^N(a) \to e(a) \text{ as } N \to \infty.
\end{align}
\end{assumption}
We note that $e^N(a) = n_a / N$ converges to a fixed, and strictly positive limit under \Cref{assumption:stabilizing_design}. Consequently, each design frequency stabilizes and $n_a$ grows proportionally to $N$ for $N$ large enough.

Throughout this section, whenever \Cref{assumption:stabilizing_design}
holds, we fix an ordering
$\{(a(l),y(l))\}_{l\in[L]}$ of the common set
$\{(a,y) \colon a\in \{a_1^N,\ldots,a_N^N\},y\in a\} = \{(a,y) \colon a\in\{a_1,\ldots,a_M\},y\in a\}$ for any sufficiently large $N \geq N_0$, where $N_0$ is as in \Cref{assumption:stabilizing_design}. With this ordering, $L^N=L$ and $t_l^N=t_l$ for any $N \geq N_0$. We also choose the nonredundant set
$\{1,t_1,\ldots,t_{\tilde{L}}\}$ consistently, so that $\tilde{L}^N=\tilde{L}$ for any $N \geq N_0$.

\begin{assumption}\label{assumption:Pconverges}
There exist $N_1 \geq N_0$,
$b\in\RR^{\tilde{L}}$, and $\lambda\in\RR^{\tilde{L}}$ such that, for any $N \geq N_1$ such that $\textsc{Dual}(P^N)$ admits a unique minimizer $(\lambda(P^N), \nu(P^N))$, we have
\begin{align}\label{equation:fjsefrewijr}
    b^N(P^N) = b,\quad \lambda(P^N) = \lambda .
\end{align}
\end{assumption}

\begin{assumption}\label{assumption:stability_dual_N_grows}
There exists a positive integer $N_2$, $\eta>0$ such that for any $N \geq N_2$, if $\textsc{Dual}(P^N)$ admits a unique minimizer $(\lambda(P^N), \nu(P^N))$, then for any $\Bar{P}^N \in \Mcal^{np,N}$ satisfying $\|b^N(\Bar{P}^N)-b^N(P^N)\|<\eta$, we have that  $\textsc{Dual}(\Bar{P}^N)$ admits a unique minimizer $(\lambda(\Bar{P}^N),\nu(\Bar{P}^N))$, the von Mises expansion from \Cref{theorem:piecewise_linearity} holds at $\Bar{P}^N$, and
\begin{align}
    \lambda(\Bar{P}^N)=\lambda(P^N) .
\end{align}
\end{assumption}

\begin{assumption}\label{assumption:nondegenerate_limiting_variance}
    We have
\begin{align}
    \sum_{a\in\{a_1,\ldots,a_M\}}
    \frac{1}{e(a)} \left[\sum_{l \colon a(l)=a} b_l\lambda_l^2-\left\{\sum_{l \colon a(l)=a} b_l\lambda_l\right\}^2 \right] >0 ,
\end{align}
where $\lambda, b$ are defined in \eqref{equation:fjsefrewijr}.
\end{assumption}

For any $\bar{P}^N \in\Mcal^{np,N}$ such that $\textsc{Dual}(\bar{P}^N)$ admits a unique
minimizer, let $(\lambda(\bar{P}^N),\nu(\bar{P}^N))$ denote this minimizer. Define
\begin{align}\label{equation:definitionvarianceXNi}
    \sigma_N^2(\bar{P}^N) = \sum_{a\in\{a_1^N,\ldots,a_N^N\}}
    \frac{1}{e^N(a)}
    \left[\sum_{l\colon a(l)=a}b_l^N(\bar{P}^N)\lambda_l(\bar{P}^N)^2 - \left\{ \sum_{l \colon a(l)=a} b_l^N(\bar{P}^N) \lambda_l(\bar{P}^N)\right\}^2\right] ,
\end{align}
and $\sigma_N(\bar{P}^N) = \sqrt{\sigma_N^2(\bar{P}^N)}$.

We note that $\lambda_l(\bar{P}^N)$ is the coefficient on $b_l^N(\bar{P}^N)$ in the local linear representation of $\tilde{\Psi}^{+,N}$. For any $a \in \{a^N_1, \ldots, a^N_N\}$, the expression in square brackets in \eqref{equation:definitionvarianceXNi} is the variance of $\sum_{l\in[\tilde L^N]\colon a(l)=a}\lambda_l(\bar P^N)\1\{Y=y(l)\}$ when $Y\sim Q^{np,N}_{Y\mid a}(\bar{P}^N)(a,\cdot)$. The factor $1/e^N(a)=N/n_a$ accounts for the number $n_a$ of indices assigned choice-set tuple $a$. Thus, $\sigma_N^2(\bar{P}^N)$ is the variance of the local linear representation. When $\bar{P}^N = P_N$, then $\sigma_N^2(P_N)$ is the plug-in estimator of $\sigma_N^2(\bar{P}^N)$.

\begin{restatable}{theorem}{theoremasymptoticnormality}\label{theorem:asymptotic_normality}
Suppose that \Cref{assumption:g_not_in_span,assumption:stabilizing_design,assumption:Pconverges,assumption:stability_dual_N_grows,assumption:nondegenerate_limiting_variance} hold. Suppose that there exists a positive integer $N_3$, such that for any $N \geq N_3$, $P^N$ is realizable under $(\Qcal_\beta,s)$ in the sense of \Cref{definition:correct_fit}, and is nondegenerate in the sense of \Cref{definition:nondegeneracy}. Then
\begin{align}
    \sqrt{N} \frac{\tilde{\Psi}^{+,N}(P_N)-\Psi^{+,N}(P^N)}{\sigma_N(P_N) } \xrightarrow{d} \Ncal(0,1) \text{ as } N \to \infty .
\end{align}
\end{restatable}

We note that \Cref{theorem:asymptotic_normality} establishes asymptotic normality of the plug-in estimator of the upper identification bound, which is achieved through a von Mises expansion of $\tilde{\Psi}^{+,N}$, from \Cref{theorem:piecewise_linearity}.

\begin{assumption}\label{assumption:fixed_dgp}
There exists $Q_\beta \in \Qcal_\beta$, a positive integer $N_4$ such that, for any $N \geq N_4$, we have
\begin{align}
    P^N = M_{\Ycal^N}\left(P^{F,s,N}(Q_\beta)\right) .
\end{align}
\end{assumption}

\begin{restatable}{lemma}{lemmapopulationmomentsdualfixed}\label{lemma:population_moments_dual_fixed}
Suppose that there exists a positive integer $N_5$ such that for any $N \geq N_5$, $P^N$ is nondegenerate in the sense of \Cref{definition:nondegeneracy}, and that \Cref{assumption:g_not_in_span,assumption:stabilizing_design,assumption:fixed_dgp} hold. Then, \Cref{assumption:Pconverges,assumption:stability_dual_N_grows} hold.
\end{restatable}

We note that \Cref{assumption:fixed_dgp} requires the existence of a latent distribution of preferences $Q_\beta$ to generate the observed choice distribution $P^N$, for
any sufficiently large $N$. This condition ensures that $P^N$ is realizable for any sufficiently large $N$. Together with \Cref{assumption:g_not_in_span,assumption:stabilizing_design}
and nondegeneracy of $P^N$, it also implies
\Cref{assumption:Pconverges,assumption:stability_dual_N_grows}, and thus, it is enough to obtain \Cref{theorem:asymptotic_normality}.

\section{Representation of the target set as an MLE}

In this section, for any realizable $P^N \in \Mcal^{np,N}$ satisfying \Cref{assumption:compact_continuity_certify}, we express the target set $\Psi^N(P^N)$ as a nonparametric maximum likelihood estimand (NPMLE) over preference distributions $\Qcal_\beta$. Nonparametric maximum likelihood estimation over an unrestricted mixing
distribution originates in the classical work of \citet{kiefer1956consistency}, who establish consistency under identifiability
and regularity conditions. \citet{laird1978nonparametric} characterizes the NPMLE through a self-consistency property and shows that, under suitable
conditions, it admits a representation supported by a finite mixture.

Here, instead, the mixing distribution need not be point-identified. We consider the negative log-likelihood induced by each $Q_\beta \in \Qcal_\beta$, profile it over values of $\tilde{\Psi}(Q_\beta)$, and show that, when $P^N$ is realizable under $(\Qcal_\beta,s)$, and \Cref{assumption:compact_continuity_certify} holds, the resulting profile maximum likelihood estimand is exactly $\Psi^N(P^N)$. We then give an equivalent finite-support representation of $\Psi^{+,N}(P^N)$, as a supremum over discrete preference distributions $\sum_{k=1}^K q_k\delta_{\beta_k} \in \Qcal_\beta, K \leq \tilde{L}^N +2, q \in \Delta_K$. This representation is close to the convex geometry of mixture likelihoods in \citet{lindsay1983geometry}. Whereas the latter derives finite-support results for NPMLEs, we use this geometry to characterize the identified set as a profile-NPMLE set with bounded finite-support representations.

\paragraph{Maximum Likelihood.} For any $Q_\beta \in \Qcal_\beta, (y^N_1,\ldots, y^N_N) \in \Ycal^N$, let
\begin{align}
    \loglike(P^{F,s,N}(Q_\beta))(y^N_1,\ldots, y^N_N) &= - \log M_{\Ycal^N}(P^{F,s,N}(Q_\beta))(y^N_1,\ldots, y^N_N) .
\end{align}
We refer to $\loglike(P^{F,s,N}(Q_\beta))(y^N_1,\ldots, y^N_N)$ as the negative log-likelihood loss evaluated at $(y^N_1,\ldots,y^N_N)$ under preference distribution $Q_\beta$.

Following empirical process notation, for any $f \colon \Ycal^N \to \RR$ measurable with respect to $2^{\Ycal^N}$ and $P^N \in \Mcal^{np, N}$, let
\begin{align}
    P^N f = \sum_{(y^N_1, \ldots, y^N_N) \in \Ycal^N}f(y^N_1, \ldots, y^N_N) P^N(y^N_1, \ldots, y^N_N) .
\end{align}
For any $P^N \in \Mcal^{np, N}$ and $Q_\beta \in \mathcal{Q}_\beta$, let
\begin{align}
    \risk(P^N, P^{F,s,N}(Q_\beta)) & = P^N \loglike (P^{F,s,N}(Q_\beta)) .
\end{align}
We refer to $\risk(P^N, P^{F,s,N}(Q_\beta))$ as the risk induced by the negative log-likelihood loss under $P^N$. For any $\psi \in [0,1]$, let
\begin{align}
    \profileloglike(P^N, \psi) = \inf_{Q_\beta \in \Qcal_\beta \colon \Tilde{\Psi}(Q_\beta) = \psi} \risk(P^N, P^{F,s,N}(Q_\beta)) .
\end{align}
We refer to $\profileloglike(P^N, \psi)$ as the profile negative log-likelihood risk of $\psi$ under $P^N$ \citep[see][]{murphy2000profile}. Note that $\arginf_\psi \profileloglike(P^N, \psi)$ may be set-valued. For any $P^N \in \Mcal^{np, N}$, let
\begin{align}
    \Psi^{\mle, N}(P^N) = \arginf_{\psi \in [0,1]} \profileloglike(P^N, \psi) .
\end{align}
We refer to $\Psi^{\mle,N}(P^N)$ as the profile maximum likelihood estimand of $\tilde{\Psi}$ under $P^N$. $\Psi^{\mle,N}(P^N)$ is the set of values $\psi\in[0,1]$ for which the profile negative log-likelihood is minimal under $P^N$.

\begin{assumption}[Compactness and continuity]
\label{assumption:compact_continuity_certify}
The set $\Bcal$ is a compact metric space, $\mathscr{B}$ is the Borel $\sigma$-field induced by its metric topology, and the maps
$\beta\mapsto g(\beta)$ and
$\beta\mapsto \bar t_l^N(\beta)$, $l=1,\ldots,L^N$, are continuous on $\Bcal$.
\end{assumption}

\begin{restatable}{theorem}{theoremmatchingmle}\label{theorem:matching_mle}
    Suppose that \Cref{assumption:compact_continuity_certify} holds. Let $P^N \in \Mcal^{np, N}$ such that $P^N$ is realizable under $(\Qcal_\beta, s)$ in the sense of \Cref{definition:correct_fit}. Then
    \begin{align}
    \Psi^{\mle, N}(P^N) = \Psi^N(P^N) .
\end{align}
\end{restatable}

\Cref{theorem:matching_mle} proves that, under realizability and \Cref{assumption:compact_continuity_certify}, we can recast the identified set $\Psi^N(P^N)$ as a profile NPMLE. We further show that the upper endpoint $\Psi^{+,N}(P^N)$ of this set can be written as a supremum of $\sum_{k=1}^K q_k g(\beta_k), q \in \Delta_K$, over all discrete preference distributions $\sum_{k=1}^K q_k\delta_{\beta_k}$, with $K\leq \tilde L^N+2$, that match the moment vector $b^N(P^N)$.

The proof of \Cref{theorem:matching_mle} relies on a finite-dimensional
extension of Tchakaloff's theorem \citep[see][]{tchakaloff1957formules}. Tchakaloff's theorem states that the integral of every polynomial up to a fixed degree against a compactly supported
positive measure equals a nonnegative weighted sum of its values at finitely many points. We use an extension by proving a similar result on the space spanned by $1, t^N_1, \ldots, t^N_{\tilde{L}^N}, g$, instead of polynomials up to a certain degree \citep[see,][Theorem~5.1]{berschneider2012theorem}. It shows that any probability measure on $\Bcal$ can be replaced by a probability measure with finite support that preserves the moments of these functions. The argument is closely related to Carathéodory's theorem \citep{caratheodory1911variabilitatsbereich} but this theorem alone only states that for a map $\phi \colon \Bcal \to \RR^M$, a point already known to belong to $\Conv({\phi(\beta):\beta\in\Bcal})\subset\RR^M$ can be written as a convex combination of at most $M+1$ points of the set $\{\phi(\beta) \colon \beta\in\mathcal B\}$. The Tchakaloff step is precisely the assertion that the barycenter $\int_{\Bcal} \phi(\beta) d\pi(\beta)$ belongs to $\Conv({\phi(\beta) \colon \beta\in\Bcal})$, and hence can be represented by an atomic probability measure matching the same moments.

For any $N \geq 1, P^N \in \Mcal^{np,N}$, define
\begin{align}\label{equation:weoriuwjeioru}
\Ccal^N(P^N) = \left\{ (K,q,\beta) \colon K\leq \tilde L^N+2,\  q\in\Delta_K, \beta\in\Bcal^K,\sum_{k=1}^K q_k t^N(\beta_k)=b^N(P^N) \right\}.
\end{align}
An element $(K,q,\beta)\in\Ccal^N(P^N)$ consists of a number of support atoms $K$, a vector of weights $q \in \Delta_K$, and atom locations $\beta = (\beta_1, \ldots, \beta_K) \in \Bcal^K$, with $K\leq \tilde{L}^N+2$, such that the discrete distribution $\sum_{k=1}^K q_k\delta_{\beta_k}$ satisfies the moment restriction $\sum_{k=1}^K q_k t^N(\beta_k)=b^N(P^N)$. In other words, the set $\Ccal^N(P^N)$ consists of all finite-support representations of preference distributions that match the moment vector $b^N(P^N)$ induced by $P^N$.

\begin{restatable}{theorem}{theoremdefinitionpsiwithsets}\label{theorem:definition_psi_with_sets}
Let $P^N \in \Mcal^{np, N}$ be realizable under
$(\Qcal_\beta, s)$ in the sense of \Cref{definition:correct_fit}. Then
\begin{align}\label{equation:riuweoiru}
\Psi^{+,N}(P^N) = \sup_{ (K,q,\beta) \in \Ccal^N (P^N)} \sum_{k=1}^K q_k g(\beta_k) .
\end{align}
\end{restatable}

The right-hand side of \eqref{equation:riuweoiru} is a supremum rather than a maximum because, although $K$ is bounded and $q$ ranges over a finite-dimensional simplex, $\beta_1,\ldots,\beta_K$ range over $\Bcal$, and without compactness of $\Bcal$, the set $\Ccal^N(P^N)$ from \eqref{equation:weoriuwjeioru} need not be compact, and attainment is not guaranteed. However, under \Cref{assumption:compact_continuity_certify}, $\Bcal^K$ is compact, thus $\Ccal^N(P^N)$ is compact as well, and the supremum becomes a maximum.

\Cref{theorem:definition_psi_with_sets} proves the existence of an optimal distribution with a finite support. It justifies the optimization procedure from \Cref{section:optimization_procedure}

\section{Target characterization via the EM algorithm}\label{section:optimization_procedure}

In this section, we propose an EM algorithm based procedure that certifies membership of a given target candidate value $\psi$ in the NPMLE set $\Psi^{\mle,N}(P^N)$. Our procedure is inspired by the alternating optimization structure of the EM algorithm \citep{dempster1977maximum} with latent data. In standard EM, the two steps are constructed to increase the observed-data likelihood under a specified latent-variable model. Instead, for a fixed candidate of the target $\psi$, we alternate $\KL$ projections between a data-matching set, and an auxiliary model set. The procedure therefore does not seek a single maximum-likelihood estimate of a point-identified latent model. Instead, it searches for a zero-KL intersection of these two sets, which, by \Cref{proposition:validity_idealized_certify}, certifies that $\psi \in \Psi^{\mle,N}(P^N)$.

\paragraph{Data matching set.} In this paragraph, we introduce what we refer to as a \emph{data-matching set}. This terminology originates in the information geometry literature \citep{amari1995information,amari2016information}. In words, it consists of the set of probability distributions over what we have referred to earlier as the full data-structure  that marginalize to the observed data-distribution. Formally, for any positive integer $N$, and scalar $\psi \in [0,1]$, define for any $l \in [L^N]$
\begin{align}\label{equation:definition_objects_alpha}
    e_l^N &= e^N(a(l)) , \quad  \bar{\alpha}_l^N(P^N) = e_l^N \bar{b}_l^N(P^N) .
\end{align}
We interpret $e_l^N$ as the fraction of individuals under the design of experiment $N$ that are exposed to choice set tuple $a(l)$. We interpret $\bar{b}^N_l(P^N)$ as the probability for a unit in experiment $N$ of choosing alternative $y(l)$ under choice set $a(l)$. We note that for any $l \in [L^N], \bar{\alpha}_l^N(P^N) \geq 0$, and $\sum_{l=1}^{L^N} \bar{\alpha}^N_l(P^N) = 1$. Hence, $(\bar\alpha_l^N(P^N))_{l \in [L^N]}$ is a probability mass function on $[L^N]$, and corresponds to the probability that an experimental unit is assigned to $a(l)$ and chooses $y(l)$. We interpret it as the joint distribution of a design-choice cell. 

For any $P^N \in \Mcal^{np, N}$, and $\psi\in[0,1]$, let $\varpsi \colon [2] \to [0,1], 1 \mapsto \psi, 2 \mapsto 1-\psi$, and define
\begin{align}\label{eq:full_cell_constraint_set}
    \Rcal^{N,K} (P^N, \psi) = \left\{
    r \in \Delta_{K \times L^N \times 2} \colon \sum_{k=1}^K \sum_{m=1}^2 r(k, l, m) = \bar{\alpha}^N_l(P^N) , \sum_{k=1}^K \sum_{l=1}^{L^N} r(k,l, m) = \varpsi(m) \right\}.
\end{align}

In words, $\Rcal^{N,K}(P^N, \psi)$ is the set of probability vectors on $[K] \times [L^N]$ such that for any $r \in \Rcal^{N,K}(P^N)$, the marginal of $r$ with respect to $l \in [L^N]$ is equal to the moments $(\bar{\alpha}^N_l (P^N))_{l \in [L^N]}$, and the marginal of $r$ with respect to $m$ is equal to the moments of $g$.

\paragraph{Model set.} In this paragraph, we introduce what we refer to as an \emph{auxiliary model set}, such that for any $\psi \in [0,1]$, the $\KL$ distance between $\Rcal^{N,K}(P^N, \psi)$ and this set is null if, and only if $\psi \in \Psi^{\mle, N}$ (see \Cref{proposition:validity_idealized_certify}). For any positive integer $K$, probability vector $w \in \Delta_K$, and tuple $\beta  = (\beta_1, \ldots, \beta_K) \in \Bcal^K$, let $\gamma \colon [2] \times \Bcal \to [0,1], (m, \beta) \mapsto \1_{\{1\}}(m) g(\beta) + \1_{\{1\}}(m - 1) (1-g(\beta))$, and define
\begin{align}\label{equation:finite_atom_cell_model}
    q_{w,\beta}^N (k,l, m) & = w_k e_l^N \bar{t}_l^N(\beta_k) \gamma (m, \beta_k) , 
    \quad k\in[K],\ l \in [L^N] , m \in [2] , \\
    \Mcal_K^{N} & = \left\{q_{w,\beta}^N \colon w \in \Delta_K, \beta_1,\dots, \beta_K \in\Bcal \right\} ,
\end{align}
as well as
\begin{align}\label{equation:finite_atom_kl_distance}
    D_K^N(P^N,\psi) = \inf_{w \in \Delta_K, \beta \in \Bcal^K} \inf_{r \in \Rcal^{N,K}(P^N, \psi)} \KL(r \| q_{w, \beta}^N ) .
\end{align}
In words, $D_K^N(P^N,\psi)$ is the infimal $\KL$ distance between an element of the set $\Rcal^{N,K}(P^N, \psi)$, and an element of the set $\Mcal_K^N$. Kullback--Leibler divergence is nonnegative. If the two infima in \eqref{equation:finite_atom_kl_distance} are attained, then $D_K^N (P^N, \psi) = 0$ if, and only if $\Rcal^{N,K}(P^N,\psi)$ and $\Mcal_K^N$ contain a common probability mass function.

\begin{restatable}{proposition}{propositionvalidityidealizedcertify}
\label{proposition:validity_idealized_certify}
Let $P^N\in\Mcal^{np,N}$ be realizable under $(\Qcal_\beta,s)$. Under
\Cref{assumption:compact_continuity_certify}, for any $\psi\in[0,1]$, we have that there exists $K \leq \tilde{L}^N+2$ such that $D_K^N(P^N,\psi)=0$ if, and only if $\psi \in \Psi^{\mle, N}(P^N)$.
\end{restatable}

\begin{algorithm}[H]
\caption{$\textsc{Certify}(\psi,P^N, K, \epsilon, H_{\max})$}
\label{algorithm:fixed_psi_certify}
\begin{algorithmic}[1]
\State \textbf{Input} Candidate $\psi$, distribution $P^N$, atom number $K$, tolerance
$\varepsilon$, maximum iterations
$H_{\max}$.
\State Draw random initialization
$(w^{(0)},\beta^{(0)}) \in \Delta_K \times\Bcal^K$.
\State Compute $r^{(0)} \leftarrow \textsc{LeftProjection}(\psi, w^{(0)}, \beta^{(0)}, P^N, \varepsilon)$.
\For{$h = 0, 1, \ldots,H_{\max}-1$}
    \State 
    \begin{align}
        (w^{(h+1)}, \beta^{(h+1)}) & \leftarrow \textsc{RightProjection}(K, \Bcal, r^{(h)}, \Bar{t}, e, \gamma) , \\
        r^{(h+1)} & \leftarrow
        \textsc{LeftProjection}(\psi, w^{(h+1)}, \beta^{(h+1)}, P^N) , \\
        d^{(h+1)}(\psi) & \leftarrow \KL( r^{(h+1)} \| q_{w^{(h+1)}, \beta^{(h+1)}}^N ).
    \end{align}
    \If{$d^{(h+1)}(\psi)\leq \varepsilon$}
        \State \Return $\mathrm{TRUE}$, $w^{(h+1)},\beta^{(h+1)},r^{(h+1)},d^{(h+1)}(\psi)$
    \EndIf
\EndFor
\State \Return $\mathrm{FALSE},\varnothing,\varnothing,\varnothing,\varnothing$
\end{algorithmic}
\end{algorithm}

\begin{assumption}\label{assumption:left_projection_support}
For any $l\in[L^N], m \in [2]$, whenever $\bar{\alpha}^N_l(P^N)>0$, we necessarily have $\sum_{k=1}^K q^N_{w,\beta}(k,l,m)>0$.
\end{assumption}

$\textsc{LeftProjection}( \psi,w,\beta,P^N)$ is any optimization oracle that returns a minimizer
$r \in \Rcal^{N,K}(P^N, \psi)$ of
the optimization problem
\begin{align}\label{equation:left_projection_problem}
\inf_{r \in \Rcal^{N,K}(P^N, \psi)}
\KL (r \| q_{w,\beta}^N) ,
\end{align}
and we refer to \eqref{equation:left_projection_problem} as the left projection problem. Under \Cref{assumption:left_projection_support}, by \Cref{lemma:left_proj_solution}, the
minimizer of \eqref{equation:left_projection_problem} is obtained by an exponential tilting of $q^N_{w, \beta}$.

\begin{assumption}\label{assumption:convexity_beta_optim}
    $\Bcal$ is a convex subset of $\RR^d$, for a positive integer $d$, and for any $l \in [L^N]$, the mappings $\bar{t}_l^N \colon \Bcal\to(0,\infty)$, and $g \colon \Bcal\to(0,1)$ are log-concave.
\end{assumption}

$\textsc{RightProjection}(K,\Bcal,r,\bar t^N,e^N, \gamma)$ is any optimization oracle that returns a minimizer
$((w_k)_{k \in [K]},$ $(\beta_k)_{k \in [K]}) \in \Delta_K \times \Bcal^K$ of the optimization problem
\begin{align}\label{equation:eroiueroiuwe}
\min_{w \in\Delta_K,(\beta_1,\ldots,\beta_K)\in\Bcal^K}
\sum_{k=1}^K \sum_{l=1}^{L^N} \sum_{m = 1}^2
r(k,l, m)\log\frac{r(k,l, m)}{w_k e_l^N \bar{t}^N_l(\beta_k) \gamma(m, \beta_k)} .
\end{align}
and we refer to \eqref{equation:eroiueroiuwe} as the right projection problem. \Cref{lemma:right_proj_solution} characterizes any minimizer $(w_k, \beta_k)_{k \in [K]}$ of \eqref{equation:eroiueroiuwe}. $(w_k)_{k \in [K]}$ is given explicitly, and $(\beta_k)_{k \in [K]}$ is given as the solution of a minimization problem. Additionally, under \Cref{assumption:convexity_beta_optim}, and log-concavity of $1-g$, \Cref{lemma:convex_atom_update} proves that the optimization problem characterizing any minimizer $(\beta_k)_{k \in [K]}$ of \eqref{equation:eroiueroiuwe} is convex. Therefore, if the maps $(\log \bar{t}_l)_{l \in [L^N]}$ are differentiable, and projection onto $\Bcal$ is available, $(\beta_k)_{k \in [K]}$ can be performed using a projected gradient method.

\begin{assumption}\label{assumption:interior_point}
There exist $r^\star \in \Rcal^{N, K}(P^N, \psi), (w^\star,\beta^\star) \in \Delta_K \times \Bcal^K$ such that $r^\star = q^N_{w^\star, \beta^\star} \in \Rcal^{N,K}(P^N,\psi)$, and for any $k\in[K], w^\star_k>0, \beta^\star_k \in \mathrm{int}(\Bcal)$.
\end{assumption}

\begin{assumption}\label{assumption:local_identifiability}
$\Bcal$ is a subset of $\RR^d$, for a positive integer $d$, and there exist an open $\Theta_0\subseteq \RR^{K-1} \times \RR^{\dim (\Bcal) K}$ and a neighborhood $U_0$ of $(w^\star, \beta^\star)$ in $\Delta_K \times \Bcal^K$, with $\overline{U_0}\subseteq \Delta_K \times \mathrm{int}(\Bcal)^K$, such that
\begin{enumerate}
\item[(i)] the map $\iota \colon \Theta_0 \to U_0, \theta=(w_1,\ldots,w_{K-1},\beta_1,\ldots,\beta_K) \mapsto (w,\beta)$, $w_K = 1-\sum_{k<K}w_k$, is a $\Ccal^1$ bijection with $\iota(\theta^\star)=(w^\star,\beta^\star)$;
\item[(ii)] $(w^\star,\beta^\star)$ is the unique element $(w, \beta) \in \overline{U_0}$ such that $q_{(w, \beta)} = q_{(w^\star, \beta^\star)}$.
\end{enumerate}
\end{assumption}

\begin{assumption}\label{assumption:smoothness_certify}
There exists an open $B_0\subseteq \mathrm{int}(\Bcal)$ containing $\{\beta^\star_k\}_{k\in[K]}$ on which $\bar t_l^N(\cdot)$ $l\in[L^N]$, and $g(\cdot)$ are twice continuously differentiable.
\end{assumption}

\begin{assumption}\label{assumption:Icom_nonsingular}
$\theta\mapsto \KL(r^\star\| q_{\iota(\theta)})$ is twice continuously differentiable on $\Theta_0$, and
\begin{align}
\nabla^2_\theta\, \KL(r^\star \| q_{\iota(\theta)}) \big|_{\theta^\star}
\end{align}
is positive definite.
\end{assumption}

\begin{assumption}\label{assumption:missing_info}
Writing $\rho(\theta)=\inf_{r\in\Rcal^{N,K}(P^N,\psi)}\KL(r \| q_{\iota(\theta)}), \rho$ is twice continuously differentiable at $\theta^\star$, and we have
\begin{align}
\nabla^2_\theta\,\rho(\theta)\big|_{\theta^\star} \succ 0 .
\end{align}
\end{assumption}

\begin{assumption}\label{assumption:local_selection}
There is a neighborhood $W_0\subseteq \Rcal^{N,K}(P^N,\psi)$ of $r^\star$ such that for any $r \in W_0, \mathrm{RightProjection}(K,\Bcal,r)$ is unique, and we have that $r \mapsto \mathrm{RightProjection}(K,\Bcal,r) \in \Ccal^1(W_0)$.
\end{assumption}

In EM terminology, the left projection plays the role of an E-step, while the right projection plays the role of an M-step.
\Cref{assumption:interior_point} requires the existence of an optimum of \eqref{equation:finite_atom_kl_distance} in the interior of the data matching and model sets. \Cref{assumption:local_identifiability} requires the model parameter to be
locally identified near this fixed optimum.
\Cref{assumption:Icom_nonsingular} and \Cref{assumption:missing_info} require
nonsingular complete data and observed data information matrices, respectively. The proof of \Cref{theorem:local_basin_certify} shows that these curvature conditions make the EM update locally contractive. \Cref{assumption:smoothness_certify} is used to establish that the E-step is continuously differentiable with respect to the current parameter, a property that holds trivially when the E-step has a closed-form expression. \Cref{assumption:local_selection}
requires the M-step solution to be locally unique and continuously
differentiable. Note that we prove that these conditions ensure convergence for initializations near $(w^\star, \beta^\star)$, but do not imply convergence from arbitrary initializations.

\begin{restatable}{theorem}{localbasincertify}
\label{theorem:local_basin_certify}
Suppose that \Cref{assumption:convexity_beta_optim,assumption:compact_continuity_certify,assumption:interior_point,assumption:local_identifiability,assumption:smoothness_certify,assumption:Icom_nonsingular,assumption:missing_info,assumption:local_selection} hold. Then, there exists a subset $U \subseteq \Delta_K \times \Bcal^K$ such that, if $(w^{(0)}, \beta^{(0})) \in U$, we have that $d^{(h)}(\psi) \to 0$ as $h \to \infty$. As a consequence, for any $\epsilon > 0$, there exists a positive integer $H$ such that for any $h \geq H$, we have that $\PP(\mathrm{Certify}(\psi,P^N,K,\epsilon, h) = \mathrm{True}) > 0$.
\end{restatable}

Like \citet[][Theorem 1]{hero1995convergence}, \Cref{theorem:local_basin_certify} proves local linear convergence through contractivity near a fixed point. The proof of \citet[][Theorem 1]{hero1995convergence} relies on the existence of a neighborhood of the optimum in which EM updates contract the distance to the fixed point by a factor $\alpha < 1$, whereas we prove that the Jacobian of the update has norm strictly smaller than one at the fixed point, and extend this property to a neighborhood by continuity.

\paragraph{Specialization to the mixed MNL.} We now specialize the optimization procedure to a mixed multinomial logit model. We show that the logit structure makes each atom update a convex optimization problem and provide sufficient conditions for
\Cref{assumption:local_identifiability,assumption:smoothness_certify,assumption:Icom_nonsingular,assumption:missing_info,assumption:local_selection} to hold. Together with the existence of an interior solution to \eqref{equation:finite_atom_kl_distance}, these
results allow \Cref{theorem:local_basin_certify} to hold, and establish local convergence of \Cref{algorithm:fixed_psi_certify}.

\begin{assumption}\label{assumption:MNL_general_design}
$J \geq 2$, $M \geq1$, and $T=1$. The choice sets $\tilde{a}_1,\ldots, \tilde{a}_M, \tilde{a}^\star \in \Acal_1$ are nonempty, $\tilde{a}^\star \notin\{\tilde{a}_1, \ldots, \tilde{a}_M\}$, and $\Card(\tilde{a}^\star) \geq 2$. Moreover, $\{\tilde{a}_i^N \colon i\in[N]\} = \{\tilde{a}_1, \ldots, \tilde{a}_M\}$ and $e_r^N=e^N(\tilde{a}_r) > 0$ for any $r \in [M]$.
\end{assumption}

\begin{assumption}[Preferences, kernel, and target]\label{assumption:MNL_general_kernel}
$d \geq 1, \Bcal$ is a compact and convex subset of $\RR^d$ with nonempty interior, and $z_j \in \RR^d$ for any $j\in[J]$. Let $y^\star\in \tilde{a}^\star$. For any $\beta \in \Bcal$, $\tilde{a} \in \Acal_1$, and $y\in[J]$,
\begin{align}\label{equation:s_instantiated_mnl}
s(\beta, \tilde{a}, y) =
\begin{cases}
\dfrac{\exp(z_y^\top\beta)}{\sum_{j\in \tilde{a}} \exp(z_j^\top\beta)}&\text{if } y \in \tilde{a} , \\
0 & \text{otherwise},
\end{cases} ,
\end{align}
and $g(\beta) = s(\beta, \tilde{a}^\star, y^\star)$. For $l \in [L^N]$ such that $a(l) = \tilde{a}_r, e_l^N = e_r^N, \bar{t}^N_l(\beta) = s(\beta, \tilde{a}_r,y(l))$.
\end{assumption}

\begin{assumption}\label{assumption:MNL_general_richness}
For any $k \in [K]$, we have that
\begin{align}
\Span(\{z_j - z_{j'} \colon j,j' \in \tilde{a}_r,\ r\in[M]\} \cup \{\nabla g(\beta_k^\star) \})  = \RR^d .
\end{align}
\end{assumption}

We note that \Cref{assumption:MNL_general_richness} requires observed alternatives' feature differences to span $\RR^d$ entirely. We interpret it as condition on the richness of the observed features.

\begin{restatable}{lemma}{lemmamnlgeneralconvexity}\label{lemma:mnl_general_convexity}
Under \Cref{assumption:MNL_general_design,assumption:MNL_general_kernel}, \Cref{assumption:convexity_beta_optim} holds.
\end{restatable}

\begin{restatable}{lemma}{lemmamnlgeneralcompletecurvature}\label{lemma:mnl_general_complete_curvature}
Under \Cref{assumption:MNL_general_design,assumption:MNL_general_kernel,assumption:interior_point,assumption:MNL_general_richness}, we have that \Cref{assumption:local_identifiability,assumption:smoothness_certify,assumption:Icom_nonsingular} hold.
\end{restatable}

\begin{assumption}\label{assumption:MNL_general_projection_identification}
The Jacobian
\begin{align}
\left. \nabla_\theta \left(\left(\sum_{k=1}^K\sum_{m=1}^2 q_{\iota(\theta)} (k, l, m) \right)_{l\in[L^N]}, \sum_{k=1}^K \sum_{l=1}^{L^N} q_{\iota(\theta)}(k,l,1) \right) \right|_{\theta^\star}
\end{align}
has full column rank. Moreover, for any $k\in[K]$ and $\beta \in \Bcal$, if we have that
\begin{align}
\text{for any } r \in [M], y \in \tilde{a}_r, s(\beta, \tilde{a}_r, y) = s(\beta_k^\star, \tilde{a}_r, y) , \text{ and } g(\beta) = g(\beta_k^\star) ,
\end{align}
then we have that $\beta = \beta_k^\star$.
\end{assumption}

\begin{remark}[Single-component MNL case]
Suppose that $K=1$. Then $w_1=1$, $\theta=\beta$, and $q_\beta^N(1,l,m) = e_l^N\bar{t}^N_l (\beta) \gamma(m,\beta)$. Consequently, we have that
$\sum_{m=1}^2 q_\beta^N (1,l,m) = e_l^N\bar{t}^N_l (\beta)$, and $\sum_{l=1}^{L^N}q _\beta^N(1,l,1) = g(\beta)$. Since $e_l^N>0$, the
full-column-rank condition in
\Cref{assumption:MNL_general_projection_identification} reduces to
\begin{align}\label{equation:weriouweriu}
\text{rank} \, \nabla_\beta \left(\left( s(\beta,\tilde a_r,y) \right)_{r\in[M],\,y\in\tilde a_r}, g(\beta) \right) \bigg|_{\beta^\star} = d .
\end{align}
For the multinomial logit kernel, we show in the proofs (see \eqref{equation:wriewriu}) that $\nabla_\beta s(\beta,\tilde a_r,y) = s(\beta,\tilde a_r,y)\{
z_y-\sum_{j\in\tilde a_r}
s(\beta,\tilde a_r,j)z_j \}$. Hence the rank condition in \eqref{equation:weriouweriu} is equivalent to
\begin{align}
\Span(\{z_j-z_{j'} \colon j,j'\in \tilde{a}_r,\ r\in[M] \} \cup \left\{\nabla g(\beta^\star)\right\}) = \RR^d ,
\end{align}
which corresponds to \Cref{assumption:MNL_general_richness}.
\end{remark}

\begin{restatable}{lemma}{lemmamnlgeneralprojectionregularities}\label{lemma:mnl_general_projection_regularities}
Under \Cref{assumption:MNL_general_design,assumption:MNL_general_kernel,assumption:interior_point,assumption:MNL_general_richness,assumption:MNL_general_projection_identification}, we have that \Cref{assumption:missing_info,assumption:local_selection} hold.
\end{restatable}

Note that mixed MNLs are widely used in discrete choice analysis to represent preference heterogeneity and substitution patterns. \Cref{lemma:mnl_general_complete_curvature,lemma:mnl_general_projection_regularities}
make explicit the local regularity conditions required by \Cref{theorem:local_basin_certify} to hold in the case of mixed MNLs.

\section{Conclusions}

Finitely many observed choice probabilities need not identify either an unrestricted preference distribution or a linear functional of it. We therefore conduct inference on the set of values of the linear functional of interest compatible with observed choice probabilities, without imposing restrictions that force point identification. We represent the upper endpoint of this set by an infinite-dimensional linear program and a finite-dimensional dual. A nondegenerate primal optimizer yields strong duality of the primal program, uniqueness of the dual optimizer, pathwise differentiability of the endpoint, and an asymptotically normal plug-in estimator. We provide conditions under which the identified set also coincides with the minimizers of a profile negative log-likelihood. A finite-support representation yields a KL-based membership criterion and an alternating-projection algorithm with local convergence guarantees. In applications such as streaming platforms, analysts can experimentally vary the choice sets shown to users and the frequencies with which users face them. Our results provide a basis for formulating experimental design as the joint selection of choice sets and assignment frequencies to narrow the identified set and increase the precision of inference on a prespecified counterfactual choice probability.

\bibliography{sample}

\appendix

\section{Supporting Lemmas}

\begin{lemma}\label{lemma:equivalence_matching_moments}
    Let $P^N \in \Mcal^{np,N}$ such that $P^N$ is realizable under $(\Qcal_\beta, s)$, for any $Q_\beta \in \Qcal_\beta, M_{\Ycal^N}(P^{F,s,N}(Q_\beta))=P^N$ if, and only if $\int_\Bcal s(\beta, a, y) dQ_\beta(\beta) = Q^{np,N}_{Y \mid a}(P^N)(a,y)$ for any $a \in \{a^N_1,\ldots,a^N_N\}, y \in a$.
\end{lemma}

\begin{proof}[Proof of \Cref{lemma:equivalence_matching_moments}]
Let $P^N \in \Mcal^{np,N}$ such that $P^N$ is realizable under $(\Qcal_\beta, s)$, and $Q_\beta \in \Qcal_\beta$. Consider $P^{F,s,N}(Q_\beta) \in \Mcal^{F,s,N}$ such that for any $Z^N =((B_1,Y^N_1),\ldots,(B^N, Y^N_N)) \in \mathscr{Z}^N$, we have
\begin{align}
    P^{F,s,N}(Q_\beta)(Z^N) =  \prod_{i=1}^N \int_{B_i} dQ_\beta(\beta_i) \sum_{y \in Y^N_i} s(\beta_i, a^N_i, y).
\end{align}
For any
$(y^N_1,\ldots,y^N_N)\in \Ycal^N$, we have
\begin{align}
M_{\Ycal^N}(P^{F,s,N}(Q_\beta))(y^N_1,\ldots,y^N_N) & = \int_{\beta_1 \in \Bcal} \ldots \int_{\beta_N \in \Bcal} d P^{F,s,N} (\beta_1,y^N_1,\ldots,\beta_N,y^N_N) \\
& = \int_{\beta_1 \in \Bcal} \ldots \int_{\beta_N \in \Bcal} \prod_{i=1}^N
 s(\beta_i,a^N_i,y^N_i)dQ_\beta(\beta_i) \\
& \label{equation:prodtilde} = \prod_{i=1}^N \int_{\Bcal} dQ_\beta(\beta_i) s(\beta_i, a^N_i, y^N_i) ,
\end{align}
and for any $a \in \{a^N_1, \ldots, a^N_N\}, y \in a$, we have
\begin{align}
    & Q^{np,N}_{Y\mid a}(M_{\Ycal^N}(P^{F,s,N}(Q_\beta)))(a,y) \\
    & \quad = \frac{1}{n_a} \sum_{k \in [N] \colon a^N_k=a} \sum_{y^N_{-k}}M_{\Ycal^N}(P^{F,s,N}(Q_\beta))(y^N_1, .. , y^N_{k-1}, y, y^N_{k+1}, .., y^N_N) \\
    & \quad = \frac{1}{n_a} \sum_{k \in [N] \colon a^N_k=a} \sum_{y^N_{-k}} \int_{\Bcal} dQ_\beta(\beta) s(\beta, a, y)  \prod_{i \in [N], i\ne k} \int_{\Bcal} dQ_\beta(\beta_i) s(\beta_i, a^N_i, y^N_i)  \\
    & \quad = \frac{1}{n_a} \sum_{k \in [N] \colon a^N_k=a} \int_{\Bcal} dQ_\beta(\beta) s(\beta, a, y) \sum_{y^N_{-k}}  \prod_{i\in [N], i\ne k} \int_{\Bcal} dQ_\beta(\beta_i) s(\beta_i, a^N_i, y^N_i)  \\
    & \quad =\frac{1}{n_a} \sum_{k \in [N] \colon a^N_k=a}\int_{\Bcal} dQ_\beta(\beta) s(\beta, a, y) \prod_{i\in [N], i\ne k}  \sum_{y^N_i \in a^N_i} \int_{\Bcal} dQ_\beta(\beta_i) s(\beta_i, a^N_i, y^N_i) \\
    & \quad =\frac{1}{n_a} \sum_{k \in [N] \colon a^N_k=a}\int_{\Bcal} dQ_\beta(\beta) s(\beta, a, y) \prod_{i\in [N], i\ne k} \int_{\Bcal} \underbrace{\sum_{y^N_i \in a^N_i} s(\beta_i, a^N_i, y^N_i)}_{=1} dQ_\beta(\beta_i) \\
      & \quad =\frac{1}{n_a} \sum_{k \in [N] \colon a^N_k=a}\int_{\Bcal} dQ_\beta(\beta) s(\beta, a, y) \prod_{i \in [N], i \ne k} \underbrace{\int_{\Bcal} dQ_\beta(\beta_i)}_{=1} \\
       & \quad =\frac{1}{n_a} \sum_{k \in [N] \colon a^N_k=a}\int_{\Bcal} dQ_\beta(\beta) s(\beta, a, y) \\
       & \label{equation:rwerk} \quad = \int_{\Bcal} dQ_\beta(\beta) s(\beta, a, y) ,
\end{align}
which proves that
\begin{align}\label{equation:dflkjerlkj}
    Q^{np,N}_{Y\mid a}(M_{\Ycal^N}(P^{F,s,N}(Q_\beta)))(\cdot,\cdot) = \int_{\Bcal} dQ_\beta(\beta) s(\beta, \cdot,\cdot) .
\end{align}
Suppose first that $M_{\Ycal^N}(P^{F,s,N}(Q_\beta))=P^N$. We thus have that $Q^{np,N}_{Y\mid a}(M_{\Ycal^N}(P^{F,s,N}(Q_\beta)))= Q^{np,N}_{Y\mid a}(P^N)$, hence \Cref{equation:rwerk} gives that for any $a \in \{a^N_1,\ldots,a^N_N\}, y \in a$, we have
\begin{align}
    Q^{np,N}_{Y\mid a}(M_{\Ycal^N}(P^{F,s,N}(Q_\beta)))(a,y) = \int_{\Bcal} dQ_\beta(\beta) s(\beta, a, y) = Q^{np,N}_{Y\mid a}(P^N)(a,y),
\end{align}
hence the first direction of the result.

Conversely, suppose that $Q_\beta \in \Qcal_\beta$ such that for any $a\in\{a^N_1,\ldots,a^N_N\}$ and every
$y\in a$,
\begin{align}\label{equation:moment_match}
\int_{\Bcal}s(\beta,a,y)\,dQ_\beta(\beta)
= Q^{np,N}_{Y\mid a}(P^N)(a,y) .
\end{align}
Since $P^N$ is realizable under $(\Qcal_\beta, s)$, there exists
$Q_\beta^\star\in\Qcal_\beta$ such that
\begin{align}\label{equation:realizable_star}
P^N=M_{\Ycal^N}(P^{F,s,N}(Q_\beta^\star)).
\end{align}
Applying the first direction of the proof to $Q_\beta^\star$, we obtain that,
for every $a\in\{a^N_1,\ldots,a^N_N\}$ and every $y\in a$,
\begin{align}\label{equation:star_moment_match}
\int_{\Bcal}s(\beta,a,y)\,dQ_\beta^\star(\beta) = Q^{np,N}_{Y\mid a}(P^N)(a,y).
\end{align}
Combining \eqref{equation:moment_match} and \eqref{equation:star_moment_match}, we get that,
for every $a\in\{a^N_1,\ldots,a^N_N\}$ and every $y\in a$,
\begin{align}\label{equation:same_choice_probs}
\int_{\Bcal}s(\beta,a,y)\,dQ_\beta(\beta)
=  Q^{np,N}_{Y\mid a}(P^N)(a,y) = \int_{\Bcal}s(\beta,a,y)\,dQ_\beta^\star(\beta).
\end{align}
Now let $(y^N_1,\ldots,y^N_N)\in\Ycal^N$. For any $i \in [N]$, we can apply \eqref{equation:same_choice_probs} to
$(a^N_i,y^N_i)$, which gives that
\begin{align}
\int_{\Bcal}s(\beta,a^N_i,y^N_i)\,dQ_\beta(\beta) = \int_{\Bcal} s(\beta,a^N_i,y^N_i) dQ_\beta^\star (\beta) ,
\end{align}
and the products of these terms are also equal, that is
\begin{align}
\prod_{i=1}^N
\int_{\Bcal}s(\beta,a^N_i,y^N_i)\,dQ_\beta(\beta) = \prod_{i=1}^N
\int_{\Bcal}s(\beta,a^N_i,y^N_i)\,dQ_\beta^\star(\beta).
\end{align}
Using the product representation of the marginal law, we therefore have
\begin{align}
M_{\Ycal^N}(P^{F,s,N}(Q_\beta))(y^N_1,\ldots,y^N_N)
&= \prod_{i=1}^N
\int_{\Bcal}s(\beta,a^N_i,y^N_i)\,dQ_\beta(\beta) \\
&= \prod_{i=1}^N
\int_{\Bcal}s(\beta,a^N_i,y^N_i)\,dQ_\beta^\star(\beta) \\
&= M_{\Ycal^N}(P^{F,s,N}(Q_\beta^\star))(y^N_1,\ldots,y^N_N) \\
&= P^N(y^N_1,\ldots,y^N_N),
\end{align}
where the last equality follows from \eqref{equation:realizable_star}. Consequently, we have that $M_{\Ycal^N}(P^{F,s,N}(Q_\beta))=P^N$, which proves the reverse direction and concludes the proof.
\end{proof}

\begin{example}
The equivalence in \Cref{lemma:equivalence_matching_moments} may fail without
realizability. Let $N=2$, $T=1$, $J=2$, and $a^2_1=a^2_2=a=\{1,2\}$. Define
$P^N\in\Mcal^{np,N}$ by
\begin{align}
P^N(1, 1) = 1/2 , \quad P^N(2, 2) = 1/2 , \quad P^N(1, 2) = P^N(2, 1) = 0 .
\end{align}
Then, we have $Q^{np,N}_{Y\mid a}(P^N)(a,1)
= Q^{np,N}_{Y\mid a}(P^N)(a,2) = 1/2$. Now let $\Bcal = \{\beta_0 \}$, $Q_\beta = \delta_{\beta_0}$, and define the choice
kernel by
\begin{align}
s(\beta_0,a,1)= 1 / 2,
\qquad
s(\beta_0,a,2)= 1 / 2 .
\end{align}
Therefore, for $y\in\{1,2\}$, we have
\begin{align}
\int_\Bcal s(\beta,a,y)\,dQ_\beta(\beta)
= \frac{1}{2} = Q^{np,N}_{Y\mid a}(P^N)(a,y).
\end{align}
Thus the moment restrictions are satisfied. However, the marginal law on
$\Ycal^N$ is
\begin{align}
M_{\Ycal^N}(P^{F,s,N}(\delta_{\beta_0}))(y^2_1,y^2_2) = \prod_{i=1}^2
\int_\Bcal s(\beta,a^2_i,y^2_i)\,dQ_\beta(\beta)=\frac{1}{4} ,
\end{align}
for every $(y^2_1,y^2_2)\in\{1,2\}^2$. In particular,
\begin{align}
M_{\Ycal^N}(P^{F,s,N}(\delta_{\beta_0}))(1,2) = 1/4 \neq 0 = P^N(1,2),
\end{align}
which proves that in the absence of realizability, if one matches the choice probabilities $Q^{np,N}_{Y\mid a}(P^N)(\cdot, \cdot) = \int_\Bcal s(\beta, \cdot, \cdot) dQ_\beta(\beta)$, it does not necessarily imply $M_{\Ycal^N}(P^{F,s,N}(\delta_{\beta_0})) = P^N$ .
\end{example}

\begin{lemma}\label{lemma:redundant_constraints_realizable}
Let $P^N\in\Mcal^{np,N}$ be realizable under $(\Qcal_\beta,s)$. Then, for any
$Q_\beta\in\Qcal_\beta$ satisfying
\begin{align}
\int_\Bcal t_l^N(\beta)\,dQ_\beta(\beta)=b_l^N(P^N),
\qquad l=1,\ldots,\tilde L^N,
\end{align}
we also have
\begin{align}
\int_\Bcal t_l^N(\beta)\,dQ_\beta(\beta)
= \Bar{b}^N_l(P^N),
\quad l=\tilde{L}^N+1,\ldots,L^N.
\end{align}
\end{lemma}

\begin{proof}
Fix $l>\tilde L^N$. Since
$t_l^N\in\mathrm{Span}\{1,t_1^N,\ldots,t_{\tilde L^N}^N\}$, there exist
$c_{l0},c_{l1},\ldots,c_{l\tilde L^N}\in\RR$ such that for any $\beta\in\Bcal$
\begin{align}
t_l^N(\beta)=c_{l0}+\sum_{m=1}^{\tilde L^N}c_{lm}t_m^N(\beta) .
\end{align}
Since $P^N$ is realizable, there exists $Q_\beta^\star$ such that
$P^N=M_{\Ycal^N}(P^{F,s,N}(Q_\beta^\star))$, hence
\begin{align}
\Bar{b}^N_l(P^N) = \int_\Bcal t_l^N(\beta)\,dQ_\beta^\star(\beta) = c_{l0}+\sum_{m=1}^{\tilde{L}^N}c_{lm}b_m^N(P^N).
\end{align}
If $Q_\beta \in \Qcal_\beta$ satisfies the first $\tilde L^N$ moment restrictions, then
\begin{align}
\int_\Bcal t_l^N(\beta)\,dQ_\beta(\beta) = c_{l0}+\sum_{m=1}^{\tilde L^N}c_{lm}
\int_\Bcal t_m^N(\beta)\,dQ_\beta(\beta) = c_{l0}+\sum_{m=1}^{\tilde{L}^N}c_{lm}b_m^N(P^N) ,
\end{align}
and combining the two displays gives the claim.
\end{proof}

\section{Proofs of the Von Mises expansion of the target functional}

\theoremrepresentation*

\subsection{Proof of \Cref{theorem:representation1}}

\begin{proof}[Proof of \Cref{theorem:representation1}]
Let $P^N \in \Mcal^{np,N}$ such that $P^N$ is realizable under $(\Qcal_\beta, s)$. By definition of $\Psi^{+,N}$, we have
\begin{align}
    \Psi^{+,N}(P^N)
    &= \sup \Psi^N(P^N) \\
    &= \sup \left\{\tilde{\Psi}(Q_\beta) \colon Q_\beta \in \Qcal_\beta,
    M_{\Ycal^N}(P^{F,s,N}(Q_\beta))=P^N \right\} \\
    & \label{eqaution:rkeljrklej}= \sup \left\{\int_\Bcal g(\beta)dQ_\beta(\beta) \colon Q_\beta \in \Qcal_\beta,
    M_{\Ycal^N}(P^{F,s,N}(Q_\beta))=P^N \right\}.
\end{align}
Since $P^N$ is realizable under $(\Qcal_\beta, s)$, we can apply
\Cref{lemma:equivalence_matching_moments}. Therefore, for any
$Q_\beta \in \Qcal_\beta$, we have that
\begin{align}
    M_{\Ycal^N}(P^{F,s,N}(Q_\beta))=P^N ,
\end{align}
if, and only if, for any $a \in \{a^N_1,\ldots,a^N_N\}$ and every $y\in a$,
\begin{align}
    \int_\Bcal s(\beta,a,y)dQ_\beta(\beta) = Q^{np,N}_{Y\mid a}(P^N)(a,y).
\end{align}
Hence the set
\begin{align}
    \left\{Q_\beta \in \Qcal_\beta \colon
    M_{\Ycal^N}(P^{F,s,N}(Q_\beta))=P^N \right\}
\end{align}
is equal to
\begin{align}
    \left\{Q_\beta \in \Qcal_\beta \colon
    \int_\Bcal s(\beta,a,y)dQ_\beta(\beta) = Q^{np,N}_{Y\mid a}(P^N)(a,y),
    \forall a \in \{a^N_1,\ldots,a^N_N\}, y\in a
    \right\}.
\end{align}
Substituting this equality of sets in \eqref{eqaution:rkeljrklej} gives that $\Psi^{+,N}(P^N)$ is equal to the value of the linear program
\begin{align}
\sup \left\{\int_\Bcal g(\beta)dQ_\beta(\beta) \colon Q_\beta \in \Qcal_\beta,
    \int_\Bcal s(\beta,a,y)dQ_\beta(\beta) = Q^{np,N}_{Y\mid a}(P^N)(a,y),
    \text{ for any } (a,y) \in \pairay^N
    \right\} .
\end{align}
Using \Cref{lemma:redundant_constraints_realizable} to solely keep the nonredundant constraints and indexing over $[\tilde{L}^N]$, we can rewrite $\Psi^{+,N}(P^N)$ as the value of the linear program
\begin{align}
    \sup_{Q_\beta \in \Qcal_\beta} \int_\Bcal g(\beta) d Q_\beta(\beta) \text{ subject to } \int_\Bcal t^N_l(\beta) dQ_\beta(\beta) = b^N_l(P^N) \text{ for any } l \in [\tilde{L}^N] .
\end{align}
\end{proof}

\subsection{Proof of Proposition \ref{proposition:strong_duality}}

Define
\begin{equation}\label{equation:defScalTcal}
    \begin{aligned}
        \Tcal & = \left\{\int_{\beta \in \Bcal} t^N(\beta)dQ_\beta(\beta) \colon Q_\beta \in \Qcal_\beta\right\} , \\
        \Scal & = \left\{\left(\int_{\beta \in \Bcal}  t^N(\beta)^\top dQ_\beta(\beta), \int_{\beta \in \Bcal}  g(\beta)d Q_\beta(\beta)\right)^\top \colon Q_\beta \in \Qcal_\beta\right\} .
    \end{aligned}
\end{equation}
For any subset $\Xcal$ of a Euclidean space, we denote $\Bar{\Xcal}$ the closure of $\Xcal$.

\begin{lemma}\label{lemma:Scalboundedconvex}
    $\Scal$ is a bounded convex set.
\end{lemma}

\begin{proof}[Proof of \Cref{lemma:Scalboundedconvex}]
    Since the mappings $t^N$ and $g$ are bounded, the set 
    \begin{align}
        \left\lbrace(t^{N^\top}(\beta), g(\beta))^\top \colon \beta \in \Bcal \right\rbrace \subset\RR^{\tilde{L}^N+1}
    \end{align}
    is bounded. Let $(u_1^\top,v_1)^\top \in \Scal, (u_2^\top,v_2)^\top \in \Scal$ given by
    \begin{align}
        & (u_1^\top,v_1)^\top = \left(\int t^{N^\top}(\beta)^\top dQ_{\beta,1}(\beta), \int g(\beta)dQ_{\beta,1}(\beta)\right)^\top, \\
        \text{and } & (u_2^\top,v_2)^\top = \left(\int t^{N^\top}(\beta)^\top dQ_{\beta,2}(\beta), \int g(\beta)dQ_{\beta,2}(\beta) \right)^\top .
    \end{align}
    Let $\theta \in (0,1)$. We then have that
    \begin{align}
         \int t^{N^\top}(\beta) (\theta dQ_{\beta,1}(\beta) + (1-\theta)dQ_{\beta,2}(\beta)) &= \theta \int t^{N^\top}(\beta) dQ_{\beta,1}(\beta)  + (1-\theta) \int t^{N^\top}(\beta) dQ_{\beta,2}(\beta) \\
        & = \theta u_1 + (1-\theta) u_2, \\
        \text{and  }
        \int g(\beta) (\theta dQ_{\beta,1}(\beta) + (1-\theta)dQ_{\beta,2}(\beta)) &= \theta \int g(\beta) dQ_{\beta,1}(\beta)  + (1-\theta) \int g(\beta) dQ_{\beta,2}(\beta)\\
        &= \theta v_1 + (1-\theta) v_2,
    \end{align}
    which proves that $\theta (u_1^\top,v_1)^\top + (1-\theta) (u_2^\top,v_2)^\top \in \Scal$ since $\theta dQ_{\beta,1}(\beta) + (1-\theta)dQ_{\beta,2}(\beta) \in \Pcal(\Bcal)$, which is the desired claim.
\end{proof}

\begin{lemma}\label{lemma:b_is_interior}
Let $P^N \in \Mcal^{np,N}$ and suppose that $P^N$ is nondegenerate in the sense of \Cref{definition:nondegeneracy}, then $b^N(P^N)$ lies in the interior of $\Tcal$.
\end{lemma}

\begin{proof}[Proof of \Cref{lemma:b_is_interior}]
Let $\delta b \in \RR^{\tilde{L}^N } \setminus \{0\}$ and $Q \colon \RR\rightarrow \RR^{\tilde{L}^N+1}$ be given for every $x\in \RR$ by 
\begin{align}
    Q(x) = M(Q_\beta)^{-1}\left(\begin{bmatrix}
b^N(P^N) \\1 \end{bmatrix}+x \begin{bmatrix}
\delta b\\ 0\end{bmatrix} \right) ,
\end{align}
where $M(Q_\beta)$ is as in \Cref{definition:nondegeneracy}. Since $P^N$ is nondegenerate, $M(Q_\beta)$ is invertible and
$Q(0)=(Q_l \colon l \in \{1,\dots, \tilde{L}^N+1\})>0$ componentwise, where $Q_l$'s are as in \Cref{definition:nondegeneracy}. By continuity of $Q$, there exists $x_0>0$ such that $Q(x)>0$ componentwise for any $x\in [0,x_0]$. Let $x$ be an arbitrary element of $(0,x_0)$. By definition of $M(Q_\beta)$, the last row of $M(Q_\beta)$ is a vector of ones, and therefore taking the last component of $M(Q_\beta) Q(x)$ gives that
\begin{align}
\sum_{l=1}^{\tilde{L}^N+1} Q_l(x)&=[M(Q_\beta) Q(x)]_{\tilde{L}^N+1} \\
& = \left[M(Q_\beta) M(Q_\beta)^{-1} \left(\begin{bmatrix}
b^N(P^N) \\ 1 \end{bmatrix}+x \begin{bmatrix}
\delta b \\ 0\end{bmatrix}\right)\right ]_{\tilde{L}^N+1}\\
& = \begin{bmatrix}
    b^N(P^N)+x\delta b\\ 1
\end{bmatrix}_{\tilde{L}^N+1} \\
& = 1 .
\end{align}
By definition, of $M(Q_\beta)$, for any $l = 1, \ldots, \tilde{L}^N$, the $l$-th row of $M(Q_\beta)$ is $(t^N_l(\beta_k) \colon k = 1, \ldots, \tilde{L}^N+1)$, and thus, 
\begin{align}
    b^N_l(P^N) + x \delta b_l &= \left[M(Q_\beta) M(Q_\beta)^{-1} \left(\begin{bmatrix}
b^N(P^N) \\1 \end{bmatrix}+x \begin{bmatrix}
\delta b\\ 0\end{bmatrix} \right) \right]_l\\
&= \left[M(Q_\beta) Q(x)\right]_l \\
& = \sum_{k=1}^{\tilde{L}^N+ 1} Q_k(x) t^N_l(\beta_k) .
\end{align}
Therefore, $b^N(P^N)+x\delta b$ is a convex combination of the vectors $t(\beta_1),\dots, t(\beta_{\tilde{L}^N+1})$, which implies that $b^N(P^N)+x\delta b \in \Tcal$.

It remains to make this argument uniform over directions. Define
\begin{align}
Q(x,u) = Q(0)+xM(Q_\beta)^{-1} \begin{bmatrix} u\\ 0 \end{bmatrix}.
\end{align}
Since $Q(0)>0$ componentwise and the unit sphere is compact, there exists
$r>0$ such that $Q(x,u)>0$
componentwise for any \(\|u\|=1\) and any $x\in[0,r]$. Thus, for any
$h$ such that $\|h\|\leq r$, choosing $u = h / \|h\|, x=\|h\|$ for $h\ne 0$ gives that $b^N(P^N)+h\in\Tcal$, hence the ball centered in $b^N(P^N)$ of radius $r$ belongs to $\Tcal$, which proves that $b^N(P^N)$ belongs to the interior of $\Tcal$.
\end{proof}

\begin{lemma}\label{lemma:b1b2existencedifferentvalues}
Suppose there exists $b_1$ such that there exist $\psi_{1,1}\neq\psi_{1,2}$ and $(b_1,\psi_{1,1}),(b_1,\psi_{1,2})\in \Scal$. Suppose that $b_0$ is in the interior of $\Tcal$, then there exist $\psi_{0,1}\neq \psi_{0,2}$ such that $(b_0,\psi_{0,1}) \in \Scal, (b_0,\psi_{0,2})\in \Scal$.
\end{lemma}

\begin{proof}[Proof of \Cref{lemma:b1b2existencedifferentvalues}]
Since $b_0$ is in the interior of $\Tcal$, there exists $\epsilon>0$ such that $b_\epsilon = b_0 + \epsilon(b_0 - b_1) \in \Tcal$. Moreover, by definition of $\Tcal$, there exists $Q_\epsilon \in \Qcal_\beta$ such that $b_\epsilon = \int d Q_\epsilon(\beta) t(\beta)$.  Let $\psi_\epsilon$ be such that $(b_\epsilon,\psi_\epsilon)\in \Scal$. Such a $\psi_\epsilon$ exists because $b_\epsilon$ is feasible.  Then  $b_0= 1/(1+\epsilon) b_\epsilon + \epsilon / (1+\epsilon) b_1$. Therefore, by convexity of $\Scal$, 
\begin{align}
    \left(b_0, \frac{1}{1+\epsilon}\psi_\epsilon + \frac{\epsilon}{ 1+\epsilon}\psi_{1,1} \right), \left( b_0, \frac{1}{1+\epsilon}\psi_\epsilon +
\frac{\epsilon}{1+\epsilon}\psi_{1,2}
\right)
\in \Scal. 
\end{align}
The second coordinate of the two couples above are distinct because $\psi_{1,1}\ne\psi_{1,2}$.
\end{proof}

\begin{lemma}\label{lemma:dualitygap_g_inspan}
If $g \not\in \Span \{t^N_1,\dots, t^N_{\tilde{L}^N},1\}$, then there exists $b \in \Tcal$, $\psi_1, \psi_2 \in [0,1]$ such that $\psi_1 \neq \psi_2$ and $(b, \psi_1), (b, \psi_2) \in \Scal$.
\end{lemma}

\begin{proof}[Proof of \Cref{lemma:dualitygap_g_inspan}] 
Suppose $g \not\in \Span \{1, t^N_1,\dots, t^N_{\tilde{L}^N}\}$. Then there exists a full rank matrix $R \in \RR^{(\tilde{L}^N + 2) \times (\tilde{L}^N+ 2)}$ of the form
\begin{align}
    R = \begin{bmatrix}
        \tilde R \\
        \bm{g}^\top
    \end{bmatrix},
    \quad \text{where} \quad \tilde R = 
    \begin{bmatrix}
     t^N_1(\beta_1) & \cdots & t^N_1(\beta_{\tilde{L}^N + 2}) \\
    \vdots & \ddots & \vdots \\
    t^N_{\tilde{L}^N}(\beta_1) & \cdots & t^N_{\tilde{L}^N}(\beta_{\tilde{L}^N + 2}) \\
    1 & \cdots & 1
    \end{bmatrix},
\end{align}
$\beta_1,\ldots,\beta_{\tilde{L}^N + 2} \in \Bcal$, and $\bm{g}^\top = [ g(\beta_1), \ldots, g(\beta_{\tilde{L}^N + 2})]$.
Since $\tilde R$ is in $\RR^{(\tilde{L}^N + 1) \times (\tilde{L}^N + 2)}$ and thus is not full column rank, there exists $c \in \RR^{\tilde{L}^N + 2}$, $c \neq 0$ such that $\tilde R c = 0$. Since $R$ is full rank, $c \neq 0$ and $\tilde R c = 0$, we must have that $\bm g^\top c \neq 0$. 

Let $\mu = \sum_{k \colon c_k > 0} c_k$. Then, since the last row on $\tilde R$ is a vector of ones, $\tilde R c = 0$ implies that $\sum_{k \colon c_k < 0} c_k = -\mu$. Let 
\begin{align}
    Q_\beta^+ = \frac{1}{\mu} \sum_{k \colon c_k > 0} c_k \delta_{\beta_k} \quad \text{and} \quad Q_\beta^- = -\frac{1}{\mu} \sum_{k \colon c_k < 0} c_k \delta_{\beta_k}.
\end{align}
Then $Q_\beta^+, Q_\beta^- \in \Pcal(\Bcal)$, $\int d Q_\beta^+ t  = \int d Q_\beta^- t$ and $\int d Q_\beta^+ g \neq \int d Q_\beta^- g$. Letting $b = \int dQ_\beta^+ t$, $\psi_1 = \int d Q_\beta^+ g$, $\psi_2 = \int d Q_\beta^- g$ yields the claim.
\end{proof}

\propositionstrongduality*

\begin{proof}[Proof of \Cref{proposition:strong_duality}]
Consider $\Scal$ and $\Tcal$ as defined in \eqref{equation:defScalTcal}. Let $P^N \in \Mcal^{np,N}$ and $\Psi^{D,N}(P^N)$ denote the value of $\textsc{Dual}(P^N)$. From weak duality, $\Psi^{D,N}(P^N) \geq \tilde{\Psi}^{+,N}(P^N)$. Therefore, it remains to show that $\Psi^{D,N}(P^N) \leq \tilde{\Psi}^{+,N}(P^N)$.

\paragraph{Step 1: Supporting hyperplane.} By \Cref{lemma:Scalboundedconvex}, $\Scal$ is convex and bounded. Hence $\overline{\Scal}$ is a nonempty compact convex subset of $\RR^{\tilde{L}^N+1}$. Since $P^N$ is nondegenerate, $\textsc{Primal}(P^N)$ admits a maximizer $Q_\beta^\star$, and therefore
\begin{align}
(b^N(P^N),\tilde{\Psi}^{+,N}(P^N))
= \left(\int_\Bcal t^N(\beta)dQ_\beta^\star(\beta),
\int_\Bcal g(\beta)dQ_\beta^\star(\beta)\right)
\in \Scal .
\end{align}
We now prove that this point lies on the boundary of $\overline{\Scal}$. Suppose the contrary. Then there exists $\varepsilon>0$ such that $\left(b^N(P^N),\tilde{\Psi}^{+,N}(P^N)+2\varepsilon\right)\in\overline{\Scal}$.

Let $(b_m,\psi_m)_{m\geq 1} \in \Scal^\NN$ be a sequence converging to
$(b^N(P^N), \tilde{\Psi}^{+,N}(P^N) + 2 \varepsilon)$. By \Cref{lemma:b_is_interior}, $b^N(P^N) \in \mathring{\Tcal}$. Hence, for any $m$ large enough, there exists $\theta_m\in(0,1)$ such that $\theta_m\to 1$ and
\begin{align}
\frac{b^N(P^N)-\theta_m b_m}{1-\theta_m} \in \Tcal .
\end{align}
By definition of $\Scal$ and $\Tcal$, for any $m$ large enough, there exists $\bar{\psi}_m$ such that $((b^N(P^N)-\theta_m b_m)/(1-\theta_m),\bar\psi_m )\in\Scal$. By convexity of $\Scal$, we have
\begin{align}
(1-\theta_m) \left(\frac{b^N(P^N)-\theta_m b_m}{1-\theta_m}, \bar{\psi}_m \right) + \theta_m (b_m, \psi_m) = \left(b^N(P^N),\theta_m\psi_m+(1-\theta_m) \bar\psi_m \right) \in \Scal .
\end{align}
Since $g$ takes values in $[0,1]$, $\bar\psi_m\in[0,1]$, and hence
\begin{align}
\theta_m\psi_m+(1-\theta_m)\bar\psi_m
\to \tilde{\Psi}^{+,N}(P^N)+2\varepsilon \text{ as } m \to \infty .
\end{align}
For any $m$ large enough, we thus have that
\begin{align}
\theta_m\psi_m+(1-\theta_m)\bar\psi_m
> \tilde{\Psi}^{+,N}(P^N),
\end{align}
which contradicts the definition of $\tilde{\Psi}^{+,N}(P^N)$. Therefore,
$(b^N(P^N),\tilde{\Psi}^{+,N}(P^N))$ lies on the boundary of $\overline{\Scal}$. The supporting hyperplane theorem \citep[2.5.2][]{boyd2004convex} applied to the closed convex set $\Bar{\Scal}$ at $(b^N(P^N), \tilde{\Psi}^{+,N} (P^N))$ gives that there exists $(\kappa,\tau) \in \RR^{\tilde{L}^N} \times \RR$, $(\kappa,\tau)\ne 0$, such that, for any $(b,\psi)\in\overline{\Scal}$,
\begin{align}\label{equation:supporting_hyperplane}
\kappa^\top b+\tau\psi \leq \kappa^\top b^N(P^N)+\tau\tilde{\Psi}^{+,N}(P^N) ,
\end{align}
and since $\Scal \subseteq \Bar{\Scal}$, \eqref{equation:supporting_hyperplane} holds for any $(b,\psi) \in \Scal$.

\paragraph{Step 2: Non-nullity of $\tau$.} Suppose that $\tau = 0$, then \eqref{equation:supporting_hyperplane} implies that $b^N(P^N)$ lies on the boundary $\partial \Tcal$ of $\Tcal$. Furthermore, no open ball centered at $b^N(P^N)$ lies fully within $\Tcal$, since for any $x>0$, $\kappa^\top (b^N(P^N)+x\kappa)>\kappa^\top b^N(P^N)$ because $\kappa\neq 0$ since $(\kappa, \tau)\neq 0$ and $\tau = 0$, therefore $b^N(P^N)\in \partial \Tcal$.

However, \Cref{lemma:b_is_interior} guarantees that if $P^N$ is nondegenerate in the sense of \Cref{definition:nondegeneracy}, $b^N(P^N)$ must lie in the interior of $\Tcal$, which is a contradiction. Therefore, we must have $\tau\neq 0$.

\paragraph{Step 3: Positivity of $\tau$.} By \Cref{assumption:g_not_in_span}, via \Cref{lemma:dualitygap_g_inspan}, there exists $b_1 \in \Tcal$ such that $(b_1,\psi_1), (b_1,\psi_2) \in \Scal$. Since $b^N(P^N)$ is in the interior of $\Tcal$, via \Cref{lemma:b1b2existencedifferentvalues}, there exist $\psi_{0,1}\ne \psi_{0,2}$ such that $(b^N(P^N), \psi_{0,1}), (b^N(P^N), \psi_{0,2}) \in \Scal$. Therefore, since $\psi_{0,1}$, and $\psi_{0,2}$ are distinct and such that $(b^N(P^N), \psi_{0,1}), (b^N(P^N), \psi_{0,2}) \in \Scal$, and since $\tilde{\Psi}^{+,N}(P^N) = \sup \{ \psi \in [0,1] \colon (b^N(P^N), \psi) \in \Scal\}$, at least one of the two scalar quantities  $\psi_{0,1}$ and $\psi_{0,2}$ must be strictly smaller than $\tilde{\Psi}^{+,N}(P^N)$. Without loss of generality, suppose it is $\psi_{0,1}$ that is strictly smaller than $\tilde{\Psi}^{+,N}(P^N)$. Therefore, from the supporting hyperplane inequality \eqref{equation:supporting_hyperplane}, $\tau(\psi_{0,1} - \tilde{\Psi}^{+,N}(P^N)) \leq 0$. Since $\psi_{0,1} - \tilde{\Psi}^{+,N}(P^N) < 0$, we must have $\tau \geq 0$. Since we have already proven non-nullity of $\tau$ in step 2, we thus have that $\tau > 0$.

\paragraph{Step 4: Strong duality and existence of a dual minimizer.} By \cref{equation:supporting_hyperplane}, we have that for any $(b, \psi) \in \Scal$
\begin{align}
    \kappa^\top b + \tau \psi \leq \kappa^\top b^N(P^N) + \tau \tilde{\Psi}^{+,N}(P^N) .
\end{align}
Let $\beta \in \Bcal$. Since $(t(\beta),g(\beta)) \in \Scal$, we have that
\begin{align}
    \kappa^\top t(\beta) + \tau g(\beta) \leq \kappa^\top b^N(P^N) + \tau \tilde{\Psi}^{+,N}(P^N) .
\end{align}
Dividing the above display by $\tau$, of which we have proven the positivity in step 3, and then rearranging yields that
\begin{align}
    g(\beta) \leq \lambda^\top t(\beta) + \nu ,
\end{align} 
with $\lambda(P^N) = -\kappa/\tau, \nu(P^N) = \tilde{\Psi}^{+,N}(P^N) + \kappa^\top b^N(P^N)/\tau$.
Therefore, $(\lambda(P^N), \nu(P^N))$ is feasible for $\textsc{Dual}(P^N)$ and we have
\begin{align}
    \lambda^\top b^N(P^N) + \nu = -\kappa^\top b^N(P^N) / \tau + \tilde{\Psi}^{+,N}(P^N) + \kappa^\top b^N(P^N) / \tau = \tilde{\Psi}^{+,N}(P^N) .
\end{align}
The reverse direction $\Psi^{D,N}(P^N)\geq \tilde{\Psi}^{+,N}(P^N)$ holds by weak duality. Therefore $\Psi^{D,N}(P^N) = \tilde{\Psi}^{+,N}(P^N)$ and strong duality holds. 

\paragraph{Step 5: Uniqueness of the dual optimizer.} For any $\lambda \in \RR^{\tilde{L}^N}$, $\nu \in \RR$, let
\begin{align}
    \Delta_{\lambda, \nu} \colon \Bcal \to \RR \; , \; \beta \mapsto \lambda^\top t(\beta) + \nu - g(\beta) .
\end{align}
Let $(\lambda, \nu)$ be feasible for $\Psi^{D,N}(P^N)$, that is, for any $\beta \in \Bcal, \Delta_{\lambda, \nu}(\beta) \geq 0$. Since $P^N$ is nondegenerate in the sense of \Cref{definition:nondegeneracy}, there exists $Q$ a solution of the primal problem $\textsc{Primal}(P^N)$ of the form $Q = \sum_{k=1}^{\tilde{L}^N + 1} Q_k \delta_{\beta_k}$ such that $Q_k > 0$ for every $k=1,\ldots,\tilde{L}^N+ 1$. Let $Q^\star$ be such. From primal feasibility, 
\begin{align}
    \int t(\beta) dQ^\star(\beta) = b^N(P^N) ,
\end{align}
and, from dual feasibility and integration against a non-negative measure, 
\begin{align}\label{equation:step5_first_refd_disp}
    \int g(\beta) dQ^\star(\beta) \leq \int (\lambda^\top t(\beta) + \nu)dQ^\star(\beta) = \lambda^\top b^N(P^N) + \nu .
\end{align}
Step 4 proves the existence of at least one dual minimizer. Suppose $(\lambda^\star_i, \nu^\star_i), i=1,2$ are two distinct dual minimizers. Specializing \eqref{equation:step5_first_refd_disp} to $(\lambda^\star_i, \nu^\star_i), i=1,2$ yields 
\begin{align}
    {\lambda_i^\star}^\top b^N(P^N) + \nu_i^\star  = \Psi^{D,N}(P^N) = \tilde{\Psi}^{+,N}(P^N) = \int g(\beta) dQ^\star(\beta),
\end{align}
where the first equality in the above display follows from dual-optimality, the second from strong duality, which we prove in step 4, and the last one from primal optimality of $Q^\star$.
Therefore, we have
\begin{align}
    0 = \int ({\lambda^\star_i}^\top t(\beta) + \nu^\star_i -g(\beta)) dQ^\star(\beta) = \sum_{k=1}^{\tilde{L}^N+ 1} Q^\star_k \Delta_{\lambda^\star_i, \nu^\star_i}(\beta_k) .
\end{align}
Since $Q^\star_k >0$ and $\Delta_{\lambda^\star_i, \nu^\star_i}(\beta_k)\geq 0$ for any $k =1, \ldots, \tilde{L}^N+ 1$, we have that $\Delta_{\lambda^\star_i, \nu^\star_i}(\beta_k) = 0$ for every $i=1,2$, $k=1,\ldots, \tilde{L}^N + 1$. Therefore, for any $k \in \{1, \ldots, \tilde{L}^N+1\}$, we have that
\begin{align}
    {\lambda_i^\star}^\top t(\beta_k) + \nu_i^\star = g(\beta_k) , \text{ for } i=1,2 ,
\end{align}
which can be rewritten as
\begin{align}
    M(Q^\star)^\top \begin{bmatrix}
    \lambda^\star_i \\ \nu^\star_i
\end{bmatrix} = \begin{bmatrix}
    g(\beta_1) \\ \vdots \\ g(\beta_{\tilde{L}^N+1})
\end{bmatrix} , \text{ for } i = 1,2,
\end{align}
and therefore
\begin{align}
    M(Q^\star)^\top \begin{bmatrix}
    \lambda^\star_2 - \lambda^\star_1 \\ \nu^\star_2 - \nu^\star_1
\end{bmatrix} = 0.
\end{align}
Since $M(Q^\star)$ is invertible, the two dual optimizers must be equal, which proves the uniqueness claim of this step.
\end{proof}

\subsection{Proof of \Cref{theorem:piecewise_linearity}}

\begin{definition}[Continuity over $\Mcal^{np,N}$]
    For any natural integer $d \geq 1$, say that a functional $f \colon \Mcal^{np,N} \to \RR^d$ is continuous if, for any $\epsilon > 0$, there exists $\eta > 0$ such that, if $P_1, P_2 \in \Mcal^{np,N}$ are such that $\|b^N(P_1) - b^N(P_2) \| \leq \eta$, then $\| f(P_1) - f(P_2) \| \leq \epsilon$, where the (slightly overloaded) notation $\| \cdot \|$ denotes the Euclidean norms over the respective Euclidean spaces. 
\end{definition}

\theorempiecewiselinearity*

\begin{proof}[Proof of \Cref{theorem:piecewise_linearity}]
Let $P^N$ be nondegenerate in the sense of \Cref{definition:nondegeneracy}, $\textsc{Primal}(P^N)$ admits a maximizer of the form $Q^\star = \sum_{k=1}^{\tilde{L}^N+ 1} Q^\star_k \delta_{\beta^\star_k}$, where, for every $k=1,\ldots,\tilde{L}^N+ 1$, $Q^\star_k > 0$  and $\beta^\star_k \in \Bcal$. Let $M(Q^\star)$ be the matrix collecting the constraints vectors evaluated at the atoms $\beta^\star_{k}$, $k = 1, \ldots,\tilde{L}^N+ 1$, as defined in \Cref{definition:nondegeneracy}. $M(Q^\star)$ is invertible by \Cref{definition:nondegeneracy}.

\paragraph{Step 1: Atoms at $P^N$ can realize $b^N(P)$ for $P$ in a neighborhood of $P^N$.} From \Cref{definition:nondegeneracy} at $P^N$ again, \Cref{lemma:b_is_interior} guarantees that $b^N(P^N)$ is in the interior of $\Tcal$. Therefore, from invertibility of $M(Q^\star)$, and continuity of $P \mapsto b^N(P)$, there exists a neighborhood $\Vcal \subseteq \Mcal^{np,N}$ of $P^N$ such that for any $P \in \Vcal$, there exists $Q = (Q_k \colon k = 1, \ldots, \tilde{L}^N+ 1)$ in the interior of the $\tilde{L}^N+ 1$-simplex (that is, for every $k=1,\ldots,\tilde{L}^N+ 1$, $Q_k > 0$) such that, for every $P \in \Vcal$, $b^N(P) = \sum_{k=1}^{\tilde{L}^N+1} Q_k t(\beta_k^\star)$.

\paragraph{Step 2: No dual feasibility slack at the atoms.} From the strong duality guarantee of \Cref{proposition:strong_duality}, there exists $(\lambda^\star_0, \nu^\star_0)$ feasible for $\textsc{Dual}(P^N)$ such that 
\begin{align}\label{equation:lambda0nu0star_realize_primal_value_at_Q0}
    \tilde{\Psi}^{+,N}(P^N) = {\lambda^\star_0}^\top b^N(P^N) + \nu^\star_0 . 
\end{align}
From optimality of $Q^\star$ for $\textsc{Primal}(P^N)$, $\tilde{\Psi}^{+,N}(P^N) = \sum_{k=1}^{\tilde{L}^N+ 1} Q^\star_k g({\beta_k^\star})$. From feasibility of $Q^\star$ for $\textsc{Primal}(P^N)$, $b^N(P^N) = \sum_{k=1}^{\tilde{L}^N+ 1} Q^\star_k t(\beta^\star_k)$. Combining the three equalities stated in this paragraph so far yields that 
\begin{align}
    \sum_{k=1}^{\tilde{L}^N+ 1} Q^\star_k \left(g(\beta^\star_k) - {\lambda^\star_0}^\top t(\beta^\star_k) - \nu^\star_0 \right) = 0 .
\end{align}
Since $(\lambda^\star_0, \nu^\star_0)$ is feasible for $\textsc{Dual}(P^N)$, $g(\beta^\star_k) - {\lambda^\star_0}^\top t(\beta^\star_k) - \nu^\star_0 \leq 0$ for every $k=1,\ldots, \tilde{L}^N+ 1$, and since $Q^\star_k> 0$ for every $k=1,\ldots,\tilde{L}^N+ 1$, we must have that $g(\beta^\star_k) = {\lambda^\star_0}^\top t(\beta^\star_k) + \nu^\star_0$ for every $k = 1,\ldots ,\tilde{L}^N+ 1$.

\paragraph{Step 3: Strong duality at $P$.} For every $P' \in \Mcal^{np,N}$, $\textsc{Primal}(P')$-feasible $Q$, $\textsc{Dual}(P')$-feasible $(\lambda, \nu)$, let $f^{\textsc{Primal}}_{P'}(Q)$ and $f^{\textsc{Dual}}_{P'}(\lambda, \nu)$ be the value of $\textsc{Primal}(P')$'s objective at $Q$ and the value of the $\textsc{Dual}(P')$'s objective at $(\lambda, \nu)$. 

Let $P \in \Vcal \subseteq \Mcal^{np,N}$, where $\Vcal$ is as in step 1. From step 1, there exists $Q$ feasible for $\textsc{Primal}(P)$ of the form $Q = \sum_{k=1}^{\tilde{L}^N+ 1} Q_k \delta_{\beta^\star_k}$, with $(Q_k)_{k \in [\tilde{L}^N+1]}$ in the interior of the $\tilde{L}^N+ 1$-simplex. We have that
\begin{align}
    f_P^{\textsc{Primal}}(Q) &= \sum_{k=1}^{\tilde{L}^N+ 1} Q_k g(\beta^\star_k) \\
    &= \sum_{k=1}^{\tilde{L}^N+ 1} Q_k \left( {\lambda_0^\star}^\top t(\beta^\star_k) + \nu^\star_0 \right) \\
    & = {\lambda_0^\star}^\top b^N(P) +  \nu^\star_0 \\
    & = f_P^{\textsc{Dual}}(\lambda_0^\star, \nu^\star_0).
\end{align}
where the first equality follows from the definition of $Q$, the second one follows from the conclusion of step 2, the third one from $\textsc{Primal}(P)$-feasibility of $Q$, and the last equality follows from the definition of $f_P^{\textsc{Dual}}$.
From, in this order, $\textsc{Primal}(P)$-feasibility of $Q$, weak duality, and $\textsc{Dual}(P)$-feasibility of $(\lambda_0^\star, \nu_0^\star)$, we have
\begin{align}
    f_P^{\textsc{Primal}}(Q) \leq \tilde{\Psi}^{+,N}(P) \leq \Psi^{D,N}(P) \leq f_P^{\textsc{Dual}}(\lambda_0^\star, \nu_0^\star).
\end{align}
Since the leftmost and rightmost quantities are equal from the before-last display,
\begin{align}\label{equation:lambda0nu0_strong_duality_at_Q}
    \tilde{\Psi}^{+,N}(P) =  \Psi^{D,N}(P) = {\lambda_0^\star}^\top b^N(P) + \nu_0^\star ,
\end{align}
that is, strong duality holds at $P$ and $(\lambda_0^\star, \nu_0^\star)$ is $\textsc{Dual}(P)$-optimal. It proves that $\lambda(P) = \lambda(P^N)$.

\paragraph{Step 4: uniqueness of the dual optimizer at $P$.} Uniqueness of dual optimizer at $P$ directly follows from step 5 in the proof of \Cref{proposition:strong_duality}.

\paragraph{Step 5: von Mises expansion.} Taking the difference of \eqref{equation:lambda0nu0star_realize_primal_value_at_Q0} and \eqref{equation:lambda0nu0_strong_duality_at_Q}, and using the fact that $\lambda_0^\star = \lambda(P)$ from step 4, yields
\begin{align}
    \tilde{\Psi}^{+,N}(P) - \tilde{\Psi}^{+,N}(P^N)= \lambda(P)^\top\{b^N(P)-b^N(P^N)\} ,
\end{align}
hence the result.
\end{proof}

\section{Proofs of the inference results}

Suppose that \Cref{assumption:g_not_in_span} holds. Let $N$ be a positive integer, $P^N \in \Mcal^{np,N}$ is nondegenerate in the sense of \Cref{definition:nondegeneracy}, and let $i\in[N]$. By \Cref{theorem:piecewise_linearity}, we have that the solution $(\lambda(P^N), \nu(P^N))$ exists and is unique. Consider $\RR$, and $\Bcal(\RR)$, the Borel $\sigma$-algebra over $\RR$. For any $A \in \Bcal(\RR)$, we write, for any $x \in \RR$, $\1_A\{x\}$ as the indicator function of $A$ at $x$. With a slight abuse of notation, for any $\{y\} \in \Bcal(\RR)$ a singleton, we use $\1\{x=y\}$ instead of $\1_{\{y\}}\{x\}$, for any $a \in \RR$, we use $\1\{a \leq x\}$ instead of $\1_{[a,+\infty)}\{x\}$, and respectively $\1\{a < x\}$, $\1\{a \geq x\}$, and $\1\{a > x\}$ instead of $\1_{(a,+\infty)}\{x\}$, $\1_{(-\infty, a]}\{x\}$, and $\1_{(-\infty, a)}\{x\}$.

Let $(Y_1^N,\ldots,Y_N^N) \sim P^N$, and let
\begin{align}\label{equation:ewlrjkwkle}
    X^N_i & = \frac{1}{\sqrt N} \frac{1}{e^N(a_i^N)} \sum_{l \colon a(l)=a_i^N}\lambda_l(P^N)\left\{\1 \{Y_i^N=y(l)\}-b_l^N(P^N)\right\} .
\end{align}
The Lindeberg condition for the sequence of random variables $\{X^N_i \colon i\in[N]\}_{N \geq 1}$ can be written, for $\epsilon>0$, as
\begin{align}
    \frac{1}{\sum_{i=1}^N \Var_{P^N}(X_i^N)} \sum_{i=1}^N \EE_{P^N} \left[ (X^N_i)^2 \1 \left\{|X^N_i|\geq \epsilon \sqrt{\sum_{i=1}^N \Var_{P^N}(X_i^N)}\right\} \right] \longrightarrow 0 \text{ as } N \to \infty .
\end{align}

\begin{lemma}\label{lemma:sum_variances}
Suppose that \Cref{assumption:g_not_in_span} holds. Let $N$ be a positive integer, $P^N\in\Mcal^{np,N}$ such that $P^N$ is realizable under $(\Qcal_\beta,s)$, nondegenerate in the sense of \Cref{definition:nondegeneracy}, and $(X^N_i)_{i \in [N]}$ defined as in \eqref{equation:ewlrjkwkle}. Then
\begin{align}
    \EE_{P^N}[X^N_i] & = 0 , \\
    \sum_{i=1}^N \Var_{P^N} (X_i^N) & = \sigma_N^2(P^N) .
\end{align}
\end{lemma}

\begin{proof}[Proof of \Cref{lemma:sum_variances}]
Let $i\in[N]$. For any $l \in [\tilde{L}^N]$ such that $a(l) = a^N_i$, by definition of $b^N_l$, we have that
\begin{align}
    \PP_{P^N}(Y_i^N=y(l)) = Q^{np,N}_{Y\mid a}(P^N)(a(l),y(l)) = b_l^N(P^N) .
\end{align}
Therefore,
\begin{align}
    \EE_{P^N} \left[\sum_{l \colon a(l)=a} \lambda_l(P^N) \1\{Y_i^N=y(l)\}
    \right] & \nonumber = \sum_{l \colon a(l)=a} \lambda_l(P^N) \EE_{P^N}\left[\1\{Y_i^N=y(l)\} \right] \\
    &= \nonumber \sum_{l \colon a(l)=a} \lambda_l(P^N) \PP_{P^N} \{Y_i^N=y(l)\} \\
    &= \label{equation:rjelrkj1} \sum_{l \colon a(l)=a} \lambda_l(P^N) b_l^N(P^N).
\end{align}
Hence, by the definition of $X_i^N$, we have
\begin{align}
    X_i^N & = \frac{1}{\sqrt N} \frac{1}{e^N(a_i^N)} \sum_{l \colon a(l)=a_i^N}\lambda_l(P^N)\left\{\1 \{Y_i^N=y(l)\}-b_l^N(P^N)\right\} \\
    & = \frac{1}{\sqrt N} \frac{1}{e^N(a_i^N)} \left( \sum_{l \colon a(l)=a_i^N} \lambda_l(P^N) \1 \{Y_i^N=y(l)\}-\sum_{l \colon a(l)=a_i^N} \lambda_l(P^N) b_l^N(P^N) \right) \\
    & = \frac{1}{\sqrt N e^N(a_i^N)}
    \left( \sum_{l \colon a(l)= a_i^N} \lambda_l(P^N)\1 \{Y_i^N=y(l)\} - \EE_{P^N}\left[\sum_{l \colon a(l)=a} \lambda_l(P^N)\1 \{Y_i^N=y(l)\} \right] \right) ,
\end{align}
which proves that $\EE_{P^N}[X^N_i] = 0$. It follows that
\begin{align}
    \Var_{P^N}(X_i^N) = \frac{1}{N e^N(a^N_i)^2} \Var_{P^N}\left( \sum_{l \colon a(l) = a_i^N} \lambda_l(P^N)\1 \{Y_i^N=y(l)\} \right).
\end{align}
We now compute the variance on the right-hand side. First,
\begin{align}
    & \label{equation:rjelrkj2} \Var_{P^N}\left( \sum_{l \colon a(l)=a^N_i} \lambda_l(P^N)\1 \{Y_i^N=y(l) \} \right) \\
    & \nonumber \quad = \EE_{P^N}\left[ \left( \sum_{l \colon a(l)=a_i^N}\lambda_l(P^N) \1 \{Y_i^N=y(l)\} \right)^2 \right] - \EE_{P^N} \left[ \sum_{l \colon a(l) = a_i^N} \lambda_l(P^N)\1 \{Y_i^N=y(l)\} \right]^2 .
\end{align}
For fixed $a^N_i$, the events $\{Y_i^N=y(l)\}$ are mutually exclusive across $l$ such that $a(l)=a_i^N$. Thus, for $l\neq m$ such that $a(l)=a(m)=a^N_i$, we have
\begin{align}
    \1 \{Y_i^N=y(l)\}\1 \{Y_i^N=y(m)\} = 0 ,
\end{align}
which gives that
\begin{align}
    \left( \sum_{l \colon a(l)= a^N_i}\lambda_l(P^N) \1 \{Y_i^N=y(l)\}
    \right)^2 = \sum_{l \colon a(l)=a^N_i}\lambda_l(P^N)^2\1 \{ Y_i^N=y(l) \}.
\end{align}
Taking expectations gives
\begin{align}
    \EE_{P^N} \left[ \left( \sum_{l \colon a(l)=a}\lambda_l(P^N)\1 \{Y_i^N=y(l) \} \right)^2 \right]
    &= \sum_{l \colon a(l)=a^N_i}\lambda_l(P^N)^2
    \EE_{P^N} \left[\1 \{Y_i^N=y(l)\}\right] \\
    & = \sum_{l \colon a(l)=a^N_i} \lambda_l(P^N)^2
    \PP_{P^N}\{Y_i^N=y(l)\} \\
    & = \sum_{l \colon a(l)=a^N_i} b_l^N(P^N)\lambda_l(P^N)^2.
\end{align}
The expression of the second term on the right-hand side of \eqref{equation:rjelrkj2} is computed from \eqref{equation:rjelrkj1}, and we obtain that
\begin{align}
    \Var_{P^N}(X_i^N) & = \frac{1}{N e^N(a^N_i)^2} \Var_{P^N}\left( \sum_{l \colon a(l) = a_i^N} \lambda_l(P^N)\1 \{Y_i^N=y(l)\} \right) \\
    & = \frac{1}{N\{e^N(a^N_i)\}^2}\left[\sum_{l \colon a(l)=a^N_i} b_l^N(P^N)\lambda_l(P^N)^2-\left\{\sum_{l \colon a(l)=a^N_i } b_l^N(P^N) \lambda_l(P^N) \right\}^2 \right] .
\end{align}
Summing over $i \in [N]$ yields
\begin{align}
    & \sum_{i=1}^N \Var_{P^N}(X_i^N) \\
    & \quad = \sum_{i=1}^N \frac{1}{N\{e^N(a^N_i)\}^2}\left[\sum_{l \colon a(l)=a^N_i} b_l^N(P^N)\lambda_l(P^N)^2-\left\{\sum_{l \colon a(l)=a^N_i } b_l^N(P^N) \lambda_l(P^N) \right\}^2 \right] \\
    & \quad = \sum_{a\in\{a_1^N,\ldots,a_N^N\}}\sum_{i:a_i^N=a}
    \frac{1}{Ne^N(a)^2} \left[
    \sum_{l \colon a(l) = a} b_l^N (P^N) \lambda_l(P^N)^2-\{\sum_{l \colon a(l)=a} b_l^N(P^N) \lambda_l(P^N)\}^2\right] \\
    & \quad = \sum_{a\in\{a_1^N,\ldots,a_N^N\}} \frac{n_a}{N \{e^N(a) \}^2} \left[\sum_{l \colon a(l)=a} b_l^N(P^N) \lambda_l (P^N)^2 - \left\{\sum_{l \colon a(l)=a}b_l^N(P^N)\lambda_l(P^N) \right\}^2 \right] \\
    & \quad = \sum_{a\in \{ a_1^N,\ldots, a_N^N \}} \frac{1}{e^N(a)} \left[ \sum_{l \colon a(l)=a} b_l^N(P^N) \lambda_l (P^N)^2 - \left\{ \sum_{l \colon a(l)=a} b_l^N(P^N) \lambda_l(P^N) \right\}^2 \right] \\
    & \quad = \sigma_N^2(P^N) ,
\end{align}
by definition of $\sigma_N^2(P^N)$.
\end{proof}

\begin{lemma}
\label{lemma:lindeberg_condition_satisfied}
Suppose that \Cref{assumption:g_not_in_span} holds. Let $(P^N)_{N \geq 1}$ be a sequence such that, for any positive integer $N$, $P^N \in \Mcal^{np,N}$ is realizable under $(\Qcal_\beta,s)$ in the sense of \Cref{definition:correct_fit}, nondegenerate in the sense of \Cref{definition:nondegeneracy}, and that
\Cref{assumption:stabilizing_design}, \Cref{assumption:Pconverges}, and \Cref{assumption:nondegenerate_limiting_variance} hold. Then the sequence of random variables $\{X^N_i \colon i \in[N]\}_{N \geq 1}$ satisfies the Lindeberg condition: for any $\epsilon>0$,
\begin{align}
    \frac{1}{\sum_{i=1}^N \Var_{P^N}(X_i^N)} \sum_{i=1}^N \EE_{P^N} \left[ (X^N_i)^2
        \1 \left\{|X^N_i|\geq \epsilon \sqrt{\sum_{i=1}^N \Var_{P^N}(X_i^N)}\right\} \right] \longrightarrow 0 \text{ as } N \to \infty.
\end{align}
\end{lemma}

\begin{proof}[Proof of \Cref{lemma:lindeberg_condition_satisfied}]
Let $\epsilon >0$. From \Cref{lemma:sum_variances}, we have that for any positive integer $N, P^N \in \Mcal^{np,N}$
    \begin{align}
        & \sum_{i=1}^N \Var_{P^N}(X_i^N) \\
        & \quad = \sum_{a\in\{a_1,\ldots,a_M\}} \frac{1}{e^N(a)}\left[ \sum_{l \colon a(l)=a} b_l^N(P^N)\lambda_l(P^N)^2-\left\{\sum_{l \colon a(l) = a} b_l^N(P^N)\lambda_l(P^N) \right\}^2 \right] .
    \end{align}
By \Cref{assumption:Pconverges}, there exists a positive integer $N_0$ such that for any $N \geq N_0$, we have that $b^N(P^N) = b, \lambda(P^N) = \lambda$. Therefore, for any $N \geq N_0$
\begin{align}
        \sum_{i=1}^N \Var_{P^N}(X_i^N) & = \sum_{a\in\{a^N_1,\ldots,a^N_N\}} \frac{1}{e^N(a)}\left[ \sum_{l \colon a(l)=a} b_l \lambda_l^2-\left\{\sum_{l \colon a(l)=a} b_l\lambda_l \right\}^2 \right] ,
    \end{align}
and by \Cref{assumption:stabilizing_design}, we have the existence of $a_1,\ldots,a_M$ such that for any $N$ large enough, $\{a^N_1, \ldots, a^N_N\} = \{a_1,\ldots,a_M\}$ and $e^N(a) \to e(a)>0 \text{ as } N \to \infty$ for any $a\in\{a_1,\ldots,a_M\}$. Consequently, we have that
\begin{align}\label{equation:rjerioeriu}
    \sum_{i=1}^N \Var_{P^N}(X_i^N) \to \sum_{a\in\{a_1,\ldots,a_M\}}
    \frac{1}{e(a)}
    \left[\sum_{l \colon a(l)=a} b_l\lambda_l^2 - \left\{\sum_{l \colon a(l)=a} b_l\lambda_l \right\}^2 \right] \text{ as } N \to \infty ,
\end{align}
and thus \Cref{assumption:nondegenerate_limiting_variance} ensures that
\begin{align}
    \sum_{a\in\{a_1,\ldots,a_M\}}
    \frac{1}{e(a)}
    \left[\sum_{l \colon a(l)=a} b_l\lambda_l^2 - \left\{\sum_{l \colon a(l)=a} b_l\lambda_l \right\}^2 \right]  > 0 .
\end{align}
Therefore, there exists $M_0 >0$ such that for any $N$ large enough, we have
\begin{align}
    \sum_{i=1}^N \Var_{P^N}(X_i^N) \geq M_0 >0 .
\end{align}
Moreover, by \Cref{assumption:stabilizing_design}, there exists $c>0$ such that,
for any $N$ large enough and any $a\in\{a_1,\ldots,a_M\}, e^N(a)\geq c$, and $\Card(\pairay^N) = L$ is fixed, and $\Card\{l \colon a(l)=a_i^N\}$ too, for any $N \geq 1, i \in [N]$. By \Cref{assumption:Pconverges}, for any $l \in [\tilde{L}]$, the sequence
$\{\lambda_l(P^N)\}_N$ is bounded. Therefore, there exists $C<\infty$ such that, for any $N$ large enough and any $i\in[N]$,
\begin{align}
    |X^N_i| \leq \frac{1}{\sqrt N}
    \frac{1}{e^N(a_i^N)} \sum_{l \colon a(l)=a_i^N} |\lambda_l(P^N)| \underbrace{\left|\1\{Y_i^N=y(l)\} - b_l^N(P^N) \right|}_{\leq 1} \leq \frac{C}{\sqrt{N}} .
\end{align}
For $N$ large enough, we have that $C / \sqrt{N} < \epsilon \sqrt{M_0}$, and consequently for any $i \in [N]$
\begin{align}
    |X^N_i| \leq \frac{C}{\sqrt{N}} < \epsilon \sqrt{M_0} \leq \epsilon \sqrt{\sum_{i=1}^N \Var_{P^N}(X_i^N)} .
\end{align}
and in that case, we have that
\begin{align}
    \1 \left\{|X^N_i|\geq \epsilon \sqrt{\sum_{i=1}^N \Var_{P^N}(X_i^N)}\right\} = 0.
\end{align}
We can thus conclude that
\begin{align}
    \frac{1}{\sum_{i=1}^N \Var_{P^N}(X_i^N)} \sum_{i=1}^N \EE_{P^N}
    \left[ (X^N_i)^2 \1\left\{|X^N_i|\geq \epsilon \sqrt{\sum_{i=1}^N \Var_{P^N}(X_i^N)}\right\} \right] \longrightarrow 0 \text{ as } N \to \infty ,
\end{align}
which proves the Lindeberg condition.
\end{proof}

\begin{lemma}\label{lemma:bN_empirical_consistency}
Suppose that \Cref{assumption:stabilizing_design}, and \Cref{assumption:Pconverges} hold.
Let $P^N \in \Mcal^{np,N}$ such that $P^N$ is realizable under $(\Qcal_\beta, s)$ in the sense of \Cref{definition:correct_fit}. Let $(Y^N_1,\ldots, Y^N_N)\sim P^N$ and $P_N = \delta_{(Y^N_1, \ldots, Y^N_N)}$. Then
\begin{align}
    b^N(P_N)-b^N(P^N)=o_p(1).
\end{align}
\end{lemma}

\begin{proof}[Proof of \Cref{lemma:bN_empirical_consistency}]
Let $l\in\{1,\ldots,\tilde{L}^N\}$. By definition,
\begin{align}
b_l^N(P_N) = \frac{1}{n_{a(l)}}\sum_{i \colon a^N_i = a(l)}\1\{Y^N_i=y(l)\}.
\end{align}
Since $P^N$ is realizable under $(\Qcal_\beta, s)$, there exists $Q_\beta^\star$ such that
$P^N=M_{\Ycal^N}(P^{F,s,N}(Q_\beta^\star))$. Hence, for every $i \in [N]$ such that $a^N_i = a$,
\begin{align}
P^N(y^N_i=y)
= \int_\Bcal s(\beta,a,y)dQ_\beta^\star(\beta) = b_l^N(P^N) ,
\end{align}
hence $\EE_{P^N}[b_l^N (P_N)] = b_l^N(P^N)$. 

Moreover, by independence of $Y^N_1, \ldots, Y^N_N$ under
$M_{\Ycal^N}(P^{F,s,N}(Q_\beta^\star))$, we have
\begin{align}
\Var_{P^N}\{b_l^N(P_N)\}
= \frac{1}{n_a^2}\sum_{i \colon a^N_i = a(l)}
\Var_{P^N}(\1\{y^N_i=y\}) \leq \frac{1}{4n_a} \to 0 \text{ as } N \to \infty ,
\end{align}
since $n_a \to \infty$ as $N \to \infty$ by \Cref{assumption:stabilizing_design}. Thus $b_l^N(P_N)-b_l^N(P^N)=o_p(1)$ by Chebyshev's inequality. Since
$\tilde{L}^N = \tilde{L} <\infty$ for any $N$ large enough, the componentwise result implies
$\|b^N(P_N)-b^N(P^N)\| = o_p(1)$.
\end{proof}

\begin{lemma}\label{lemma:ratio_variances}
Suppose that \Cref{assumption:g_not_in_span},
\Cref{assumption:stabilizing_design},
\Cref{assumption:Pconverges}, \Cref{assumption:stability_dual_N_grows}, and \Cref{assumption:nondegenerate_limiting_variance} hold. Let
$(P^N)_{N\geq 1}$ be a sequence such that, for each $N$,
$P^N\in\Mcal^{np,N}$ is realizable under $(\Qcal_\beta,s)$ in the sense of
\Cref{definition:correct_fit}, and is nondegenerate in the sense of
\Cref{definition:nondegeneracy}. Then $\sigma_N(P_N)$ is well-defined with
probability tending to one and
\begin{align}
    \frac{\sigma_N(P_N)}{\sigma_N(P^N)}
    \xrightarrow{p} 1
    \quad \text{as } N\to\infty .
\end{align}
\end{lemma}

\begin{proof}[Proof of \Cref{lemma:ratio_variances}]
To ensure that all the objects are defined, for any $\bar{P}^N \in \Mcal^{np,N}$ if $\textsc{Dual}(\bar P^N)$ does not admit a unique minimizer, we set $\sigma_N^2(\bar P^N)=1$. From \Cref{lemma:sum_variances} and \eqref{equation:rjerioeriu}, we have that
\begin{align}\label{equation:dreowriuwieoru}
\sigma_N^2(P^N) 
&\to \sum_{a\in\{a_1,\ldots,a_M\}}
\frac{1}{e(a)} \left[ \sum_{l\colon a(l)=a} b_l\lambda_l^2 - \left\{ \sum_{l\colon a(l)=a} b_l\lambda_l \right\}^2 \right] >0 \text{ as } N \to \infty ,
\end{align}
From \Cref{lemma:bN_empirical_consistency}, we have that $\|b^N(P_N) - b^N(P^N)\| \to 0$ in probability as $N \to \infty$, hence for any $\eta >0$, we have $\PP_{P^N}
(\|b^N(P_N)-b^N(P^N)\|<\eta) \to 1$ as $N \to \infty$. On the event
$\|b^N(P_N)-b^N(P^N)\|<\eta$,
\Cref{assumption:stability_dual_N_grows} gives
\begin{align}
\lambda(P_N)=\lambda(P^N).
\end{align}
Thus, for $N$ large enough, with probability tending to one,
\begin{align}
\sigma_N^2(P_N)-\sigma_N^2(P^N)
&= \sum_{a\in\{a_1,\ldots,a_M\}}
\frac{1}{e^N(a)} \Bigg[ \sum_{l\colon a(l)=a}
\left\{ b_l^N(P_N)-b_l^N(P^N) \right\} \lambda_l(P^N)^2 \\
&\qquad - \left\{ \sum_{l\colon a(l)=a}
b_l^N(P_N)\lambda_l(P^N)
\right\}^2 + \left\{ \sum_{l\colon a(l)=a} b_l^N(P^N)\lambda_l(P^N) \right\}^2 \Bigg] \\
&=o_p(1),
\end{align}
because $\tilde L^N=\tilde L$ for all $N$ large enough,
$\lambda(P^N)=\lambda$, $e^N(a)$ is bounded away from zero, and
$b_l^N(P_N)-b_l^N(P^N)=o_p(1)$ for every $l\in[\tilde L]$.
Consequently,
\begin{align}
\frac{\sigma_N^2(P_N)}{\sigma_N^2(P^N)}
&= 1+ \frac{\sigma_N^2(P_N)-\sigma_N^2(P^N)}{\sigma_N^2(P^N)} \xrightarrow{p}1.
\end{align}
Since $\sigma_N(P^N)>0$ by \eqref{equation:dreowriuwieoru}, the continuous mapping theorem \citep[][Theorem 2.3.]{van2000asymptotic} yields
\begin{align}
\frac{\sigma_N(P_N)}{\sigma_N(P^N)}
= \left\{\frac{\sigma_N^2(P_N)}{\sigma_N^2(P^N)} \right\}^{1/2} \xrightarrow{p} 1 .
\end{align}
\end{proof}

\theoremasymptoticnormality*

\begin{proof}[Proof of \Cref{theorem:asymptotic_normality}]
The sequence of random variables $\{X^N_i \colon i \in [N]\}_{N \geq 1}$ satisfies the Lindeberg condition (\Cref{lemma:lindeberg_condition_satisfied}) where $\EE[X^N_i] = 0$ for any $N \geq 1, i \in [N]$ (\Cref{lemma:sum_variances}). For any $N$ sufficiently large, by realizability of $P^N$ under $(\Qcal_\beta,s)$, we have that $\{X^N_i\}_{i \in [N]}$ is a sequence of independent random variables. Therefore, by the Lindeberg central limit theorem \citep[][Theorem 27.2]{billingsley1986probability}, we have that
\begin{align}\label{equation:erjeljrlrur}
\frac{\sum_{i=1}^N X^N_i}
{\sqrt{\sum_{i=1}^N \Var_{P^N}(X^N_i)}}
\xrightarrow{d} \Ncal(0,1) .
\end{align}
Using \Cref{lemma:sum_variances} one more time, we can rewrite \eqref{equation:erjeljrlrur} as follows
\begin{align}\label{equation:kewjrlkjrk}
    \frac{\sum_{i=1}^N X^N_i}
{\sigma_N(P^N)}
\xrightarrow{d} \Ncal(0,1) .
\end{align}
Let $\eta > 0$ be as in \Cref{assumption:stability_dual_N_grows}. By \Cref{lemma:bN_empirical_consistency}, we have
\begin{align}
\|b^N(P_N) - b^N(P^N)\|=o_p(1) ,
\end{align}
and hence $\PP_{P^N}(\{\|b^N(P_N)-b^N(P^N)\|< \eta\})\to 1$ as $N \to \infty$. On the event $\{\|b^N(P_N)-b^N(P^N)\| < \eta\}$, \Cref{assumption:stability_dual_N_grows} implies that the dual minimizer $\lambda(P_N)$ is unique, that $\lambda(P_N)=\lambda(P^N)$, and that the local von Mises expansion holds. Therefore, on the event $\{\|b^N(P_N) - b^N(P^N)\|<\eta\}$, we can write
\begin{align}
\tilde{\Psi}^{+,N}(P_N)-\tilde{\Psi}^{+,N}(P^N) = \lambda(P^N)^\top (b^N(P_N)-b^N(P^N)) .
\end{align}
and this holds with probability tending to $1$, since $\PP_{P^N}(\{\|b^N(P_N) - b^N(P^N)\|<\eta\})\to 1$. Multiplying by $\sqrt N$ and using the definition of $b^N(P_N)$ yields
\begin{align}
& \sqrt N\left\{\tilde{\Psi}^{+,N}(P_N)-\tilde{\Psi}^{+,N}(P^N)\right\} \\
& \quad = \sqrt N
\sum_{l=1}^{\tilde L^N}
\lambda_l(P^N) \left\{b_l^N(P_N)-b_l^N(P^N)\right\} \\
& \quad = \sqrt N \sum_{l=1}^{\tilde L^N} \lambda_l(P^N) \left[\frac{1}{n_{a(l)}}
\sum_{i \colon a_i^N=a(l)} \left\{\1\{Y_i^N=y(l)\}-b_l^N(P^N)\right\} \right] \\
& \quad  = \sum_{l=1}^{\tilde L^N} \frac{\lambda_l(P^N)}{\sqrt N\, e^N(a(l))} \sum_{i\colon a_i^N=a(l)} \left\{\1\{Y_i^N=y(l)\}-b_l^N(P^N)\right\} \\
& \quad  = \sum_{i=1}^N
\frac{1}{\sqrt N}
\frac{1}{e^N(a_i^N)}
\sum_{l\colon a(l)=a_i^N}
\lambda_l(P^N)
\left\{\1\{Y_i^N=y(l)\}-b_l^N(P^N)\right\} \\
& \quad  = \sum_{i=1}^N X^N_i .
\end{align}
which gives with \eqref{equation:kewjrlkjrk} that
\begin{align}
    \sqrt{N} \frac{\tilde{\Psi}^{+,N}(P_N)-\tilde{\Psi}^{+,N}(P^N)}{\sigma_N(P^N)} \xrightarrow{d} \Ncal(0,1) .
\end{align}
\Cref{lemma:ratio_variances} and Slutsky's lemma thus give
\begin{align}
\sqrt{N} \frac{
\tilde{\Psi}^{+,N}(P_N)-\tilde{\Psi}^{+,N}(P^N)}{\sqrt{\sigma_N^2(P_N)}} \xrightarrow{d} \Ncal(0,1).
\end{align}
Finally, since $P^N$ is realizable under $(\Qcal_\beta,s)$,
\Cref{theorem:representation1} gives
\begin{align}
\tilde{\Psi}^{+,N}(P^N)=\Psi^{+,N}(P^N) ,
\end{align}
and therefore,
\begin{align}
\frac{ \sqrt N\left\{\tilde{\Psi}^{+,N} (P_N)-\Psi^{+,N}(P^N) \right\}}{\sqrt{\sigma_N^2(P_N)}} \xrightarrow{d} \Ncal(0,1),
\end{align}
which proves the theorem.
\end{proof}

\lemmapopulationmomentsdualfixed*

\begin{proof}[Proof of \Cref{lemma:population_moments_dual_fixed}]
By \Cref{assumption:stabilizing_design}, there exist $N_0 \geq 1$ $a_1,\ldots,a_M$ such that,
for any $N \geq N_0$ large enough, $L = L^N, \tilde{L} = \tilde{L}^N$,
\begin{align}\label{equation:eiwieiqwieu}
    \{a^N_1,\ldots,a^N_N\}=\{a_1,\ldots,a_M\} ,
\end{align}
and the ordering of $\{(a(l),y(l))\}_{l \in [L]}$ is fixed.

Under \Cref{assumption:fixed_dgp}, there exists $Q_\beta \in \Qcal_\beta$, a positive integer $N_1$ such that for any $N \geq N_1$
\begin{align}
    P^N=M_{\Ycal^N}\left(P^{F,s,N}(Q_\beta)\right) .
\end{align}
Let $\Bar{N} = \max\{N_0, N_1\}$, for any $N \geq \Bar{N}$, we have, for any $l \in [\tilde{L}]$, by definition of
$b_l^N(P^N)$
\begin{align}
    b_l^N(P^N) = Q^{np,N}_{Y\mid a}(P^N)(a(l),y(l))= \int_\Bcal s(\beta,a(l),y(l))\,dQ_\beta(\beta).
\end{align}
The right-hand side does not depend on $N$, thus we define for any $l \in [\tilde{L}]$
\begin{align}
    b_l = \int_\Bcal s(\beta,a(l),y(l))\,dQ_\beta(\beta) ,
\end{align}
and $b^N(P^N)=b$ for any $N \geq \Bar{N}$.

Let $\textsc{Dual}(b, t, g)$ be the linear program
\begin{align}\label{equation:dfjsfjk}
    \inf_{\lambda\in\RR^{\tilde L},\,\nu\in\RR}
    \lambda^\top b+\nu
    \quad
    \text{subject to}
    \quad
    \lambda^\top t(\beta)+\nu\geq g(\beta)
    \text{ for every }\beta\in\Bcal .
\end{align}
Let $N \geq \Bar{N}$, we have that $b^N(P^N)=b$, hence $\textsc{Dual}(b, t, g) = \textsc{Dual}(P^N)$. Since \Cref{assumption:g_not_in_span} holds and $P^N$ is nondegenerate in the sense of \Cref{definition:nondegeneracy}, by \Cref{proposition:strong_duality} $\textsc{Dual}(P^N)$ admits a unique solution $(\lambda(P^N), \nu(P^N)) \in \RR^{\tilde{L}} \times \RR$, which is thus also a solution of  $\textsc{Dual}(b, t, g)$ and does not depend on $N$. Writing $(\lambda, \nu) = (\lambda(P^N), \nu(P^N))$ such a minimizer, we have that \eqref{equation:fjsefrewijr} holds, hence \Cref{assumption:Pconverges} is satisfied.

Since \Cref{assumption:g_not_in_span} holds, and $P^N$ is nondegenerate in the sense of \Cref{definition:nondegeneracy}, by \Cref{theorem:piecewise_linearity}, applied to the dual problem \eqref{equation:dfjsfjk}, there exists a neighborhood $v$ of $b$ such that, for any $\Bar{P}^N \in \Mcal^{np,N}$ such that $b^N(\Bar{P}^N)\in v$, we have
\begin{align}
    \lambda(\Bar{P}^N) = \lambda(P^N) = \lambda ,
\end{align}
where $(\lambda(\Bar{P}^N),\nu(\Bar{P}^N))$ is the unique minimizer of $\textsc{Dual}(\Bar{P}^N)$, and the von Mises expansion from \Cref{theorem:piecewise_linearity} holds at $\Bar{P}^N$. Since $v$ is a neighborhood of $b$, there exists $\eta>0$ such that
\begin{align}
    \{c\in\RR^{\tilde L}\colon \|c-b\|<\eta\}\subseteq v .
\end{align}
Since $b^N(P^N)=b$, we have that if  $\|b^N(\Bar{P}^N)-b^N(P^N)\|<\eta$, then $\|b^N(\Bar{P}^N)-b\| = \|b^N(\Bar{P}^N)-b^N(P^N)\| < \eta$, and $b^N(\Bar{P}^N)\in v$. Consequently
\begin{align}
    \lambda(\Bar{P}^N)=\lambda(P^N) ,
\end{align}
which proves that \Cref{assumption:stability_dual_N_grows} holds.
\end{proof}

\section{Proofs for the representation of the plug-in as an NPMLE for a finite mixture}

\begin{lemma}\label{lemma:rewriting_risk}
    Let $P^N \in \Mcal^{np, N}$, and $Q_\beta \in \Qcal_\beta$. Then, we have that
    \begin{align}
        \risk(P^N, P^{F,s,N}(Q_\beta)) & = \sum_{a \in \{a^N_1,\ldots, a^N_N\}} n_a \KL(Q^{np,N}_{Y\mid a}(P^N)(a,\cdot)\|  \int_\Bcal s(\beta,a,\cdot)dQ_\beta(\beta)) - C^N(P^N) .
\end{align}
where $C^N(P^N) = \sum_{(a,y) \in \pairay^N} n_a  Q^{np,N}_{Y\mid a}(P^N)(a,y) \log Q^{np,N}_{Y\mid a}(P^N)(a,y)$.
\end{lemma}

\begin{proof}[Proof of \Cref{lemma:rewriting_risk}]
    For any $P^N \in \Mcal^{np,N}$, and $Q_\beta \in \Qcal_\beta$, we have that
\begin{align}
     & \risk(P^N, P^{F,s,N}(Q_\beta)) \\
     & \quad = \sum_{(y^1,\ldots,y^N) \in \Ycal^N} P^N(y^1,\ldots,y^N) \loglike(P^{F,s,N}(Q_\beta)) (y^1,\ldots,y^N) \\
     & \quad = - \sum_{(y^1,\ldots,y^N) \in \Ycal^N} P^N(y^1,\ldots,y^N) \log(M_{\Ycal^N}(P^{F,s,N}(Q_\beta))(y^1,\ldots, y^N))  \\
     & \quad = - \sum_{(y^1,\ldots,y^N) \in \Ycal^N} P^N(y^1,\ldots,y^N) \log\prod_{i=1}^N \int_\Bcal dQ_\beta(\beta_i) s(\beta_i, a^N_i,y^i)  \\
    & \quad = - \sum_{(y^1,\ldots,y^N) \in \Ycal^N} P^N(y^1,\ldots,y^N) \sum_{i=1}^N \log \int_\Bcal s(\beta,a^N_i,y^i)dQ_\beta(\beta)
    \\
    & \quad = - \sum_{i=1}^N \sum_{(y^1,\ldots,y^N) \in \Ycal^N} P^N(y^1,\ldots,y^N)  \log \int_\Bcal dQ_\beta(\beta_i) s(\beta_i, a^N_i,y^i) \\
    & \quad = -\sum_{a \in \{a^N_1,\ldots, a^N_N\}} \sum_{(y^1,\ldots,y^N) \in \Ycal^N} \sum_{i \in \couNA} P^N(y^1,\ldots,y^N)  \log \int_\Bcal s(\beta,a^N_i,y^i)dQ_\beta(\beta) \\
    & \quad = - \sum_{(a,y) \in \pairay^N} n_a Q^{np,N}_{Y\mid a}(P^N)(a,y) \log  \int_\Bcal s(\beta,a,y)dQ_\beta(\beta) .
\end{align}
Adding and subtracting $\log Q^{np,N}_{Y\mid a}(P^N)(a,y)$ for any $a, y$, we have
\begin{align}
    \risk(P^N, P^{F,s,N}(Q_\beta)) & = \sum_{(a,y) \in \pairay^N} n_a Q^{np,N}_{Y\mid a}(P^N)(a,y) \log \left(\frac{Q^{np,N}_{Y\mid a}(P^N)(a,y)}{\int_\Bcal s(\beta,a,y)dQ_\beta(\beta)}\right) \\
    & \quad - C^N(P^N) \\
    & =  \sum_{a \in \{a^N_1,\ldots, a^N_N\}} n_a \sum_{y \in a} Q^{np,N}_{Y\mid a}(P^N)(a,y) \log \left(\frac{Q^{np,N}_{Y\mid a}(P^N)(a,y)}{\int_\Bcal s(\beta,a,y)dQ_\beta(\beta)}\right) \\
    & \quad - C^N(P^N) \\
    & = \sum_{a \in \{a^N_1,\ldots, a^N_N\}} n_a \KL(Q^{np,N}_{Y\mid a}(P^N)(a,\cdot)\|  \int_\Bcal s(\beta,a,\cdot)dQ_\beta(\beta)) - C^N(P^N) .
\end{align}
where
\begin{align}
    C^N(P^N) = \sum_{(a,y) \in \pairay^N} n_a  Q^{np,N}_{Y\mid a}(P^N)(a,y) \log Q^{np,N}_{Y\mid a}(P^N)(a,y) .
\end{align}
\end{proof}

\begin{lemma}\label{lemma:profile_attainment}
Suppose that \Cref{assumption:compact_continuity_certify} holds, and let $P^N\in\Mcal^{np,N}$. Then, for every $\psi \in [0,1]$ such that
\begin{align}
\{Q_\beta \in \Qcal_\beta \colon \tilde{\Psi}(Q_\beta) = \psi\} \neq \varnothing,
\end{align}
we have that $\inf_{Q_\beta \in \Qcal_\beta \colon \Tilde{\Psi}(Q_\beta) = \psi} \risk(P^N, P^{F,s,N}(Q_\beta))$ is attained for some $Q_\beta \in \Qcal_\beta$.
\end{lemma}

\begin{proof}[Proof of \Cref{lemma:profile_attainment}]
Since $\Bcal$ is a compact metric space, $\Qcal_\beta = \Pcal(\Bcal,\mathscr{B})$ is compact under the weak topology induced by the metric topology on $\Bcal$. Since $g$ is continuous and bounded, the map $\tilde{\Psi}$ is continuous. Hence the set
\begin{align}
\Qcal_{\beta,\psi} = \{Q_\beta \in \Qcal_\beta \colon \tilde\Psi(Q_\beta) = \psi\}
\end{align}
is closed in $\Qcal_\beta$, and therefore compact. Since $\bar{t}_l^N$ is continuous and bounded, for any $l \in [L^N]$, the map
\begin{align}
\Qcal_\beta \to \RR , \quad Q_\beta\mapsto \int_\Bcal \bar t_l^N(\beta)\,dQ_\beta(\beta)
\end{align}
is continuous. Therefore the
model choice probabilities are continuous functions of $Q_\beta$. From \Cref{lemma:rewriting_risk}, we can rewrite the risk as
\begin{align}
\risk(P^N,P^{F,s,N}(Q_\beta)) = \sum_{a\in\{a_1^N,\ldots,a_N^N\}} n_a
\KL( Q^{np,N}_{Y\mid a}(P^N)(a,\cdot) \|
\int_\Bcal s(\beta, a, \cdot) dQ_\beta(\beta)) - C^N(P^N) ,
\end{align}
where $C^N(P^N)$ does not depend on $Q_\beta$. Since $\KL$ divergence is lower
semicontinuous on a finite simplex, the risk is lower semicontinuous in $Q_\beta$. Hence, by compactness of $\Qcal_{\beta,\psi}$, the infimum $\inf_{Q_\beta \in \Qcal_\beta \colon \Tilde{\Psi}(Q_\beta) = \psi}$ $\risk(P^N, P^{F,s,N}(Q_\beta))$ is attained for some $Q_\beta \in \Qcal_{\beta, \psi}$.
\end{proof}

\theoremmatchingmle*

\begin{proof}[Proof of \Cref{theorem:matching_mle}]
Suppose that \Cref{assumption:compact_continuity_certify} holds. Since $P^N$ is
realizable under $(\Qcal_\beta,s)$, there exists $Q_\beta^0\in\Qcal_\beta$ such
that $M_{\Ycal^N}(P^{F,s,N}(Q_\beta^0))=P^N$. By \Cref{lemma:equivalence_matching_moments}, for any
$a \in \{a^N_1,\ldots,a^N_N\}$ and $y\in a$, we have
\begin{align}
\int_\Bcal s(\beta,a,y)dQ_\beta^0(\beta)
=
Q^{np,N}_{Y\mid a}(P^N)(a,y).
\end{align}
Therefore, by \Cref{lemma:rewriting_risk}, we have $\risk(P^N,P^{F,s,N}(Q_\beta^0)) = - C^N(P^N)$. Moreover, for any $Q_\beta\in\Qcal_\beta$, \Cref{lemma:rewriting_risk} gives that
\begin{align}
\risk(P^N,P^{F,s,N}(Q_\beta)) & = \sum_{a\in\{a^N_1,\ldots,a^N_N\}}
n_a \KL (Q^{np,N}_{Y\mid a}(P^N)(a,\cdot) \| \int_\Bcal s(\beta,a,\cdot)dQ_\beta(\beta)) -C^N(P^N) \\
& \geq - C^N(P^N) ,
\end{align}
and hence
\begin{align}\label{equation:ereurliu}
    \inf_{\psi \in [0,1]} \profileloglike(P^N,\psi) = - C^N(P^N) .
\end{align}

We first prove that $\Psi^N(P^N)\subseteq \Psi^{\mle,N}(P^N)$. Let $\psi \in \Psi^N(P^N)$. By definition of $\Psi^N$, there exists
$Q_\beta\in\Qcal_\beta$ such that
\begin{align}
\tilde{\Psi}(Q_\beta)=\psi,
\qquad
M_{\Ycal^N}(P^{F,s,N}(Q_\beta))=P^N .
\end{align}
By \Cref{lemma:equivalence_matching_moments}, for any
$a\in\{a^N_1,\ldots,a^N_N\}$ and $y\in a$, we have
\begin{align}
\int_\Bcal s(\beta,a,y)dQ_\beta(\beta)
=
Q^{np,N}_{Y\mid a}(P^N)(a,y).
\end{align}
Therefore, by \Cref{lemma:rewriting_risk}, we have $\risk(P^N,P^{F,s,N}(Q_\beta)) = - C^N(P^N)$.
Since $\tilde{\Psi}(Q_\beta)=\psi$, we have
\begin{align}
\profileloglike(P^N,\psi)\leq - C^N(P^N) ,
\end{align}
but from \eqref{equation:ereurliu}, we have $\profileloglike(P^N,\psi)\geq - C^N(P^N)$, and thus $\profileloglike(P^N,\psi)= - C^N(P^N)$. Consequently, $\psi\in\Psi^{\mle,N}(P^N)$.

We now prove that $\Psi^{\mle,N}(P^N)\subseteq \Psi^N(P^N)$. Let $\psi \in \Psi^{\mle,N}(P^N)$. Since
\begin{align}
\inf_{\psi' \in [0,1]} \profileloglike(P^N, \psi') = - C^N(P^N) ,
\end{align}
we have $\profileloglike(P^N,\psi)= - C^N(P^N)$, and hence $\{Q_\beta\in\Qcal_\beta\colon \tilde{\Psi}(Q_\beta)=\psi\}\neq\varnothing$. By \Cref{lemma:profile_attainment}, there exists $Q_\beta^\psi\in\Qcal_\beta$
such that
\begin{align}
\tilde{\Psi}(Q_\beta^\psi)=\psi, \quad \risk(P^N,P^{F,s,N}(Q_\beta^\psi)) = \profileloglike(P^N,\psi) = -C^N(P^N) .
\end{align}
By \Cref{lemma:rewriting_risk},
\begin{align}
\sum_{a\in\{a^N_1,\ldots,a^N_N\}} n_a \KL(Q^{np,N}_{Y\mid a}(P^N)(a,\cdot) \| \int_\Bcal s(\beta,a,\cdot)dQ_\beta^\psi(\beta)
) = 0 .
\end{align}
Since each term in the sum is nonnegative and $n_a>0$ for any
$a\in\{a^N_1,\ldots,a^N_N\}$, we have
\begin{align}
\KL\left(
Q^{np,N}_{Y\mid a}(P^N)(a,\cdot)
\|
\int_\Bcal s(\beta,a,\cdot)dQ_\beta^\psi(\beta)
\right)
=
0 .
\end{align}
Therefore, for any $a\in\{a^N_1,\ldots,a^N_N\}$ and $y\in a$, we have
\begin{align}
\int_\Bcal s(\beta,a,y)dQ_\beta^\psi(\beta)
=
Q^{np,N}_{Y\mid a}(P^N)(a,y).
\end{align}
and thus, by \Cref{lemma:equivalence_matching_moments}, $M_{\Ycal^N}(P^{F,s,N}(Q_\beta^\psi))=P^N $. Since $\tilde{\Psi}(Q_\beta^\psi)=\psi$, it follows that $\psi\in\Psi^N(P^N)$, which concludes the proof.
\end{proof}

\begin{definition}[Convex hull]
    For any subset $U$ of a real affine space, the convex hull of $U$, which we write $\mathrm{Conv}(U)$, is defined as the set of all finite convex combinations of elements of $U$, that is
    \begin{align}
        \mathrm{Conv}(U) = \left\lbrace \sum_{i=1}^n \lambda_i u_i \colon \exists n \geq 1,\lambda_1,\ldots, \lambda_n \in \RR^+, \sum_{i=1}^n \lambda_i=1, u_1,\ldots,u_n \in U \right\rbrace.
    \end{align}
\end{definition}

\begin{definition}[Affine hull]
    For any subset $U$ of a real affine space, the affine hull of $U$, which we write $\mathrm{aff}(U)$, is defined as the set of all finite affine combinations of elements of $U$, that is
    \begin{align}
        \mathrm{aff}(U) =
        \left\lbrace
        \sum_{i=1}^n \lambda_i u_i \colon
        \exists n \geq 1,\ 
        \lambda_1,\ldots,\lambda_n \in \RR,\
        \sum_{i=1}^n \lambda_i = 1,\
        u_1,\ldots,u_n \in U
        \right\rbrace.
    \end{align}
\end{definition}

\begin{lemma}\label{lemma:h_in_closure_hull}
    Let $L$ be a positive integer, $h \colon \Bcal \to \RR^L$ be a bounded and measurable map, and $H = \{h(\beta) \colon \beta \in \Bcal \} \subseteq \RR^L$ denote the range of $h$. Suppose that $H$ is bounded. Then for any $\pi \in \Pcal(\Bcal, \mathscr{B})$, we have that
    \begin{align}
        \int_{\Bcal} h d\pi \in \overline{\Conv(H)} .
    \end{align}
\end{lemma}

\begin{proof}[Proof of \Cref{lemma:h_in_closure_hull}]
Since $H$ is bounded, for any $n \in \NN$, there exists $N_n \in \NN$ and a finite Borel partition $C_n = (C_{n,k})_{k=1}^{N_n}$ of $H$ such that for every $k=1,\ldots,N_n$, $C_{n,k}$ has Euclidean diameter at most $1/n$. Let $(C_n)_{n \geq 1}$ be a sequence of such partitions.

For any $n \in \NN, k = 1, \ldots, N_n$, let $x_{n,k}$ be an arbitrary element of $C_{n,k}$, and $\Bcal_{n,k} = \{\beta \in \Bcal \colon h(\beta) \in C_{n,k}\}$ be the preimage of $C_{n,k}$ by $h$. For any $n \in \NN$, let $\phi_n \colon \Bcal \to \RR^L$ be defined, for any $\beta \in \Bcal$, by
\begin{align}\label{equation:ewroiuw4eri}
    \phi_n(\beta) = \sum_{k=1}^{N_n} \1 \{ \beta \in \Bcal_{n,k} \} x_{n,k} .
\end{align}
Observe that since $h$ is measurable, every $\Bcal_{n,k} \in \mathscr{B}$, and therefore every $\phi_n$ is $\mathscr{B}$-measurable.
For any $\beta \in \Bcal$ and $n \in \NN$, there exists $k =1,\ldots,N_n$ such that $h(\beta) \in C_{n,k}$ and therefore, since $x_{n,k}$ is in $C_{n,k}$ and $C_{n,k}$ has diameter at most $1/n$, $\| \phi_n(\beta) - h(\beta)\| \leq 1/n$. Since $\beta$ is arbitrary and the right-hand side does not depend on $\beta$, $\phi_n$ converges uniformly to $h$. From \eqref{equation:ewroiuw4eri}, we have that
\begin{align}
    \int_\Bcal \phi_n d\pi \in \Conv(H).
\end{align}
By uniform convergence of $(\phi_n)_{n \geq 1}$ to $h$, we have that
    \begin{align}
        \lim_{n \to \infty} \int_\Bcal \phi_n d\pi =  \int_\Bcal \lim_{n \to \infty} \phi_n d\pi = \int_\Bcal h d\pi ,
    \end{align}
    and consequently we have that $\int_\Bcal h d\pi \in \overline{\Conv(H)}$ since $\int_\Bcal h d\pi$ is a limit of a convergent sequence with elements in $\Conv(H)$.
\end{proof}

\begin{lemma}\label{lemma:one_dimensional_case}
Let $h \colon \Bcal \to \RR$ be bounded and measurable with respect to $\mathscr{B}$, and $H = \{h(\beta) \colon \beta \in \Bcal\}$. For any $\pi \in \Pcal(\Bcal, \mathscr{B})$, we have that
    \begin{align}
        \int_\Bcal h(\beta) d\pi(\beta) \in \Conv(H) .
    \end{align}
\end{lemma}

\begin{proof}[Proof of \Cref{lemma:one_dimensional_case}]
    Since $h$ has its image in $\RR$, for any $\beta \in \Bcal$, we have that
    \begin{align}
        \inf_\Bcal h \leq h(\beta) \leq \sup_\Bcal h ,
    \end{align}
    and integrating with a positive measure gives
    \begin{align}
        \inf_\Bcal h \leq \int_\Bcal h(\beta) d\pi(\beta) \leq \sup_\Bcal h .
    \end{align}
    If $\inf_\Bcal h = \int_\Bcal h(\beta) d\pi(\beta)$, then we have that $h(\beta) = \inf_\Bcal h$ $\pi$-almost everywhere, which implies that $\inf_\Bcal h \in H$, and hence $\int_\Bcal h(\beta) d\pi(\beta) \in H \subseteq \Conv(H)$. The same holds if we have $\sup_\Bcal h = \int_\Bcal h(\beta) d\pi(\beta)$. Finally, if $\inf_\Bcal h < \int_\Bcal h(\beta) d\pi(\beta) < \sup_\Bcal h$, then $(\inf_\Bcal h, \sup_\Bcal h) \subseteq \Conv(H)$ and hence $\int_\Bcal h(\beta) d\pi(\beta) \in \Conv(H)$.
\end{proof}

\begin{lemma}\label{lemma:h_in_convex_hull}
    Let $d$ and $L$ be positive integers, $h \colon \Bcal \to \RR^L$ be bounded and measurable with respect to $\mathscr{B}$, and $H = \{h(\beta) \colon \beta \in \Bcal\} \subseteq \RR^L$. Then for any $\pi \in \Pcal(\Bcal)$ a probability measure on $\Bcal$, we have that
    \begin{align}
        \int_\Bcal h d\pi \in \Conv(H) .
    \end{align}
\end{lemma}

\begin{proof}[Proof of \Cref{lemma:h_in_convex_hull}]
    We prove the lemma by induction on $L = \dim(\aff(H))$.
    \paragraph{Initialization.} The case $\dim(\aff(H))=0$ being trivial, we consider $\dim(\aff(H))=1$. Let $x_0 \in H$, there exists  $v \in \RR^{L}$ such that $\aff(H) = x_0 + \RR v$. Let $R = \{r \in \RR \text{ such that } x_0 + r v \in H \}$. $H$ being bounded implies that $R$ is bounded. Therefore, there exists $r \colon \Bcal \to \RR$ measurable such that for any $\beta \in \Bcal, h(\beta) = x_0 + r(\beta) v$. Let $\Bar{r} = \int_\Bcal r(\beta) d\pi(\beta)$, we have that
    \begin{align}
        \int_\Bcal h(\beta) d\pi(\beta) = \int_\Bcal (x_0 + r(\beta) v) d\pi(\beta) = x_0 + \int_\Bcal r(\beta) d\pi(\beta) \, v = x_0 + \Bar{r} v \eqsp.
    \end{align}
    Therefore, \Cref{lemma:one_dimensional_case} gives that $\Bar{r} \in \Conv(R)$. Hence, there exists $r_1, \ldots, r_m \in R, \lambda_1, \ldots, \lambda_m$ $\in (0,1)^m$ such that $\sum_{l=1}^m\lambda_l = 1, \Bar{r} = \sum_{l=1}^m \lambda_l r_l$ where $x_0 + r_l v \in H$ for any $l = 1, \ldots, m$ and we have
    \begin{align}
        \int_\Bcal h(\beta) d\pi(\beta) = x_0 + \Bar{r} v = x_0 + \sum_{l=1}^m \lambda_l r_l v = \sum_{l=1}^m \lambda_l \underbrace{(x_0 +r_l v)}_{\in H} \in \Conv(H) .
    \end{align}
    Therefore, our initialization holds.

    \paragraph{Induction.} Now assume that the property holds for any $H$ such that $\dim(\aff(H)) \leq L-1$. Let $L = \dim(\aff(H))$. By \Cref{lemma:h_in_closure_hull}, we have that $\int_\Bcal h d\pi \in \overline{\Conv(H)}$. If $\int_\Bcal h d\pi \in \Conv(H)$, then we have the result. Otherwise, we have that
    $\int_\Bcal h d\pi \in \overline{\Conv(H)} \backslash \Conv(H)$ and hence $\int_\Bcal h d\pi \in \partial \Conv(H)$.

Since $H \subset \aff(H)$ and $\aff(H)$ is closed, $\int_\Bcal h\,d\pi\in \aff(H)$. Thus $\int_\Bcal h\,d\pi$ lies on
the relative boundary of $\Conv(H)$ in the affine space $\aff(H)$. Therefore, by
the supporting hyperplane theorem \citep[2.5.2][]{boyd2004convex} applied in the space $\aff(H)$, there
exist $a\in\RR^L$ and $c\in\RR$ such that
\begin{align}
    & a^\top u \leq c
    \quad \text{for any } u \in \Conv(H), \\
    & \label{equation:dfjleskjf}
    a^\top \int_\Bcal h\,d\pi = c ,
\end{align}
and such that the affine hyperplane $\{x\in\RR^L:a^\top x=c\}$ does not contain $\aff(H)$. This gives that for any $\beta \in \Bcal$,
\begin{align}\label{equation:dfjewlrtkj}
    a^\top h(\beta) \leq c ,
\end{align}
since $h(\beta)\in H\subseteq \Conv(H)$. Rewriting
\Cref{equation:dfjleskjf} using the fact that $\pi$ integrates to $1$, we
have
\begin{align}\label{equation:dfjelkaqq}
    \int_\Bcal (a^\top h(\beta)-c)\,d\pi(\beta)=0 .
\end{align}
Since \Cref{equation:dfjewlrtkj} ensures that the integrand in
\Cref{equation:dfjelkaqq} is everywhere non-positive, we have
$a^\top h(\beta)=c$ $\pi$-almost everywhere. Let
\begin{align}
\Bcal_0=\{\beta\in\Bcal:a^\top h(\beta)=c\}.
\end{align}
Then $\pi(\Bcal_0)=1$, hence $\int_\Bcal h\,d\pi=\int_{\Bcal_0}h\,d\pi$. Moreover,
\begin{align}
    h(\Bcal_0)\subseteq H\cap\{x\in\RR^L \colon a^\top x=c\}.
\end{align}
Since the affine hyperplane $\{x\in\RR^L:a^\top x=c\}$ does not contain
$\aff(H)$, the intersection $\aff(H) \cap \{x\in\RR^L \colon a^\top x = c\}$ is an affine subspace of $\aff(H)$ satisfying $\dim(\aff(H) \cap\{x\in\RR^L \colon a^\top x=c\})\leq L-1$. Since $h(\Bcal_0) \subseteq \aff(H) \cap\{x \in\RR^L \colon a^\top x = c\}$,
we obtain $\dim(\aff(h(\Bcal_0)))\leq L-1$. Applying the induction hypothesis to $h(\Bcal_0)$ gives that
\begin{align}
    \int_{\Bcal_0} h\,d\pi\in \Conv(h(\Bcal_0)).
\end{align}
Since $h(\Bcal_0)\subseteq H$, we have $\Conv(h(\Bcal_0))\subseteq \Conv(H)$, thus
\begin{align}
    \int_\Bcal h\,d\pi\in \Conv(H),
\end{align}
hence the induction step and the result.
\end{proof}

A more general result than \Cref{lemma:h_in_convex_hull} is given in \citet[][Lemma 2.16.]{schafer2025beyond}.

We now define
\begin{align}
h \colon \Bcal & \to \RR^{\tilde L^N+1}, \\
\beta &\mapsto (t^N(\beta)^\top,g(\beta))^\top ,
\end{align}
which is bounded and measurable with respect to $\mathscr{B}$. For any
$P^N \in \Mcal^{np, N}$, let
\begin{align}
\Lcal^N(P^N) = \{R^{F,N} \in \Pcal(\Zcal^N, \mathscr{Z}^N)
\colon
M_{\Ycal^N}(R^{F,N}) = P^N \}.
\end{align}

\begin{lemma}\label{lemma:finite_support_realization}
Let $P^N\in\Mcal^{np,N}$ be realizable under $(\Qcal_\beta,s)$ in the sense
of \Cref{definition:correct_fit}. Let $Q_\beta\in\Qcal_\beta$ be such that $P^{F,s,N}(Q_\beta)\in \Lcal^N(P^N)$ and let $\psi=\int_\Bcal g(\beta)dQ_\beta(\beta)$. Then there exist $B^N\subset \Bcal$ and $Q_\beta^N\in\Qcal_\beta$ such that $1\leq \Card(B^N)\leq \tilde L^N+2, Q_\beta^N(B^N)=1$, and
\begin{align}
P^{F,s,N}(Q_\beta^N)\in \Lcal^N(P^N),
\quad \int_\Bcal g(\beta) dQ_\beta^N (\beta)=\psi .
\end{align}
\end{lemma}

\begin{proof}[Proof of \Cref{lemma:finite_support_realization}]
Since $P^{F,s,N}(Q_\beta)\in \Lcal^N(P^N)$, we have
\begin{align}
M_{\Ycal^N}(P^{F,s,N}(Q_\beta))=P^N .
\end{align}
By \Cref{lemma:equivalence_matching_moments},
\begin{align}
\int_\Bcal t^N(\beta)dQ_\beta(\beta)=b^N(P^N).
\end{align}
Moreover,
\begin{align}
(b^N(P^N)^\top, \psi)^\top = \int_\Bcal (t^N(\beta)^\top,g(\beta))^\top
dQ_\beta(\beta) = \int_\Bcal h(\beta)dQ_\beta(\beta).
\end{align}
By \Cref{lemma:h_in_convex_hull},
\begin{align}
(b^N(P^N)^\top,\psi)^\top
\in \Conv\left( \{(t^N(\beta)^\top, g(\beta))^\top \colon \beta\in\Bcal\}
\right) .
\end{align}
Since $\Conv\left( \{(t^N(\beta)^\top, g(\beta))^\top \colon \beta\in\Bcal\}
\right)$ is a convex subset of $\RR^{\tilde{L}^N+1}$, Carathéodory's theorem
implies that there exist $K\leq \tilde L^N+2$,
$\beta_1,\ldots,\beta_K\in\Bcal$, and
$q_1,\ldots,q_K\in[0,1]$ with $\sum_{k=1}^K q_k=1$ such that
\begin{align}
(b^N(P^N)^\top,\psi)^\top = \sum_{k=1}^K q_k (t^N(\beta_k )^\top , g(\beta_k))^\top .
\end{align}
Let
\begin{align}
B^N=\{\beta_1,\ldots,\beta_K\}, \quad Q_\beta^N=\sum_{k=1}^K q_k\delta_{\beta_k}.
\end{align}
Then $Q_\beta^N(B^N)=1$, $1\leq \Card(B^N)\leq K\leq \tilde L^N+2$, and
\begin{align}
\int_\Bcal t^N(\beta)dQ_\beta^N(\beta)
= b^N(P^N), \qquad \int_\Bcal g(\beta)dQ_\beta^N(\beta) = \psi . 
\end{align}
The remaining moment restrictions indexed by
$l=\tilde L^N+1,\ldots,L^N$ are linear combinations of
$1,t^N_1,\ldots,t^N_{\tilde L^N}$, hence are also matched. Therefore, for any
$(a, y) \in \pairay^N$,
\begin{align}
\int_\Bcal s(\beta,a,y)dQ_\beta^N(\beta) = Q^{np,N}_{Y\mid a}(P^N)(a,y).
\end{align}
Applying \Cref{lemma:equivalence_matching_moments} again gives
\begin{align}
M_{\Ycal^N}(P^{F,s,N}(Q_\beta^N))=P^N .
\end{align}
Thus
\begin{align}
P^{F,s,N}(Q_\beta^N)\in \Lcal^N(P^N) ,
\end{align}
which concludes the proof.
\end{proof}

\theoremdefinitionpsiwithsets*

\begin{proof}[Proof of \Cref{theorem:definition_psi_with_sets}]
By definition,
\begin{align}
\Psi^{+,N}(P^N) = \sup\left\{
\int_\Bcal g(\beta)dQ_\beta(\beta)
\colon Q_\beta\in\Qcal_\beta,\ 
P^{F,s,N}(Q_\beta)\in\Lcal^N(P^N) \right\}.
\end{align}
Let $Q_\beta\in\Qcal_\beta$ be such that
$P^{F,s,N}(Q_\beta)\in\Lcal^N(P^N)$, and let
\begin{align}
\psi=\int_\Bcal g(\beta)dQ_\beta(\beta).
\end{align}
By \Cref{lemma:finite_support_realization}, there exists
$Q_\beta^N\in\Qcal_\beta$ with support $B^N=\{\beta_1,\ldots,\beta_K\}$
such that $K \leq \tilde{L}^N+2, Q_\beta^N=\sum_{k=1}^K q_k\delta_{\beta_k},
\sum_{k=1}^K q_k=1$, and
\begin{align}
P^{F,s,N}(Q_\beta^N)\in\Lcal^N(P^N),
\quad \int_\Bcal g(\beta)dQ_\beta^N(\beta)=\psi .
\end{align}
By \Cref{lemma:equivalence_matching_moments}, $\sum_{k=1}^K q_k t^N(\beta_k)=b^N(P^N)$, and therefore
\begin{align}
\sum_{k=1}^K q_k g(\beta_k)=\psi,
\qquad
\sum_{k=1}^K q_k t^N(\beta_k)=b^N(P^N).
\end{align}
Hence every value $\psi \in \Psi^N(P^N)$ belongs to the set
\begin{equation}\label{equation:definition_set_proof_atoms}
\begin{aligned}
    \bigg\{ \psi\in[0,1] \colon & \exists K \leq \tilde{L}^N+2,\ 
\beta_1,\ldots,\beta_K \in \Bcal,\ 
q_1,\ldots, q_K \in [0,1],\ 
\sum_{k=1}^K q_k =1 , \\
& \sum_{k=1}^K q_k g(\beta_k) = \psi,\quad
\sum_{k=1}^K q_k t^N(\beta_k) = b^N(P^N)
\bigg\}.
\end{aligned}
\end{equation}
Conversely, suppose that there exist $K \leq \tilde{L}^N+2, \beta_1,\ldots,\beta_K \in \Bcal$, and $(q_1,\ldots,q_K) \in \Delta_K$ such that $\sum_{k=1}^K q_k g(\beta_k)=\psi$ and $\sum_{k=1}^K q_k t^N(\beta_k)=b^N(P^N)$. Let $Q_\beta=\sum_{k=1}^K q_k\delta_{\beta_k}$, then
\begin{align}
\int_\Bcal g(\beta)dQ_\beta(\beta)=\psi, \quad \int_\Bcal t^N(\beta)dQ_\beta(\beta)=b^N(P^N).
\end{align}
The remaining moment restrictions indexed by
$l=\tilde{L}^N+1,\ldots,L^N$ are linear combinations of
$1,t^N_1,\ldots,t^N_{\tilde{L}^N}$, hence are also matched. Therefore,
\Cref{lemma:equivalence_matching_moments} gives that $P^{F,s,N}(Q_\beta)\in\Lcal^N(P^N)$, hence $\psi \in \Psi^N(P^N)$.

Therefore, the sets $\Psi^N(P^N)$ and the one defined in \eqref{equation:definition_set_proof_atoms} are equal, taking suprema gives the desired result.
\end{proof}

\section{Proofs for the computation of the estimator via the expectation-maximization algorithm}\label{section:appendix_optimization_procedure}

\subsection{Supporting Lemmas}

\begin{restatable}{lemma}{lemmarightprojsolution}\label{lemma:right_proj_solution}
Let $r \in \Rcal^{N,K}(P^N,\psi)$. Then, for any $(w,\beta) \in \Delta_K \times \Bcal^K$ such that
\begin{align}
(w,\beta) \in \argmin_{w \in \Delta_K, (\beta_1,\ldots,\beta_K) \in \Bcal^K} \sum_{k=1}^K \sum_{l=1}^{L^N} \sum_{m=1}^2 r(k,l,m)\log\frac{r(k,l,m)}{w_k e_l^N \bar{t}_l^N(\beta_k)\gamma(m,\beta_k)}
\end{align}
we have that, for any $k \in [K]$,
\begin{align}
w_k \label{equation:objective_w_left} &= \sum_{l=1}^{L^N} \sum_{m=1}^2 r(k,l,m) , \\
\beta_k \label{equation:objective_beta} & \in \argmin_{\beta \in \Bcal} - \sum_{l=1}^{L^N} \sum_{m=1}^2 r(k,l,m) \left\{ \log \bar{t}_l^N(\beta) + \log \gamma(m,\beta) \right\} .
\end{align}
\end{restatable}

\begin{proof}[Proof of \Cref{lemma:right_proj_solution}]
Let $(w,\beta) \in \Delta_K \times \Bcal^K$ such that
\begin{align}\label{equation:wrpiwrp}
(w,\beta) \in \argmin_{w \in \Delta_K, (\beta_1,\ldots,\beta_K) \in \Bcal^K} \sum_{k=1}^K \sum_{l=1}^{L^N} \sum_{m=1}^2 r(k,l,m)\log\frac{r(k,l,m)}{w_k e_l^N \bar{t}_l^N(\beta_k)\gamma(m,\beta_k)} .
\end{align}
Dropping all terms that do not depend on $w$ or $\beta$, the optimization problem in \eqref{equation:wrpiwrp} is equivalent to maximizing
\begin{align}
\sum_{k=1}^K \sum_{l=1}^{L^N} \sum_{m=1}^2 r(k,l,m) \left\{ \log w_k + \log \bar{t}_l^N(\beta_k) + \log \gamma(m,\beta_k) \right\}.
\end{align}
The part of \eqref{equation:wrpiwrp} depending on $w$ is
\begin{align}
\sum_{k=1}^K  \sum_{l=1}^{L^N} \sum_{m=1}^2 r(k,l,m) \log w_k .
\end{align}
For any $w \in \Delta_K$, we have that
\begin{align}
\sum_{k=1}^K  \sum_{l=1}^{L^N} \sum_{m=1}^2 r(k,l,m) \log\frac{ \sum_{l=1}^{L^N} \sum_{m=1}^2 r(k,l,m)}{w_k} \geq 0 ,
\end{align}
with equality if, and only if $w_k =  \sum_{l=1}^{L^N} \sum_{m=1}^2 r(k,l,m)$ for any $k \in [K]$. By optimality in \eqref{equation:wrpiwrp}, we necessarily have that
\begin{align}
w_k = \sum_{l=1}^{L^N} \sum_{m=1}^2 r(k,l,m) .
\end{align}
It remains to characterize the coordinates $\beta_k$ in \eqref{equation:wrpiwrp}. Once $w$ is fixed, we have that
\begin{align}
(\beta_1, \ldots, \beta_K) \in \argmin_{\Bcal^K} \sum_{k=1}^K \sum_{l=1}^{L^N} \sum_{m=1}^2 r(k,l,m)\log\frac{r(k,l,m)}{w_k e_l^N \bar{t}_l^N(\beta_k)\gamma(m,\beta_k)} ,
\end{align}
and keeping only the part depending on $\beta_1, \ldots, \beta_K$, we obtain
\begin{align}
    (\beta_1, \ldots, \beta_K) \in \argmin_{\Bcal^K} - \sum_{k=1}^K \sum_{l=1}^{L^N} \sum_{m=1}^2 r(k,l,m) \left\{ \log \bar{t}_l^N(\beta_k) + \log \gamma(m,\beta_k) \right\} .
\end{align}
Therefore, for any $k \in [K]$, we have that
\begin{align}
\beta_k \in \argmin_{\Bcal}  -\sum_{l=1}^{L^N} \sum_{m=1}^2 r(k,l,m) \left\{ \log \bar{t}_l^N(\beta) + \log \gamma(m,\beta) \right\}.
\end{align}
If $\sum_{l=1}^{L^N} \sum_{m=1}^2 r(k,l,m) = 0$, then $r(k,l,m)=0$ for any $(l, m) \in [L^N] \times [2]$, and the objective is independent of $\beta_k$. It concludes the proof.
\end{proof}

\begin{restatable}{lemma}{lemmaconvexatomupdate}\label{lemma:convex_atom_update}
Let $r\in\Rcal^{N,K}(P^N, \psi)$ and $k\in[K]$. Suppose that \Cref{assumption:convexity_beta_optim} holds, and that $1 - g$ is log-concave. Then, the optimization problem
\begin{align}
\argmin_{\beta\in\Bcal}
-\sum_{l=1}^{L^N}\sum_{m=1}^2r(k,l,m)\{\log\bar{t}_l^N(\beta)+\log\gamma(m,\beta)\}
\end{align}
is convex.
\end{restatable}

\begin{proof}[Proof of \Cref{lemma:convex_atom_update}]
Since $\bar{t}_l^N$, $g$, and $1-g$ are positive and log-concave on $\Bcal$, the mappings $-\log\bar{t}^N_l$, $- \log g$, and $-\log (1-g)$ are convex on $\Bcal$. Since $r(k,l,m) \geq 0$ for every $l\in[L^N]$ and $m \in [2]$, the mapping
\begin{align}
\beta \mapsto - \sum_{l=1}^{L^N} \sum_{m=1}^2 r(k,l,m) \{\log \bar{t}_l^N(\beta) + \log\gamma(m,\beta)\}
\end{align}
is convex on $\Bcal$. The set of minimizers of a convex mapping over a convex set is convex.
\end{proof}

\begin{lemma}\label{lemma:attainmentKL}
    Under \Cref{assumption:compact_continuity_certify}, $D_K^N(P^N,\psi)=0$ if, and only if there exist $w\in\Delta_K$, $\beta \in \Bcal^K$, and $r \in \Rcal^{N,K} (P^N, \psi)$ such that $r = q_{w,\beta}^N$.
\end{lemma}

\begin{proof}[Proof of \Cref{lemma:attainmentKL}]
For any $w \in \Delta_K, \beta \in \Bcal^K$, we have that
\begin{align}
    \sum_{k=1}^K \sum_{l=1}^{L^N} \sum_{m=1}^2 q_{w,\beta}^N (k,l, m) & = \sum_{k=1}^K \sum_{l=1}^{L^N} \sum_{m=1}^2 w_k e_l^N \bar{t}_l^N(\beta_k) \gamma (m, \beta_k) \\
    & = \sum_{k=1}^K \sum_{l=1}^{L^N} w_k e_l^N \bar{t}_l^N(\beta_k) \underbrace{\sum_{m=1}^2 \{\1_{\{1\}}(m) g(\beta_k) + \1_{\{1\}}(m - 1) (1-g(\beta_k))\}}_{=1} \\
    & = \sum_{k=1}^K w_k \underbrace{\sum_{l=1}^{L^N}  e_l^N \bar{t}_l^N(\beta_k)}_{= 1} \\
    & = 1 ,
\end{align}
hence for any $w \in \Delta_K, \beta \in \Bcal^K$, we have that $q_{w,\beta}^N \in \Delta_{K \times L^N \times 2}$. Under \Cref{assumption:compact_continuity_certify}, the feasible set
\begin{align}
\{(w,\beta,r)\in \Delta_K \times \Bcal^K \times \Rcal^{N,K}(P^N, \psi) \}
\end{align}
is compact. Moreover, the map
\begin{align}
\Delta_K \times \Bcal^K \times \Rcal^{N,K}(P^N, \psi) & \to \bar{\RR} , \\
(w,\beta,r) & \mapsto \KL(r \| q_{w,\beta}^N)
\end{align}
is lower semicontinuous. Hence, whenever $D_K^N(P^N,\psi)=0$, the infimum
defining $D_K^N(P^N,\psi)$ is attained. Consequently, if $D_K^N(P^N,\psi)=0$, then there exist $w \in \Delta_K$, $\beta\in\Bcal^K$, and $r \in \Rcal^{N,K}(P^N, \psi)$ such that $\KL(r \| q_{w,\beta}^N) = 0$. Since both $r$ and $q_{w,\beta}^N$ are probability vectors on $\Delta_{K \times L^N \times 2}$, we thus have that $r = q_{w,\beta}^N$.

The reverse direction is immediate, if there exists $w \in \Delta_K, \beta \in \Bcal^K, r \in \Rcal^{N,K}(P^N, \psi)$ such that $r=q_{w,\beta}^N$, we thus have that $\KL(r \| q_{w,\beta}^N) = 0$, hence $D^N_K(P^N, \psi) = 0$, hence the equivalence.
\end{proof}

\subsection{Proofs of \Cref{proposition:validity_idealized_certify} and \Cref{theorem:local_basin_certify}}

\propositionvalidityidealizedcertify*

\begin{proof}[Proof of \Cref{proposition:validity_idealized_certify}]
Let $P^N\in\Mcal^{np,N}$ be realizable under $(\Qcal_\beta,s)$. Under
\Cref{assumption:compact_continuity_certify}, by \Cref{theorem:matching_mle}, we have that
\begin{align}
\Psi^{\mle,N}(P^N) = \Psi^N(P^N).
\end{align}
By the finite-support characterization in the proof of \Cref{theorem:definition_psi_with_sets}, $\psi\in\Psi^N(P^N)$ if and only if there exist $K\leq \tilde L^N+2$, $w\in\Delta_K$, and $\beta\in\Bcal^K$ such that $\sum_{k=1}^K w_k g(\beta_k) = \psi, \sum_{k=1}^K w_k t^N(\beta_k) = b^N(P^N)$. Since $P^N$ is realizable under $(\Qcal_\beta, s)$, this is equivalent
\begin{align}\label{equation:finite_atom_condition_certify}
\sum_{k=1}^K w_k g(\beta_k) &= \psi,
\qquad
\sum_{k=1}^K w_k \bar t^N(\beta_k) = \bar b^N(P^N).
\end{align}

Suppose first that $\psi\in\Psi^{\mle,N}(P^N)$. There exists $K \leq \tilde L^N+2$, $w \in \Delta_K$, and $\beta \in \Bcal^K$ satisfying \eqref{equation:finite_atom_condition_certify}. Define $r, q^N_{w, \beta} \in \Delta_{K \times L^N \times 2}$ as $r(k, l, m) = q_{w,\beta}^N (k,l,m) = w_k e_l^N \bar{t}^N_l(\beta_k) \gamma (m, \beta_k)$ for any $(k,l,m) \in [K] \times[L^N] \times[2]$.
We now verify that $r \in \Rcal^{N,K}(P^N,\psi)$. For any $l \in [L^N]$, we have
\begin{align}
\sum_{k=1}^K \sum_{m=1}^2 r(k,l,m)
& = \sum_{k=1}^K\sum_{m=1}^2
w_k e_l^N \bar{t}^N_l (\beta_k) \gamma(m, \beta_k) \\
& = e_l^N\sum_{k=1}^K \{w_k \bar{t}^N_l (\beta_k) \sum_{m=1}^2 \gamma (m, \beta_k) \} \\
& = e_l^N\sum_{k=1}^K w_k \bar{t}^N_l (\beta_k) \\
& = e_l^N\bar b_l^N(P^N) \\
& = \bar{\alpha}^N_l(P^N) .
\end{align}
Moreover, for $m=1$, we have
\begin{align}
\sum_{k=1}^K\sum_{l=1}^{L^N} r(k,l,1) & = \sum_{k=1}^K\sum_{l=1}^{L^N}
w_k e_l^N \bar{t}^N_l (\beta_k) g(\beta_k) \\
& = \sum_{k=1}^K w_k g(\beta_k) \sum_{l=1}^{L^N} e_l^N \bar{t}^N_l (\beta_k) \\
& = \sum_{k=1}^K w_k g(\beta_k) \\
& = \psi ,
\end{align}
which is equal to $\varpsi(1)$. Similarly, we obtain that $\sum_{k=1}^K\sum_{l=1}^{L^N} r(k,l,2) = \varpsi(2)$, which proves that $r \in \Rcal^{N,K}(P^N, \psi)$. We also have that $q^N_{w, \beta} \in \Mcal^N_K$. Since $r = q^N_{w, \beta}$, it proves that
\begin{align}
    \inf_{w \in \Delta_K, \beta \in \Bcal^K} \inf_{r \in \Rcal^{N,K}(P^N, \psi)} \KL(r \| q_{w, \beta}^N) = 0 ,
\end{align}
hence $D^N_K(P^N, \psi) = 0$, and we have the first direction.

Conversely, suppose that there exists $K \leq \tilde{L}^N+2$ such that
$D_K^N(P^N,\psi) = 0$. By \Cref{lemma:attainmentKL}, there exist
$w \in \Delta_K, \beta \in \Bcal^K$, and $r \in \Rcal^{N,K} (P^N,\psi)$ such that
\begin{align}
r=q_{w,\beta}^N .
\end{align}
Since $r\in\Rcal^{N,K}(P^N,\psi)$, for any $l\in[L^N]$, we have that $\sum_{k=1}^K\sum_{m=1}^2 r(k,l,m)
= \bar{\alpha}^N_l(P^N)$. Using $r = q_{w,\beta}^N$, it gives
\begin{align}
\bar\alpha_l^N(P^N) = \sum_{k=1}^K\sum_{m=1}^2
w_k e_l^N \bar{t}^N_l(\beta_k) \gamma(m, \beta_k) = e_l^N\sum_{k=1}^K w_k \bar{t}^N_l (\beta_k).
\end{align}
Since $\bar\alpha_l^N(P^N) = e_l^N\bar b_l^N(P^N)$ and $e_l^N>0$, we obtain for any $l \in [L^N]$
\begin{align}
\sum_{k=1}^K w_k\bar t_l^N(\beta_k)
= \bar{b}^N_l (P^N) .
\end{align}
Similarly, using the marginal constraint for $m=1$, we have
\begin{align}
\psi = \sum_{k=1}^K\sum_{l=1}^{L^N} r(k,l,1) = \sum_{k=1}^K \sum_{l=1}^{L^N}
w_k e_l^N\bar{t}^N_l (\beta_k) g(\beta_k) = \sum_{k=1}^K w_k g(\beta_k).
\end{align}
Therefore \eqref{equation:finite_atom_condition_certify} holds. Hence
$\psi\in\Psi^N(P^N)$. Since
$\Psi^N(P^N)=\Psi^{\mle,N}(P^N)$, we conclude that
\begin{align}
\psi \in \Psi^{\mle,N} (P^N) ,
\end{align}
which proves the equivalence.
\end{proof}

\begin{restatable}{lemma}{lemmaleftprojectionmultiplier}
\label{lemma:left_projection_multiplier}
Fix $N, K, P^N, w, \beta$. Suppose that \Cref{assumption:left_projection_support} holds. Let $G$ be defined as
\begin{align}
G \colon \RR & \to \RR , \\
\nu & \mapsto \sum_{l \in [L^N] \colon \bar{\alpha}_l^N(P^N) > 0} \bar{\alpha}^N_l(P^N) \frac{\sum_{k=1}^K q^N_{w,\beta}(k,l,1)e^\nu}{\sum_{k=1}^K q^N_{w,\beta}(k,l,1)e^\nu + \sum_{k=1}^K q^N_{w,\beta}(k,l,2)}.
\end{align}
Then, for every $\psi \in (0,1)$, there exists a unique $\nu^\star \in \RR$ such that $G(\nu^\star) = \psi$.
\end{restatable}

\begin{proof}[Proof of \Cref{lemma:left_projection_multiplier}]
For any $l \in [L^N]$ such that $\bar{\alpha}_l^N(P^N) > 0$, \Cref{assumption:left_projection_support} gives that
\begin{align}
\sum_{k=1}^K q^N_{w,\beta}(k,l,1) > 0, \qquad \sum_{k=1}^K q^N_{w,\beta}(k,l,2) > 0.
\end{align}
Hence, for any such $l$, the map
\begin{align}
\nu \mapsto \frac{\sum_{k=1}^K q^N_{w,\beta}(k,l,1)e^\nu}{\sum_{k=1}^K q^N_{w,\beta}(k,l,1)e^\nu + \sum_{k=1}^K q^N_{w,\beta}(k,l,2)}
\end{align}
is continuous on $\RR$. Since $G$ is a finite sum of continuous functions, $G$ is continuous on $\RR$. Moreover, for any $l \in [L^N]$ such that $\bar{\alpha}_l^N(P^N) > 0$, we have, for any $\nu \in \RR$,
\begin{align}
\frac{d}{d\nu} \frac{\sum_{k=1}^K q^N_{w,\beta}(k,l,1)e^\nu}{\sum_{k=1}^K q^N_{w,\beta}(k,l,1)e^\nu + \sum_{k=1}^K q^N_{w,\beta}(k,l,2)} & = \frac{\left(\sum_{k=1}^K q^N_{w,\beta}(k,l,1)\right) \left(\sum_{k=1}^K q^N_{w,\beta}(k,l,2)\right)e^\nu}{\left( \sum_{k=1}^K q^N_{w,\beta}(k,l,1)e^\nu + \sum_{k=1}^K q^N_{w,\beta}(k,l,2) \right)^2} \\
& > 0.
\end{align}
Since $\sum_{l=1}^{L^N}\bar{\alpha}_l^N(P^N)=1$, there exists $l \in [L^N]$ such that $\bar{\alpha}_l^N(P^N) > 0$. Therefore, $G$ is strictly increasing on $\RR$. For any $l \in [L^N]$ such that $\bar{\alpha}_l^N(P^N) > 0$, we have
\begin{align}
\lim_{\nu \to -\infty} \frac{\sum_{k=1}^K q^N_{w,\beta}(k,l,1)e^\nu}{\sum_{k=1}^K q^N_{w,\beta}(k,l,1)e^\nu + \sum_{k=1}^K q^N_{w,\beta}(k,l,2)} & = 0 , \\
\lim_{\nu \to +\infty} \frac{\sum_{k=1}^K q^N_{w,\beta}(k,l,1)e^\nu}{\sum_{k=1}^K q^N_{w,\beta}(k,l,1)e^\nu + \sum_{k=1}^K q^N_{w,\beta}(k,l,2)} & = 1 ,
\end{align}
and hence, we have that
\begin{align}
\lim_{\nu \to -\infty} G(\nu) = 0, \qquad \lim_{\nu \to +\infty} G(\nu) = \sum_{l \in [L^N] \colon \bar{\alpha}_l^N(P^N) > 0} \bar{\alpha}_l^N(P^N) = 1.
\end{align}
Since $G$ is continuous, for every $\psi \in (0,1)$, there exists $\nu^\star \in \RR$ such that $G(\nu^\star) = \psi$. Since $G$ is strictly increasing, this $\nu^\star$ is unique.
\end{proof}

\begin{restatable}{lemma}{lemmaleftprojsolution}
\label{lemma:left_proj_solution}
Let $P^N \in \Mcal^{np,N}, \psi \in (0,1), w \in \Delta_K,\beta \in \Bcal^K$. Suppose that
\Cref{assumption:left_projection_support} holds. Then, there exists $\nu^\star \in \RR$ such that
\begin{align}
\sum_{l=1}^{L^N}
\bar{\alpha}^N_l(P^N) \frac{\sum_{k=1}^K q^N_{w, \beta}(k,l,1)e^{\nu^\star}}{\sum_{k=1}^K q^N_{w, \beta}(k,l,1)e^{\nu^\star} + \sum_{k=1}^K q^N_{w,\beta}(k,l,2)} = \psi ,
\end{align}
and \eqref{equation:left_projection_problem} admits a unique minimizer $r^\star$. For any $l\in[L^N]$ such that
$\bar\alpha_l^N(P^N)>0$, and $k \in [K]$, we have that
\begin{align}\label{equation:dlsfjskf}
r^\star(k,l,1) & = \bar{\alpha}^N_l(P^N)
\frac{q^N_{w, \beta}(k,l,1) e^{\nu^\star}}{\sum_{u=1}^K q^N_{w,\beta}(u,l,1)e^{\nu^\star} + \sum_{u=1}^K q^N_{w,\beta}(u,l,2)} , \\
r^\star(k,l,2) &= \bar\alpha_l^N(P^N)
\frac{ q^N_{w, \beta}(k,l,2)}{\sum_{u=1}^K q^N_{w, \beta}(u,l,1)e^{\nu^\star} + \sum_{u=1}^K q^N_{w,\beta}(u,l,2)}.
\end{align}
For any $l\in[L^N]$ such that $\bar{\alpha}^N_l (P^N) = 0$, we have for any $ k\in[K],\ m\in[2], r^\star (k,l,m) = 0$.
\end{restatable}

\begin{proof}[Proof of \Cref{lemma:left_proj_solution}]
For any $l \in [L^N]$ such that $\bar\alpha_l^N(P^N) = 0$, the constraint $\sum_{k=1}^K \sum_{m=1}^2 r(k, l, m)$ $= \bar{\alpha}^N_l(P^N) $ in $\Rcal^{N,K}(P^N, \psi)$ implies that
\begin{align}
\sum_{k=1}^K\sum_{m=1}^2 r(k, l, m) = 0,
\end{align}
and hence $r(k, l, m) = 0$ for any $k \in [K], m \in [2]$. Let $S = \{(k, l, m) \colon q_{w, \beta}^N(k, l, m) > 0\}$. Any feasible $r$ with finite $\KL(r \| q^N_{w, \beta})$ satisfies $r(k, l, m) = 0$ for any $(k, l, m) \notin S$. For any $(k, l, m) \in S, \lambda \in \RR^{L^N}$, and $\nu \in \RR$, consider the Lagrangian
\begin{align}
\Lcal(r, \lambda, \nu) &= \sum_{(k, l, m) \in S} r(k, l, m)\log\frac{r(k, l, m)}{q_{w, \beta}^N(k, l, m)} + \sum_{l=1}^{L^N}\lambda_l \left( \sum_{k=1}^K\sum_{m=1}^2 r(k, l, m) - \bar{\alpha}^N_l(P^N) \right) \\
&\quad - \nu \left( \sum_{k=1}^K \sum_{l=1}^{L^N} r(k, l, 1) - \psi \right) .
\end{align}
The first-order condition with respect to $r(k, l, m)$, for $(k, l, m) \in S$, gives that
\begin{align}\label{equation:rweoiru}
\log\frac{r(k, l, m)}{q_{w, \beta}^N(k, l, m)} + 1 + \lambda_l - \nu \1\{m = 1\} = 0 ,
\end{align}
and thus $r(k, l, m) = q_{w, \beta}^N(k, l, m) \exp\{-1 - \lambda_l + \nu \1\{m = 1\}\}$. The constraint $ \sum_{k=1}^K \sum_{m=1}^2 r(k, l, m) = \bar{\alpha}^N_l(P^N)$ gives, for any $l$ such that $\bar{\alpha}_l^N(P^N) > 0$,
\begin{align}
\exp(-1 - \lambda_l) = \frac{\bar{\alpha}^N_l (P^N)}{ \sum_{u=1}^K q_{w, \beta}^N(u, l, 1)e^\nu + \sum_{u=1}^K q_{w, \beta}^N(u, l, 2) }.
\end{align}
Therefore, any stationary point satisfying the constraints must be of the form
\begin{align}
r_\nu(k, l, 1) &= \bar\alpha_l^N(P^N) \frac{q_{w, \beta}^N(k, l, 1)e^\nu}{ \sum_{u=1}^K q_{w, \beta}^N(u, l, 1)e^\nu + \sum_{u=1}^K q_{w, \beta}^N(u, l, 2) }, \\
r_\nu(k, l, 2) &= \bar\alpha_l^N(P^N) \frac{q_{w, \beta}^N(k, l, 2)}{ \sum_{u=1}^K q_{w, \beta}^N(u, l, 1)e^\nu + \sum_{u=1}^K q_{w, \beta}^N(u, l, 2) }.
\end{align}
By \Cref{assumption:left_projection_support}, the denominators are strictly positive whenever $\bar\alpha_l^N(P^N) > 0$. By \Cref{lemma:left_projection_multiplier}, there exists $\nu^\star$ such that $G(\nu^\star) = \psi$, which yields the candidate minimizer $r^\star = r_{\nu^\star}$.

By construction, for any $l \in [L^N]$, we have that $\sum_{k=1}^K \sum_{m=1}^2 r^\star(k,l,m) = \bar{\alpha}^N_l(P^N)$. Moreover,
\begin{align}
\sum_{k=1}^K\sum_{l=1}^{L^N} r^\star(k,l,1) = G(\nu^\star) = \psi .
\end{align}
Since $\sum_{l=1}^{L^N} \bar{\alpha}_l^N(P^N)=1$, we also have that $\sum_{k=1}^K \sum_{l=1}^{L^N} r^\star(k, l, 2) = 1 -\psi $. Therefore, $r^\star \in \Rcal^{N,K}(P^N,\psi)$.

To verify optimality, since $r^\star$ satisfies \eqref{equation:rweoiru}, on the support of $q_{w, \beta}^N$, we have that
\begin{align}
\log \frac{r^\star(k, l, m)}{q_{w, \beta}^N(k, l, m)} =  \nu^\star \1\{m = 1\} - 1 -\lambda_l .
\end{align}
Hence, for any $r \in \Rcal^{N,K}(P^N, \psi)$ such that $\KL (r \| q_{w,\beta}^N)$ is finite, we have
\begin{align}
& \sum_{(k,l, m) \in S} \{r(k,l,m)-r^\star (k, l, m)\} \log \frac{r^\star(k,l,m)}{q_{w, \beta}^N(k,l,m)} \\
& \quad = \sum_{k,l,m}
\{r(k,l,m)-r^\star(k,l,m)\}
\left(-1-\lambda_l + \nu^\star \1\{m = 1\} \right) \\
& \quad = \sum_{l \in [L^N]} (-1-\lambda_l)
\sum_{k,m}\{r(k,l,m)-r^\star(k,l,m)\} + \nu^\star \sum_{k = 1}^K \sum_{l = 1}^{L^N} \{r(k, l, 1) - r^\star(k,l,1)\}.
\end{align}
Since $r$ and $r^\star$ satisfy the same marginal constraints, for every $l \in [L^N]$, we have
\begin{align}
\sum_{k=1}^K\sum_{m=1}^2 r(k,l,m) = \sum_{k=1}^K\sum_{m=1}^2 r^\star(k,l,m) = \bar{\alpha}_l^N(P^N) ,
\end{align}
and thus $\sum_{k=1}^K \sum_{m=1}^2
\{r(k,l,m)-r^\star(k,l,m) \}
=0$. Similarly, considering the constraint $\sum_{k=1}^K \sum_{l=1}^{L^N} r(k, l, 2)
$ $= \sum_{k=1}^K \sum_{l=1}^{L^N} r^\star(k, l, 2) = 1 - \psi$, we can conclude that
\begin{align}
    \sum_{(k,l, m) \in S} \{r(k,l,m)-r^\star (k, l, m)\}
\log \frac{r^\star(k,l,m)}{q_{w, \beta}^N (k,l,m)} = 0 .
\end{align}
Therefore, we have that
\begin{align}\label{equation:rweroiwjr}
\KL(r \| q_{w, \beta}^N) - \KL(r^\star \| q_{w, \beta}^N) = \sum_{k, l, m}r(k, l, m)\log\frac{r(k, l, m)}{r^\star(k, l, m)} = \KL(r \| r^\star) \geq 0.
\end{align}
If $\KL(r \| q_{w, \beta}^N) = +\infty$, \eqref{equation:rweroiwjr} is immediate. Therefore, we have that $r^\star$ is a minimizer. Since $\KL(r \| r^\star) = 0$ if and only if $r = r^\star$, the minimizer $r^\star$ is unique.
\end{proof}

\begin{lemma}\label{lemma:left_projection_smoothness}
Suppose that \Cref{assumption:convexity_beta_optim,assumption:interior_point,assumption:local_identifiability,assumption:smoothness_certify} hold. Then there exists an open neighborhood $U\subseteq\Theta_0$ of $\theta^\star$ such that, for every $\theta\in U$, the left projection problem admits a unique minimizer, and the mapping
\begin{align}
\eta\colon U \to\Rcal^{N,K}(P^N,\psi) \; , \; \theta \mapsto \argmin_{r\in\Rcal^{N,K}(P^N,\psi)} \KL(r \| q_{\iota(\theta)}^N)
\end{align}
is continuously differentiable.
\end{lemma}

\begin{proof}[Proof of \Cref{lemma:left_projection_smoothness}]
By \Cref{assumption:interior_point}, we have $r^\star = q_{\iota(\theta^\star)}^N
\in\Rcal^{N,K}(P^N,\psi)$, and thus
\begin{align}
\psi = \sum_{k=1}^K w_k^\star g(\beta_k^\star) \in(0,1) ,
\end{align}
where the strict inequalities follow from
\Cref{assumption:convexity_beta_optim}. By \Cref{assumption:interior_point,assumption:local_identifiability}, there exists an open
neighborhood $U \subseteq \Theta_0$ such that, for every $\theta \in U$,
all the weights $w_1, \ldots, w_K$ in $\iota(\theta)$ are positive and all the atoms $\beta_1, \ldots, \beta_K$ belong to $B_0$. Let
\begin{align}
S = \{l\in[L^N] \colon \bar{\alpha}^N_l(P^N) > 0  \} .
\end{align}
For any $l \in S$, we have $e^N_l > 0$. Moreover,
\Cref{assumption:convexity_beta_optim} gives
$\bar{t}^N_l > 0$, $g > 0$, and $1 - g > 0$. It follows that, for any $\theta \in U, k \in [K], l \in S, m \in [2]$
\begin{align}
q_{\iota(\theta)}^N (k,l,m) > 0 .
\end{align}
Thus \Cref{assumption:left_projection_support} holds for $q_{\iota(\theta)}^N$. By
\Cref{lemma:left_projection_multiplier,lemma:left_proj_solution}, the left projection therefore has a unique minimizer for every $\theta\in U$,
given by \eqref{equation:dlsfjskf}, with multiplier $\nu$ equal to the unique solution of
\begin{align}
G_\theta(\nu) = \psi .
\end{align}
Define $H \colon \Theta_0 \times \RR \to \RR , (\theta,\nu) \mapsto G_\theta(\nu)-\psi$. By the definition of $q_{\iota(\theta)}^N$ and
\Cref{assumption:local_identifiability,assumption:smoothness_certify},
the map $H$ is continuously differentiable on $U \times \RR$. Since $q_{\iota(\theta^\star)}^N = r^\star$ satisfies the defining marginal
restrictions of $\Rcal^{N,K}(P^N,\psi)$,
\begin{align}
H(\theta^\star, 0) = 0.
\end{align}
Writing
\begin{align}
A_l(\theta)=\sum_{k=1}^Kq_{\iota(\theta)}^N(k,l,1),
\quad B_l(\theta) = \sum_{k=1}^K q_{\iota(\theta)}^N (k,l,2),
\end{align}
we have
\begin{align}
\partial_\nu H(\theta, \nu) = \sum_{l \in S}\bar\alpha_l^N(P^N)
\frac{A_l(\theta)B_l(\theta)e^\nu}
     {\{A_l(\theta)e^\nu+B_l(\theta)\}^2}>0.
\end{align}
In particular, $\partial_\nu H(\theta^\star,0)>0$. The implicit function
theorem therefore yields an open neighborhood $V\subseteq U$ of
$\theta^\star$ and a continuously differentiable map
$\bar{\nu} \colon V \to \RR$ such that
\begin{align}
G_\theta(\bar{\nu}(\theta)) = \psi.
\end{align}
By the uniqueness assertion in
\Cref{lemma:left_projection_multiplier}, $\bar{\nu}(\theta)$ is precisely the multiplier appearing in \Cref{lemma:left_proj_solution}.

Substituting $\bar{\nu}(\theta)$ into \eqref{equation:dlsfjskf} shows that every
coordinate of the unique left projection $\eta(\theta)$ is continuously
differentiable on $V$; coordinates corresponding to $l\notin S$ are
identically zero. Hence $\eta$ is continuously differentiable on $V$, hence the result.
\end{proof}

\localbasincertify*

In order to help the understanding of the proof, we first provide a proof sketch of \Cref{theorem:local_basin_certify}.

\begin{proof}[Proof sketch of \Cref{theorem:local_basin_certify}]
Importantly, we mostly work in the local coordinates $\theta$ induced by $\iota$, and denote the left and right projection maps by
\begin{align}
\eta(\theta) &=\mathrm{LeftProjection}\bigl(\psi,\iota(\theta),P^N\bigr) , & \tau(r) & = \iota^{-1}\bigl(\mathrm{RightProjection}(K,\Bcal,r)\bigr).
\end{align}
The left and right projection update in these coordinates is therefore $T = \tau \circ \eta$. The assumptions and the preceding regularity results ensure that $\eta$, $\tau$, and $T$ are continuously differentiable in neighborhoods of $r^\star$ and $\theta^\star$. Since $r^\star = q^N_{\iota(\theta^\star)}$ belongs to
$\Rcal^{N,K}(P^N,\psi)$, its left projection is: $\eta(\theta^\star) = r^\star$. Local uniqueness of the right projection similarly gives $\tau(r^\star) = \theta^\star$, and hence $T(\theta^\star) = \theta^\star$, which proves that the optimum $\theta^\star$ is a fixed point of $T$. Let
\begin{align}
\Ecal(r,\theta)= \KL(r \| q^N_{\iota(\theta)}) ,
\; F(r,\theta) = \nabla_\theta\Ecal(r,\theta) , \; \rho(\theta) = \Ecal(\eta(\theta), \theta) .
\end{align}
The first-order condition for the right projection is
\begin{align}\label{equation:proof_sketch_foc}
F(\eta(\theta),T(\theta))=0.
\end{align}
Moreover, the left projection is taken over a set defined by fixed
affine marginal constraints. Its first-order condition and the differentiated feasibility constraints imply that
\begin{align}
\nabla_\theta \rho(\theta) = F(\eta(\theta), \theta) .
\end{align}
Define $u \colon \theta \mapsto \Ecal(r^\star,\theta)-\rho(\theta)$. Since $r^\star$ is feasible for the left projection, $u \geq 0$, while
$u(\theta^\star)=0$. Consequently, we have that $\nabla_\theta^ 2u(\theta^\star) \succeq 0$. The definition of $u$ also gives that
\begin{align}
\nabla_\theta^2 \Ecal(r^\star,\theta^\star)  = \nabla_\theta^2\rho(\theta^\star)+ \nabla_\theta^ 2u(\theta^\star).
\end{align}
Differentiating \eqref{equation:proof_sketch_foc} at $\theta^\star$ and
using the envelope identity yields
\begin{align}
\nabla_\theta T(\theta^\star)=\{\nabla_\theta^2 \Ecal(r^\star,\theta^\star)  \}^{-1} \nabla_\theta^2 u(\theta^\star) .
\end{align}
By \Cref{assumption:Icom_nonsingular}, $\nabla_\theta^2 \Ecal(r^\star,\theta^\star)  \succ 0$, whereas
\Cref{assumption:missing_info} gives that
$\nabla_\theta^2 \Ecal(r^\star,\theta^\star) -  \nabla_\theta^2 u(\theta^\star) = \nabla_\theta^2 \rho(\theta^\star) \succ 0$. Therefore,
\begin{align}
S
=
\{\nabla_\theta^2 \Ecal(r^\star,\theta^\star) \}^{-1/2} \nabla_\theta^ 2u(\theta^\star) \{\nabla_\theta^2 \Ecal(r^\star,\theta^\star) \}^{-1/2}
\end{align}
is symmetric and has all its eigenvalues in $[0,1)$ ($S \succeq 0$, and two lines of algebra give that $\text{Id} - S \succ 0$). Since
$\nabla_\theta T(\theta^\star)$ is similar to $S$, it is a strict contraction at $\theta^\star$ in the norm $\|x\|_{\nabla_\theta^2 \Ecal(r^\star,\theta^\star) } = \|\{\nabla_\theta^2 \Ecal(r^\star,\theta^\star) \}^{1/2} x\|_2$. By continuity of
$\nabla_\theta T$, this contraction holds uniformly on a sufficiently
small closed ball around $\theta^\star$. Thus this ball is invariant under $T$, and every initialization in the ball satisfies
\begin{align}
\|\theta^{(h)}-\theta^\star \|_{\nabla_\theta^2 \Ecal(r^\star, \theta^\star)} \leq y^h\| \theta^{(0)} - \theta^\star\|_{\nabla_\theta^2 \Ecal(r^\star,\theta^\star) } ,
\end{align}
for some $y\in(0,1)$, and we thus have that
\begin{align}
    \|\theta^{(h)}-\theta^\star \|_{\nabla_\theta^2 \Ecal(r^\star, \theta^\star)} \to 0 \text{ as } h \to \infty .
\end{align}
Finally, let $U$ be the image of this ball under $\iota$. For every
$(w^{(0)},\beta^{(0)})\in U$, we have
$\theta^{(h)}\to\theta^\star$ and
\begin{align}
d^{(h)}(\psi) = \KL (\eta(\theta^{(h)})
\| q^N_{\iota(\theta^{(h)})} ) = \rho(\theta^{(h)}) \to
\rho(\theta^\star) = 0 \text{ as } h \to \infty .
\end{align}
The set $U$ has positive relative Lebesgue measure, hence $\PP((w^{(0)}, \beta^{(0)}) \in U) > 0$, hence the result
\end{proof}

\begin{proof}[Formal proof of \Cref{theorem:local_basin_certify}]
    Let $U_1 \subseteq U_0$ be an open neighborhood of $(w^\star, \beta^\star)$ such that $\bar{U_1} \subseteq \mathrm{int}(\Delta_K) \times B_0^K$, and let $W_1 \subseteq W_0$ such that for any $r \in W_1$, we have $\mathrm{RightProjection}(K,\Bcal,r)$ $\in U_1$. Up to restricting $\Theta_0$ to a neighborhood of $\theta^\star$ on which $\eta$ is defined (see \Cref{lemma:left_proj_solution}), we can define
    \begin{align}
        \tau \colon W_1 & \to \Theta_0 \; , \;  r \mapsto  \iota^{-1}(\mathrm{RightProjection}(K,\Bcal,r)) , \\
        \eta \colon \Theta_0 & \to \Rcal^{N,K}(P^N, \psi) \; , \; \theta \mapsto  \mathrm{LeftProjection}(\psi, \iota(\theta), P^N) , \\
        T \colon \Theta_0 \cap \eta^{-1}(W_1) & \to \RR^{K-1} \times \RR^{K \dim(\Bcal)} \; , \; \theta \mapsto \tau (\eta(\theta)) ,
    \end{align}
    where the existence of $\eta$ is guaranteed by \Cref{lemma:left_proj_solution}, and the existence of $\tau$ is guaranteed by \Cref{assumption:local_selection}. Let $V = \Theta_0 \, \cap \, \eta^{-1}(W_1)$.
    
\textbf{Existence of a Fixed Point.} By \Cref{assumption:interior_point}, there exists $r^\star = q_{\iota(\theta^\star)} \in \Rcal^{N,K}(P^N, \psi)$, and $\theta^\star$ is thus a global minimizer of $\theta \mapsto \KL(r^\star \| q_{\iota(\theta)}) \geq 0$. By \Cref{assumption:local_selection}, we have that $\mathrm{RightProjection}(K,\Bcal,r^\star) = (w^\star, \beta^\star)$, and since $\iota$ is a bijection in a neighborhood of $\iota^{-1}(w^\star, \beta^\star)$, we have that $\iota^{-1}(\mathrm{RightProjection}(K,\Bcal,r^\star)) = \iota^{-1}((w^\star, \beta^\star))$, and hence $\tau(r^\star) = \theta^\star$, which gives that $T(\theta^\star) = \tau(\eta(\theta^\star)) = \theta^\star$. Therefore, $\theta^\star$ is a fixed point of the map $T$.

Define
\begin{align}
    \Ecal \colon \Rcal^{N, K}(P^N, \psi) \times V & \to \RR_+ \; , \; (r, \theta) \mapsto \KL(r \| q_{\iota(\theta)}) , \\
    F \colon \Rcal^{N, K}(P^N, \psi) \times V & \to \RR^{K - 1 + K d} \; , \; (r, \theta) \mapsto \nabla_\theta \Ecal(r, \theta) .
\end{align}
By \Cref{assumption:smoothness_certify}, there exists an open $B_0 \subseteq \mathrm{int}(\Bcal)$ such that for any $k \in [K], \beta_k \in B_0$, and for any $l \in [L^N], \bar{t}_l^N(\cdot) \in \Ccal^2(B_0), g \in \Ccal^2(B_0), g(\cdot) \in (0,1)$. Let $\Theta_1 = \{\theta \in V \colon \iota(\theta)_k \in B_0 \text{ for any } k \in \{K, \ldots, 2 K-1\}\}$; since $\iota(\theta^\star) = (w^\star, \beta^\star)$ with $\beta_k^\star \in B_0$ for any $k \in [K], \Theta_1$ is an open neighborhood of $\theta^\star$. We thus have that $\Ecal, F$ are respectively jointly $\Ccal^2$, and $\Ccal^1$ on $\Rcal^{N, K}(P^N, \psi) \times \Theta_1$. For any $\theta\in V$, the definition of $T$ gives
\begin{align}
    \iota(T(\theta))
    =\mathrm{RightProjection}
      \bigl(K,\Bcal,\eta(\theta)\bigr)
    \in U_1 .
\end{align}
Consequently, $T(\theta)$ is an interior minimizer of
\begin{align}
    \theta' \mapsto \Ecal(\eta(\theta), \theta').
\end{align}
Since this map is differentiable on $\Theta_1$, the first order condition yields, for any $\theta \in \eta^{-1}(W_0)$
\begin{align}\label{equation:rpwirpoi}
    F(\eta(\theta),T(\theta))
    = \nabla_{\theta}  \Ecal(\cdot, \cdot)
    \big|_{\eta(\theta), T(\theta)} = 0 .
\end{align}
We can then differentiate $\theta \mapsto \nabla_\theta \KL(r^\star \| q_{\iota(\theta)})$, which gives by \Cref{assumption:Icom_nonsingular}, that
\begin{align}
    \nabla_\theta F(r, \theta)\big|_{r^\star, \theta^\star} \succ 0 .
\end{align}

\paragraph{Show that $\nabla_r \Ecal(\eta(\theta), \theta) \nabla \eta(\theta) =0$.} Define, for any $l = 1, \ldots, L^N$, the vector $c_l$ as $c_l(k, l', m) = \1\{l = l'\}$, as well as $c_{L^N+1}(k, l, m) = \1\{m=1 \}$. For any $r \in \Rcal^{N, K}(P^N, \psi)$, the constraints satisfied by $r$ can thus be written $\langle c_l, r\rangle = \alpha^N_l (P^N)$ for any $l = 1, \ldots, L^N$, and $\langle c_{L^N +1}, r \rangle = \psi$. Let $C$ be the matrix whose $l$-th row is $c_l^T$, and $\alpha = ((\alpha^N_l(P^N))_{l \in [L^N]}^T, \psi)^T$, we thus have that $\Rcal^{N, K}(P^N, \psi)$ can be written
\begin{align}
    \Rcal^{N, K}(P^N, \psi) = \{r \in \RR_+^{2 \times K \times L^N} \colon C r = \alpha \} .
\end{align}
Therefore, the left projection optimization problem can be written
\begin{align}
    \min_{r \in \RR_+^{2 K L^N}} \Ecal(r, \theta) \; \text{ subject to } C r = \alpha .
\end{align}
The Lagrangian of the problem can be written
\begin{align}\label{equation:roeurioe}
    \Lcal(\theta, r, \lambda, \nu) = \Ecal(r, \theta) + \sum_{l = 1}^{L^N} \lambda_l (\langle c_l, r \rangle - \alpha_l) - \nu(\langle c_{L^N+1}, r \rangle - \psi)
\end{align}
Since \eqref{equation:roeurioe} is minimized at $r = \eta(\theta)$, we thus have that
\begin{align}
    \nabla_r \Ecal(\eta(\theta), \theta) + \sum_{l=1}^{L^N} \lambda(\theta) c_l - \nu(\theta) c_{L^N + 1} = 0 ,
\end{align}
and hence $\nabla_r \Ecal(\eta(\theta), \theta) = - \sum_{l=1}^{L^N} \lambda(\theta) c_l + \nu(\theta) c_{L^N + 1}$, which gives that
\begin{align}
    \nabla_r \Ecal(\eta(\theta), \theta) \in \Span(c_1, \ldots, c_{L^N+1}) .
\end{align}
Since for any $\theta \in \Theta_0, \eta(\theta)$ is feasible, we thus have that for any $\theta \in \Theta_0$
\begin{align}\label{equation:reirpeir}
    C \eta(\theta) = \alpha .
\end{align}
Since the right-hand side $\alpha$ is fixed and does not depend on $\theta$, and \Cref{lemma:left_projection_smoothness} ensures that $\eta \in \Ccal^1(V)$ for some neighborhood $V \subseteq \Theta_0$, we can thus differentiate \eqref{equation:reirpeir}, and obtain $C \nabla \eta(\theta) = 0$, which proves that any column of $\nabla \eta (\theta)$ belongs to $\ker (C)$. Since $\nabla_r \Ecal(\eta(\theta), \theta) \in \range (C^\top)$, and $\range(C^\top) = \ker(C)^\perp$, we thus have for any $\theta \in V$
\begin{align}\label{equation:roeurerer}
    \nabla_r \Ecal(\eta(\theta), \theta) \nabla \eta(\theta) = 0 .
\end{align}
By \Cref{assumption:smoothness_certify}, $\Ecal(\cdot, \theta)$ is $\Ccal^2$, and we showed that $\eta$ is $\Ccal^1$. Therefore, since for any $\theta \in \Theta_0, \rho(\theta) = \Ecal(\eta(\theta), \theta)$, we have that for any $\theta \in V \subseteq \Theta_0$
\begin{align}\label{equation:ereuoriueir}
    \nabla_\theta \rho(\theta) = \nabla_\theta \Ecal(\cdot, \cdot) \big|_{\eta(\theta), \theta} + \underbrace{\nabla_r \Ecal(\cdot, \cdot) \big|_{\eta(\theta), \theta} \nabla \eta(\theta)}_{= 0 \text{ by } \eqref{equation:roeurerer}} = F(\eta(\theta), \theta) .
\end{align}
Since \Cref{lemma:left_projection_smoothness} proves that there exists a neighborhood $V$ of $\theta^\star$ such that $\eta \in \Ccal^1(V)$, and we showed that $F \in \Ccal^1(\Rcal^{N, K}(P^N, \psi) \times \Theta_1)$. By composition, $\theta \mapsto F(\eta(\theta), \theta) \in \Ccal^1(\Theta_1 \cap V)$. Therefore, by \eqref{equation:ereuoriueir}, we have that $\rho \in \Ccal^2(\Theta_1 \cap V)$, and
\begin{align}\label{equation:reoriuwnvsd}
    \nabla_\theta^2 \rho \big|_{\theta^\star} = \nabla \{\theta \mapsto F(\eta(\theta), \theta) \}\big|_{\theta^\star} .
\end{align}
We can define 
\begin{align}
    u \colon \Theta_0 \to \RR \; , \; \theta \mapsto \Ecal(r^\star, \theta) - \rho(\theta) .
\end{align}
Since $r^\star \in \Rcal^{N, K}(P^N, \psi)$ is feasible for the left projection, we have that for any $\theta \in \Theta_0, u(\theta) \geq 0$, and $u(\theta^\star) = 0$. Since $\theta^\star$ is a global minimizer of $u$ in $\Theta_0$ (and is in the interior), and $u \in \Ccal^2(\Theta_0)$, we have that
\begin{align}
    \nabla_\theta u(\cdot)\big|_{\theta^\star} = 0 \; , \; \nabla^2_\theta u(\cdot)\big|_{\theta^\star} \succeq 0 ,
\end{align}
as well as
\begin{align}
    \nabla^2_\theta u(\cdot)\big|_{\theta^\star} = \nabla^2_\theta \KL(r^\star \| q_{\iota(\theta)}) \big|_{\theta^\star} 
    - \nabla^2_\theta \rho(\cdot) \big|_{\theta^\star} \eqsp,
\end{align}
and hence
\begin{align}\label{equation:rwieroivdns}
    \nabla^2_\theta \KL(r^\star \| q_{\iota(\theta)}) \big|_{\theta^\star} = \nabla^2_\theta \rho(\cdot) \big|_{\theta^\star} + \nabla^2_\theta u(\cdot)\big|_{\theta^\star} ,
\end{align}
where $\nabla^2_\theta u(\cdot)\big|_{\theta^\star}  \succeq 0$.

\textbf{Express $\nabla_\theta T$.} We showed that there exists a neighborhood of $\theta^\star$ on which $\eta$ is $\Ccal^1$. By \Cref{assumption:local_selection}, we thus have that there exists a neighborhood $\Theta_2 \subseteq \Theta_1$ of $\theta^\star$ such that $T \in \Ccal^1(\Theta_2)$. Since $F, \eta, T \in \Ccal^1(\Theta_2)$, we can differentiate \eqref{equation:rpwirpoi} at any $\theta \in \eta^{-1}(W_0)$, which gives that
\begin{align}
    \nabla_\theta F(\cdot, \cdot) \big|_{r^\star, \theta^\star} \nabla_\theta T(\cdot)\big|_{\theta^\star} + \nabla_r F (\cdot, \cdot) \big|_{r^\star, \theta^\star} \nabla_\theta \eta(\cdot)\big|_{\theta^\star} = 0 ,
\end{align}
and hence
\begin{align}\label{equation:wepriweopri}
    \nabla_\theta T(\cdot) \big|_{\theta^\star} = - (\nabla_\theta F(\cdot, \cdot) \big|_{r^\star, \theta^\star})^{-1} \nabla_r F(\cdot, \cdot) \big|_{r^\star, \theta^\star} \nabla \eta(\cdot)\big|_{\theta^\star} .
\end{align}
We can differentiate \eqref{equation:ereuoriueir}, and use the chain rule, which gives that
\begin{align}
    \nabla_\theta \{\theta \mapsto F(\eta(\theta), \theta)\}\big|_{\theta^\star} & = \nabla_r F(\cdot, \cdot) \big|_{r^\star, \theta^\star} \nabla_\theta \eta (\cdot)\big|_{\theta^\star}  + \nabla_\theta F(\cdot, \cdot) \big|_{r^\star, \theta^\star} \\
    & = \nabla_r F(\cdot, \cdot)\big|_{r^\star, \theta^\star} \nabla_\theta \eta(\cdot)\big|_{\theta^\star} + \nabla_\theta^2 \KL(r^\star \| q_{\iota(\theta)}) \big|_{\theta^\star} ,
\end{align}
and thus, using \eqref{equation:reoriuwnvsd}, we obtain that
\begin{align}
    \nabla_r F(\cdot, \cdot)\big|_{r^\star, \theta^\star} \nabla_\theta \eta(\cdot)\big|_{\theta^\star} = \nabla_\theta^2 \rho (\cdot)\big|_{\theta^\star} - \nabla_\theta^2 \KL(r^\star \| q_{\iota(\theta)} )\big|_{\theta^\star} = - \nabla^2_\theta u(\cdot)\big|_{\theta^\star} ,
\end{align}
where we use \eqref{equation:rwieroivdns} in the second equality. Plugging this expression in \eqref{equation:wepriweopri} yields
\begin{align}\label{equation:slfjskdfj}
    \nabla_\theta T(\cdot)\big|_{\theta^\star} = (\nabla_\theta F(\cdot, \cdot) \big|_{r^\star, \theta^\star})^{-1} \nabla^2_\theta u(\cdot)\big|_{\theta^\star} .
\end{align}
\textbf{Range of eigenvalues of $\nabla_\theta T(\theta^\star)$.} By \Cref{assumption:Icom_nonsingular}, there exists a symmetric positive-definite square root of $\nabla_\theta F\big|_{r^\star, \theta^\star}$, which we refer to as $(\nabla_\theta F(r^\star, \theta^\star))^{1/2}$, and it is invertible. By \eqref{equation:slfjskdfj}, we have that
\begin{align}
   \nabla_\theta T(\theta^\star) = (\nabla_\theta F(r^\star, \theta^\star))^{-1/2} S (\nabla_\theta F(r^\star, \theta^\star))^{1/2}
\end{align}
where $S = (\nabla_\theta F(r^\star, \theta^\star))^{-1/2} \nabla_\theta^2 u(\theta^\star) (\nabla_\theta F(r^\star, \theta^\star))^{-1/2} \succeq 0$, where $(\nabla_\theta F(r^\star, \theta^\star))^{-1/2} \succ 0$ is a symmetric matrix. Therefore, $\nabla_\theta T(\theta^\star)$, and $S$, have the same eigenvalues. We also have that
\begin{align}
    \text{Id} - S & = (\nabla_\theta F(r^\star, \theta^\star))^{-1/2} (\nabla_\theta F(r^\star, \theta^\star)) (\nabla_\theta F(r^\star, \theta^\star))^{-1/2} \\
    & \quad - (\nabla_\theta F(r^\star, \theta^\star))^{-1/2} \nabla_\theta^2 u(\theta^\star) (\nabla_\theta F(r^\star, \theta^\star))^{- 1/2} \\
    & = (\nabla_\theta F(r^\star, \theta^\star))^{-1/2} (\nabla_\theta F(r^\star, \theta^\star) - \nabla_\theta^2 u(\theta^\star)) (\nabla_\theta F(r^\star, \theta^\star))^{-1/2} \\
    & =  (\nabla_\theta F(r^\star, \theta^\star))^{-1/2} \nabla_\theta^2 \rho(\theta^\star) (\nabla_\theta F(r^\star, \theta^\star))^{-1/2} \succ 0 ,
\end{align}
where $\text{Id}$ denotes the identity matrix, and positive definiteness holds by \Cref{assumption:missing_info}. Therefore, we have that the eigenvalues of $S$ all lie in $[0, 1)$, and as $S$ is symmetric, we have that $\|S\|_2 < 1$ (where $\| \cdot \|_2$ is the operator norm defined, for any positive integer $d$, and matrix $A \in \RR^{d \times d}$, as $\| \cdot \|_2 \colon \RR^{d \times d} \to \RR_+, A \mapsto \sup_{x \ne 0} (\sum_{i=1}^d [Ax]_i^2)^{1/2} / (\sum_{i=1}^d x_i^2)^{1/2}$). Since $(\nabla_\theta F(r^\star, \theta^\star))^{1/2} \succ 0$, we define a norm $\|\cdot\|_{2, F(r^\star, \theta^\star))^{1/2}} = \|(\nabla_\theta F(r^\star, \theta^\star))^{1/2} \cdot \|_2$, and we have that $\|\nabla_\theta T(\theta^\star)\|_{2, F(r^\star, \theta^\star))^{1/2}} = \|S\|_2 < 1$. Let $y \in (\|S\|_2, 1)$, by continuity of $\nabla_\theta T(\theta^\star)$, there exists a convex neighborhood $\Theta_3 \subseteq \Theta_2$ of $\theta^\star$ such that $T \in \Ccal^1(\Theta_3)$. Therefore, $\nabla_\theta T$ is continuous on $\Theta_3$, and for any $\theta \in \Theta_3$, we have
\begin{align}
    \|\nabla_\theta T(\theta^\star) - \nabla_\theta T(\theta)\|_{2, F(r^\star, \theta^\star))^{1/2}} < y - \|S\|_2 .
\end{align}
For any $\theta \in \Theta_3$, we have, by the triangle inequality, that
\begin{align}
    \|\nabla_\theta T(\theta)\|_{2, F(r^\star, \theta^\star))^{1/2}} & = \|\nabla_\theta T(\theta) - \nabla_\theta T(\theta^\star) + \nabla_\theta T(\theta^\star) \|_{2, F(r^\star, \theta^\star))^{1/2}} \\
    & \leq \|\nabla_\theta T(\theta^\star)\|_{2, F(r^\star, \theta^\star))^{1/2}} + \|\nabla_\theta T(\theta^\star) - \nabla_\theta T(\theta)\|_{2, F(r^\star, \theta^\star))^{1/2}} \\
    & \leq y - \|S\|_2 + \|S\|_2 \\
    & = y .
\end{align}
Let $\theta_1, \theta_2 \in \Theta_3$, and consider the path $\gamma \colon [0,1] \to \Theta_3, t \mapsto \theta_2 + t (\theta_1 - \theta_2)$ (the range of $\gamma$ is contained in $\Theta_3$ by convexity of $\Theta_3$). Therefore, we have that
\begin{align}
    \|T(\theta_1) - T(\theta_2)\|_{2, F(r^\star, \theta^\star))^{1/2}} & = \|\int_{t=0}^1 \nabla T(\gamma(t)) (\theta_1 - \theta_2) dt\|_{2, F(r^\star, \theta^\star))^{1/2}} \\
    & \leq \|\int_{t=0}^1 \|\nabla T(\gamma(t))\|_{2, F(r^\star, \theta^\star))^{1/2}} \|\theta_1 - \theta_2\|_{2, F(r^\star, \theta^\star))^{1/2}} dt \\
    & \leq \|\int_{t=0}^1 y \|\theta_1 - \theta_2\|_{2, F(r^\star, \theta^\star))^{1/2}} dt \\
    & = y \|\theta_1 - \theta_2\|_{2, F(r^\star, \theta^\star))^{1/2}} .
\end{align}
Let $r_0 > 0$ such that $\Bar{\mathrm{B}}(\theta^\star, r_0) \subseteq \Theta_3$
Consequently, since $\theta^\star \in \Bar{\mathrm{B}}(\theta^\star, r_0)$, and $T(\theta^\star) = \theta^\star$, we have that for any $\theta \in \Bar{\mathrm{B}}(\theta^\star, r_0), \|T(\theta) - T(\theta^\star)\|_{2, F(r^\star, \theta^\star))^{1/2}} = \|T(\theta) - \theta^\star\|_{2, F(r^\star, \theta^\star))^{1/2}} \leq y \| \theta - \theta^\star \|_{2, F(r^\star, \theta^\star))^{1/2}} < r_0$. Consequently, if $\theta^{(0)} \in \Bar{\mathrm{B}}(\theta^\star, r_0)$, we can define  for any $h \geq 0, \theta^{(h+1)} = T(\theta^{(h)})$, and we have that $\theta^{(h)} \in \Bar{\mathrm{B}}(\theta^\star, r_0)$. A trivial induction also gives
\begin{align}
    \|\theta^{(h)} - \theta^\star\|_{2, F(r^\star, \theta^\star))^{1/2}} & = \|T(\theta^{(h-1)}) - \theta^\star\|_{2, F(r^\star, \theta^\star))^{1/2}} \\
    & \leq y \|\theta^{(h-1)} - \theta^\star\|_{2, F(r^\star, \theta^\star))^{1/2}} \\
    & \leq \ldots \\
    & \leq y^h \|\theta^{(0)} - \theta^\star\|_{2, F(r^\star, \theta^\star))^{1/2}} , 
\end{align}
and consequently, we have that
\begin{align}
     \|\theta^{(h)} - \theta^\star\|_{2, F(r^\star, \theta^\star))^{1/2}} \to 0 \text { as } h \to \infty .
\end{align}
Since $ \|\cdot \|_{2, F(r^\star, \theta^\star ))^{1/2}}$ is a norm, it implies that $\theta^{(h)} \to \theta^\star$ as $h \to \infty$.

Let $U = \iota(\Bar{\mathrm{B}}(\theta^\star, r_0))$ a neighborhood of $(w^\star, \beta^\star)$. Since $\iota$ is a $\Ccal^1$ bijection onto its image, $U$ is a neighborhood of $(w^\star,\beta^\star)$ in $\Delta_K\times\Bcal^K$, and it has positive measure. Consider any
initialization $(w^{(0)},\beta^{(0)}) \in U$, and define
\begin{align}
\theta^{(0)}
    =\iota^{-1}(w^{(0)},\beta^{(0)}),
\qquad
\theta^{(h+1)}=T(\theta^{(h)}).
\end{align}
Then $\theta^{(0)} \in \overline{\mathrm{B}}(\theta^\star,r_0)$, and the
preceding argument gives that $\theta^{(h)} \to \theta^\star$ as $h\to\infty$. For any $h \geq 0$, let $(w^{(h)},\beta^{(h)}) = \iota(\theta^{(h)}), r^{(h)} = \eta(\theta^{(h)}) = \mathrm{LeftProjection} (\psi, \iota(\theta^{(h)}), P^N)$. By the definition of $\rho$, we therefore have
\begin{align}
d^{(h)}(\psi) = \KL (r^{(h)} \| q^N_{w^{(h)}, \beta^{(h)}} ) = \KL (\eta(\theta^{(h)}) \| q^N_{\iota(\theta^{(h)})} )
= \rho (\theta^{(h)}) .
\end{align}
Since $\rho$ is continuous in a neighborhood of $\theta^\star$,
$\eta(\theta^\star)=r^\star$, and
$q^N_{\iota(\theta^\star)}=q^N_{w^\star,\beta^\star} = r^\star$, it follows
that
\begin{align}
d^{(h)}(\psi) \to \rho(\theta^\star) = \KL(r^\star \| q^N_{w^\star, \beta^\star}) = 0 \text{ as } h \to \infty ,
\end{align}
hence the first part of the claim. As a consequence, for any $\epsilon > 0$, there exists a positive integer $H$ such that for any $h \geq H$, we have that $d^{(h)}(\psi) < \epsilon$, and hence $\mathrm{Certify}(\psi, P^N, K, \epsilon, h) = \mathrm{True}$. The above holds if $(w^{(0)}, \beta^{(0)}) \in \iota(\Bar{\mathrm{B}}(\theta^\star, r_0))$. By \Cref{assumption:compact_continuity_certify}, $\Bcal$ is compact. Since $U$ has positive measure, we can conclude that for any $h \geq H, \PP(\mathrm{Certify}(\psi,P^N,K,\epsilon, h) = \mathrm{True}) > 0$.
\end{proof}

\subsection{Proofs of the instantiation example with a mixed MNL}

\lemmamnlgeneralconvexity*

\begin{proof}[Proof of \Cref{lemma:mnl_general_convexity}]
By \Cref{assumption:MNL_general_kernel}, $\Bcal$ is a convex subset of $\RR^d$. For any nonempty $\tilde{a} \subseteq [J]$, $y \in \tilde{a}$, and $\beta \in \RR^d$, we have
\begin{align}
\log s(\beta, \tilde{a}, y) = z_y^\top \beta-\log \left\{\sum_{j\in \tilde{a}} \exp(z_j^\top \beta) \right\}.
\end{align}
Consequently,
\begin{align}
\nabla_\beta^2\log s(\beta, \tilde{a}, y) = - \sum_{j \in \tilde{a}}s(\beta, \tilde{a}, j)
\left\{z_j-\sum_{a \in \tilde{a}}s(\beta, \tilde{a}, a)z_a\right\} \{z_j-\sum_{a\in \tilde{a}}s(\beta, \tilde{a}, a) z_a\}^\top
\preceq0.
\end{align}
Thus, $\beta \mapsto s(\beta, \tilde{a}, y)$ is log-concave on $\RR^d$ and hence on $\Bcal$. Moreover, it is strictly positive. For any $l \in [L^N]$, there exists $r \in [M]$ such that $a(l) = \tilde{a}_r$, and
\begin{align}
\bar{t}^N_l(\beta) = s(\beta, \tilde{a}_r, y(l)).
\end{align}
Therefore, $\bar{t}^N_l \colon \Bcal \to (0,\infty)$ is log-concave. Similarly,
\begin{align}
g(\beta) = s(\beta, \tilde{a}^\star, y^\star)
\end{align}
is log-concave. Since $\Card(\tilde{a}^\star)\geq2$ and every exponential term is strictly positive, we have that for any $\beta \in \Bcal, 0 < g(\beta) < 1$. Consequently, \Cref{assumption:convexity_beta_optim} holds.
\end{proof}

\lemmamnlgeneralcompletecurvature*

\begin{proof}[Proof of \Cref{lemma:mnl_general_complete_curvature}]
For any nonempty $\tilde{a} \subseteq [J], j\in \tilde{a}$, and $\beta \in \RR^d$
\begin{align}\label{equation:ewpriwirp}
& \bar{z}_{\tilde{a}}(\beta) = \sum_{j \in \tilde{a}} s(\beta, \tilde{a}, j) z_j \; , \; \Cov_{\tilde{a}}(\beta) = \sum_{j\in \tilde{a}} s(\beta, \tilde{a}, j) \{z_j-\bar{z}_{\tilde{a}}(\beta)\} \{z_j-\bar{z}_{\tilde{a}}(\beta)\}^\top , \\
& \Ical_{\mathrm{obs}}(\beta) = \sum_{r=1}^M e_r^N \Cov_{\tilde{a}_r}(\beta) \; , \; \Ical_g(\beta)=\frac{\nabla g(\beta)\nabla g(\beta)^\top}{g(\beta)\{1-g(\beta)\}} \; , \; \Ical_{\mathrm{com}}(\beta) = \Ical_{\mathrm{obs}}(\beta)+\mathcal I_g(\beta) .
\end{align}
For any nonempty $\tilde{a} \subseteq [J]$, and $j \in \tilde{a}$, the map $\beta \mapsto \log s(\beta, \tilde{a},j)$ is infinitely differentiable on $\RR^d$, since
\begin{align}\label{equation:weriwpeoir}
\log s(\beta, \tilde{a}, j) =z_j^\top\beta - \log\left(\sum_{j \in \tilde{a}} \exp(z_j^\top \beta) \right) ,
\end{align}
and the argument of the logarithm is strictly positive. Differentiating $\beta \mapsto \log s(\beta, \tilde{a}, j)$ with respect to $\beta$ gives that
\begin{align}\label{equation:wriewriu}
\nabla_\beta \log s(\beta, \tilde{a}, j) & = z_j-
  \frac{\sum_{j \in \tilde{a}} \exp(z_j^\top \beta) z_j}{\sum_{j \in \tilde{a}} \exp(z_j^\top \beta)} = z_j - \sum_{j \in \tilde{a}} s(\beta, \tilde{a}, j) z_j = z_j - \bar{z}_{\tilde{a}}(\beta) .
\end{align}
By definition of $s$ in \eqref{equation:s_instantiated_mnl}, $\nabla_\beta s(\beta, \tilde{a}, j) = s(\beta, \tilde{a}, j) \{z_j - \bar{z}_{\tilde{a}}(\beta)\}$, and we thus have
\begin{align}
\nabla_\beta \bar{z}_{\tilde{a}}(\beta) & = \sum_{j \in \tilde{a}} z_j s(\beta, \tilde{a}, j) \{z_j - \bar{z}_{\tilde{a}}(\beta)\}^\top \\
& = \sum_{j \in \tilde{a}} s(\beta, \tilde{a}, j) \{z_j - \bar{z}_{\tilde{a}}(\beta)\} \{z_j -\bar{z}_{\tilde{a}}(\beta)\}^\top \\
& = \Cov_{\tilde{a}}(\beta) ,
\end{align}
where we use
\begin{align}
    \sum_{j \in \tilde{a}} s(\beta, \tilde{a}, j) \bar{z}_{\tilde{a}}(\beta) \{z_j -\bar{z}_{\tilde{a}}(\beta)\}^\top = \bar{z}_{\tilde{a}}(\beta) \underbrace{\sum_{j \in \tilde{a}} s(\beta, \tilde{a}, j) z_j^\top}_{= \bar{z}_{\tilde{a}}(\beta)^\top} - \bar{z}_{\tilde{a}}(\beta) \bar{z}_{\tilde{a}}(\beta)^\top \underbrace{\sum_{j \in \tilde{a}} s(\beta, \tilde{a}, j)}_{=1} = 0 ,
\end{align}
and consequently, we have that
\begin{align}
\nabla_\beta^2\log s(\beta, \tilde{a}, j) = -\Cov_{\tilde{a}}(\beta).
\end{align}
\eqref{equation:weriwpeoir} shows that $\log g$ is infinitely differentiable on $\RR^d$, and we have
\begin{align}
\nabla_\beta^2\log g(\beta) & = \frac{\nabla_\beta^2g(\beta)}{g(\beta)} - \frac{\nabla g(\beta)\nabla g(\beta)^\top}{g(\beta)^2} , \\
\nabla_\beta^2 \log \{1-g(\beta) \}
& = - \frac{\nabla_\beta^2g(\beta)}{1-g(\beta)}
  - \frac{\nabla g(\beta)\nabla g(\beta)^\top}{\{1-g(\beta)\}^2} .
\end{align}
Since $\gamma(1, \beta) = g(\beta), \gamma(2, \beta) = 1 - g(\beta)$, and $0 < g(\beta) < 1$, we have that
\begin{align}
- \sum_{m=1}^2 \gamma(m,\beta) \nabla^2 \log\gamma(m,\beta) & = - g(\beta) \nabla_\beta^2\log g(\beta) - (1-g(\beta)) \nabla_\beta^2 \log\{1-g(\beta) \} \\
 & = \left\{\frac1{g(\beta)} + \frac1{1-g(\beta)}\right\} \nabla g(\beta)\nabla g(\beta)^\top \\
 & = \Ical_g(\beta)
\end{align}
By \Cref{assumption:MNL_general_richness}, for any $k\in[K], \Ical_{\mathrm{com}}(\beta_k^\star) \succ 0$.  Let $F \colon \beta \mapsto \left((s(\beta, \tilde{a}_r, y))_{r \in [M], y\in \tilde{a}_r}, g(\beta) \right)$. By definition of $\Ical_{\text{obs}}$ in \eqref{equation:ewpriwirp}, we have that
\begin{align}
    v^\top \Ical_{\text{obs}}(\beta) v & = \sum_{r=1}^M e_r^N v^\top \Cov_{\tilde{a}_r}(\beta) v \\
    & = \sum_{r=1}^M e_r^N v^\top \sum_{j\in \tilde{a}_r} s(\beta, \tilde{a}_r, j) \{z_j-\bar{z}_{\tilde{a}_r}(\beta)\} \{z_j-\bar{z}_{\tilde{a}_r}(\beta)\}^\top v \\
    & = \sum_{r=1}^M e_r^N \sum_{j\in \tilde{a}_r} s(\beta, \tilde{a}_r, j) v^\top \nabla_\beta \{\log s(\beta, \tilde{a}_r, j)\} \nabla_\beta \{\log s(\beta, \tilde{a}_r, j)\}^\top v \\
    & = \sum_{r=1}^M e_r^N \sum_{j\in \tilde{a}_r} v^\top s(\beta, \tilde{a}_r, j) \frac{\nabla_\beta s(\beta, \tilde{a}_r, j) \{\nabla_\beta s(\beta, \tilde{a}_r, j)\}^\top}{s(\beta, \tilde{a}_r, j)^2} v ,
\end{align}
and hence using the definition of $\Ical_{g}$ in \eqref{equation:ewpriwirp}, we have that for any $v \in \RR^d$
\begin{align}
v^\top \Ical_{\text{com}}(\beta) v = \sum_{r=1}^M \sum_{j \in \tilde{a}_r}
\frac{e_r^N}{s(\beta, \tilde{a}_r, j)} \{\nabla s(\beta, \tilde{a}_r, j)^\top v\}^2 + \frac{\{\nabla g(\beta)^\top v\}^2}{g(\beta) \{1-g(\beta)\}} .
\end{align}
Let $(w^\star_1, \ldots, w^\star_K, \beta^\star_1, \ldots, \beta^\star_K)$, and $r^\star$ be the parameters whose existence is given by \Cref{assumption:interior_point}. By definition of $F$, we have that $\nabla F(\beta_k^\star)v=0$ implies $v^\top \Ical_{\mathrm{com}}(\beta_k^\star) v = 0$. Since $\Ical_{\mathrm{com}}(\beta_k^\star)\succ 0$, it necessarily implies that $v = 0$. Consequently, $\nabla F(\beta_k^\star)$ has full column rank $d$. Therefore, by the inverse function theorem \citep[][Theorem 4.5]{lee2003smooth}, there exists a neighborhood $\tilde{V}_k$ of $\beta_k^\star$ on which $F$ is injective. Since $\beta_k^\star \in \operatorname{int}(\Bcal)$, we can choose open neighborhoods $V_k$ satisfying
\begin{align}
\beta_k^\star\in V_k , \quad \overline{V_k} \subseteq \tilde{V}_k\cap\operatorname{int}(\Bcal) .
\end{align}
Let $\Omega = \left\{u \in \RR^{K-1} \colon u_i > 0 , \sum_{i=1}^{K-1} u_i < 1 \right\}$. Since $w^\star\in\operatorname{int}(\Delta_K)$, there exists an open neighborhood $W$ of
$(w_1^\star,\ldots,w_{K-1}^\star)$ such that
$\overline W\subseteq\Omega$. Define
\begin{equation}\label{equation:wroiuweoiruwier}
\begin{aligned}
\Theta_0 & = W \times V_1 \times \cdots \times V_K , \\
\iota \colon \Theta_0 & \to \Delta_K \times \Bcal^K \; , \; (u,\beta_1,\ldots,\beta_K) \mapsto \left(u_1,\ldots,u_{K-1}, 1-\sum_{i=1}^{K-1}u_i, \beta_1,\ldots,\beta_K
\right) , \\
U_0 & = \iota(\Theta_0) .
\end{aligned}
\end{equation}
Then $\iota$ is a smooth bijection from $\Theta_0$ onto $U_0$,
$\iota(\theta^\star)=(w^\star,\beta^\star)$, and $\overline{U_0} \subseteq \operatorname{int}(\Delta_K) \times \operatorname{int}(\Bcal)^K$. To prove uniqueness, suppose that $(w,\beta) \in \overline{U_0}$ and $q_{w, \beta}^N = q_{w^\star, \beta^\star}^N$. For any $k \in [K]$, we have
\begin{align}
\sum_{l=1}^{L^N}\sum_{m=1}^2q_{w,\beta}^N(k, l, m)
& = w_k \left(\sum_{l=1}^{L^N}e_l^N\bar{t}^N_l(\beta_k) \right) \left(\sum_{m=1}^2\gamma(m,\beta_k)\right) = w_k ,
\end{align}
because both factors in parentheses equal one. Therefore,
$q_{w,\beta}^N=q_{w^\star,\beta^\star}^N$ implies that $w_k = w_k^\star$ for any $k \in [K]$. Similarly, summing over $m$ gives that for any $k \in [K], l \in [L^N], \bar{t}^N_l(\beta_k) = \bar{t}^N_l(\beta_k^\star)$, and summing over $l$ with $m=1$ gives that $g(\beta_k) = g(\beta_k^\star)$. Thus $F(\beta_k) = F(\beta_k^\star)$. Since
$\beta_k\in\overline{V_k}\subseteq\widetilde V_k$, and $F$ is injective on $\tilde{V}_k$, we obtain that $\beta_k = \beta_k^\star$ for any $k \in [K]$. Hence $(w, \beta) = (w^\star, \beta^\star)$. Therefore, \Cref{assumption:local_identifiability} is satisfied.

As shown above, $\bar{t}^N_l$ and $g$ are infinitely differentiable on $\RR^d$. Therefore, \Cref{assumption:smoothness_certify} holds.

Let $\theta \in \Theta_0$, and $(w_1, \ldots w_K, \beta_1, \ldots, \beta_K) = \iota(\theta)$. By definition of $\Theta_0$, we have that for any $k \in [K], w_k >0$, and $\bar{t}^N_l(\beta_k) > 0, \gamma(m, \beta_k) >0$ for any $m \in [2]$. By composition, the map $\theta \mapsto \KL(r^\star \|  q_{\iota(\theta)})$ is twice continuously
differentiable on $\Theta_0$. Moreover,
\begin{align}
\KL(r^\star \| q_{\iota(\theta)}) = C^\star + \Lcal_w(w) + \sum_{k=1}^K \Lcal_k(\beta_k),
\end{align}
where $C^\star$ is independent of $\theta$,
\begin{align}
\Lcal_w(w) & = -\sum_{i=1}^{K-1}w_i^\star\log w_i
-w_K^\star\log\left(1-\sum_{i=1}^{K-1}w_i\right) , \\
\Lcal_k(\beta) & = - w_k^\star\sum_{l=1}^{L^N}
e_l^N\bar t_l^N(\beta_k^\star)\log\bar t_l^N(\beta) -w_k^\star \sum_{m=1}^2 \gamma(m,\beta_k^\star) \log \gamma(m,\beta).
\end{align}
Therefore, for any $k, l$, we have that
\begin{align}
    \frac{\partial^2}{\partial w_l \partial \beta_k} \KL(r^\star \| q_{\iota(\theta)}) = 0 ,
\end{align}
and hence the Hessian of $\theta \mapsto \KL (r^\star \| q_{\iota(\theta)})$ is block diagonal. The $w$-block at $\theta^\star$ is
\begin{align}
\operatorname{diag}\left(
\frac1{w_1^\star},\ldots,\frac1{w_{K-1}^\star}
\right) + \frac1{w_K^\star} (1, \ldots, 1)^\top (1, \ldots, 1) \succ 0 .
\end{align}
For any $k\in[K]$, the $\beta_k$-block is
\begin{align}
- w_k^\star\sum_{l=1}^{L^N}
e_l^N \bar{t}^N_l(\beta_k^\star)
\nabla_\beta^2\log \bar{t}^N_l(\beta_k^\star)
- w_k^\star \sum_{m=1}^2 \gamma(m,\beta_k^\star)
\nabla_\beta^2 \log\gamma(m,\beta_k^\star) & =
w_k^\star\left\{ \Ical_{\mathrm{obs}}(\beta_k^\star)
+\Ical_g(\beta_k^\star)
\right\} \\
& = w_k^\star\Ical_{\mathrm{com}}(\beta_k^\star)
\succ0.
\end{align}
Thus the Hessian of $\theta \mapsto \KL(r^\star \| q_{\iota(\theta)})$ is positive definite, and
\Cref{assumption:Icom_nonsingular} holds.
\end{proof}

\lemmamnlgeneralprojectionregularities*

\begin{proof}[Proof of \Cref{lemma:mnl_general_projection_regularities}]
Let $(w^\star_1, \ldots, w^\star_K, \beta^\star_1, \ldots \beta^\star_K), r^\star$ be the parameters whose existence is given by \Cref{assumption:interior_point}, and $\theta^\star = \iota^{-1}(w^\star_1, \ldots, w^\star_K, \beta^\star_1, \ldots \beta^\star_K)$, where $\iota$ is defined in \eqref{equation:wroiuweoiruwier}, and its existence justified by \Cref{lemma:mnl_general_complete_curvature}. Suppose that \Cref{assumption:MNL_general_design,assumption:MNL_general_kernel,assumption:interior_point,assumption:MNL_general_richness,assumption:MNL_general_projection_identification} hold. By \Cref{lemma:mnl_general_convexity}, \Cref{assumption:convexity_beta_optim} holds, and by \Cref{lemma:mnl_general_complete_curvature}, we have that \Cref{assumption:local_identifiability,assumption:smoothness_certify,assumption:Icom_nonsingular} hold. Therefore, we have that \Cref{assumption:convexity_beta_optim,assumption:interior_point,assumption:local_identifiability,assumption:smoothness_certify} hold, and thus, \Cref{lemma:left_projection_smoothness} ensures the mapping
\begin{align}
\eta\colon U \to\Rcal^{N,K}(P^N,\psi) \; , \; \theta \mapsto \argmin_{r\in\Rcal^{N,K}(P^N,\psi)}\KL(r \| q_{\iota(\theta)}^N)
\end{align} is well-defined and $\Ccal^1$ in a neighborhood $\Theta_0$ of $\theta^\star$. By definition of $\rho$ and $\eta$, we have that for any $\theta \in \Theta_0$
\begin{align}
    \rho(\theta) = \KL(\eta(\theta) \| q_{\iota(\theta)}) ,
\end{align}
and as shown in the proof of \Cref{theorem:local_basin_certify}, there exists a neighborhood $\theta_1 \subseteq \Theta_0$ of $\theta^\star$ such that $\rho \in \Ccal^2(\theta_1)$.

Since $\Rcal^{N,K}(P^N,\psi)$ is affine, its feasible first-order perturbations $u$ of $r^\star$ are exactly those satisfying
\begin{align}\label{equation:weruweoiruwir}
\sum_{k=1}^K\sum_{m=1}^2 u(k,l,m)&=0 , \text{ for any } l \in [L^N] , \\
\sum_{k=1}^K\sum_{l=1}^{L^N}u(k,l,1)&=0,
\end{align}
where the analogous constraint for $m=2$ is redundant. Thus, we define
\begin{align}
\Ucal = \bigg\{u \colon \sum_{k=1}^K \sum_{m=1}^2 u(k,l,m) = 0
\text{ for any }l\in[L^N], \; \sum_{k=1}^K \sum_{l=1}^{L^N} u(k,l,1) = 0 \bigg\}
\end{align}
Since $r^\star$ belongs to the relative interior of $\Rcal^{N,K}(P^N,\psi)$, any such perturbation $t u, u \in \Ucal$ generates a feasible path for all sufficiently small $t$.

Let $u \in \Ucal$, and $v \in \RR^{K-1+Kd}$. Since $r^\star=q_{\theta^\star}$ has strictly positive coordinates, the mappings
$t \mapsto \log (r^\star+t u)$, and $t \mapsto q_{\iota(\theta^\star+ t v)}$ are twice differentiable in a neighborhood of $t=0$. We thus have that for $t$ small enough
\begin{align}
q_{\theta^\star+tv}(k,l,m) & = r^\star(k,l,m)
+t \nabla_\theta q_{\iota(\theta)} (k,l,m)^\top
\big|_{\theta^\star} v + \frac{t^2}{2} v^\top \nabla_\theta^2 q_{\iota(\theta)} (k,l,m) v
\big|_{\theta^\star} + o(t^2) .
\end{align}
A second-order expansion of each summand in the $\KL(r^\star+t u\|q_{\theta^\star + t v})$ gives that
\begin{align}
& \{r^\star(k,l,m) + t u(k,l,m)\} \log \frac{r^\star(k,l,m)+tu(k,l,m)}{q_{\theta^\star + t v}(k,l,m)} \nonumber \\
& \quad = t \left\{u(k,l,m)
-\nabla_\theta q_{\iota(\theta)}(k,l,m)^\top
\big|_{\theta^\star} v \right\} + \frac{t^2}{ 2r^\star(k,l,m)} \{ u(k,l,m)
-\nabla_\theta q_{\iota(\theta)}(k,l,m)^\top
\big|_{\theta^\star}v\}^2 \\
& \quad \quad - \frac{t^2}{2} v^\top\nabla_\theta^2 q_{\iota(\theta)}(k,l,m) v
\big|_{\theta^\star} + o(t^2) .
\end{align}
Summing the coordinatewise expansions therefore yields
\begin{align}
\KL(r^\star + tu \| q_{\theta^\star+tv}) = \frac{t^2}{2} \sum_{k=1}^K \sum_{l=1}^{L^N} \sum_{m=1}^2 \frac{\left\{u(k,l,m) -\nabla_\theta q_{\iota(\theta)}(k,l,m)^\top \big|_{\theta^\star}v
\right\}^2}{r^\star(k,l,m)} + o(t^2) .
\end{align}
By definition of $\rho$ and optimality of $\theta^\star$, we have that for $t$ small enough
\begin{align}
\rho(\theta^\star + tv) & \nonumber = \inf_{u \in \Ucal} \frac{t^2}{2} \sum_{k=1}^K \sum_{l=1}^{L^N}\sum_{m=1}^2
\frac{\left\{u(k,l,m) - \nabla_\theta q_{\iota(\theta)} (k,l,m)^\top \big|_{\theta^\star}v
\right\}^2 }{r^\star(k,l,m)} + o(t^2). \\
& \label{equation:roiweriufwuotiu} = \frac{t^2}{2} v^\top \nabla_\theta^2 \rho(\theta^\star) v + o(t^2) ,
\end{align}
and $v^\top \nabla_\theta^2 \rho(\theta^\star) v$ is null if, and only if every squared term $u(k,l,m) - \nabla_\theta q_{\iota(\theta)}(k,l,m)^\top \big|_{\theta^\star} v$ is equal zero, that is, if, and only if there exists $u \in \Ucal$ such that
\begin{align}
u(k, l, m) = \nabla_\theta  q_{\iota(\theta)}(k,l,m)^\top \big|_{\theta^\star} v
\end{align}
for any $k \in [K], l \in[L^N]$, and $m\in[2]$. Substituting this expression into the feasibility restrictions in \eqref{equation:weruweoiruwir} gives that
\begin{align}
\text{for any } l \in [L^N] , \quad \left. \sum_{k=1}^K \sum_{m=1}^2
\nabla_\theta q_{\iota(\theta)}(k,l,m)^\top v \right|_{\theta^\star} = 0 \text{, and } \; \sum_{k=1}^K \sum_{l=1}^{L^N}
\nabla_\theta q_{\iota(\theta)}(k,l,1)^\top v \big|_{\theta^\star} = 0 .
\end{align}
Equivalently, we have that
\begin{align}\label{equation:priwriwitu}
\left. \nabla_\theta \left(\left( \sum_{k=1}^K \sum_{m=1}^2 q_{\iota(\theta)}(k,l,m) \right)_{l \in [L^N]}, \sum_{k=1}^K \sum_{l=1}^{L^N} q_{\iota(\theta)} (k,l,1)
\right) v \right|_{\theta^\star} = 0.
\end{align}
Thus, the quadratic form in \eqref{equation:roiweriufwuotiu} is zero if, and only if the Jacobian in \eqref{equation:priwriwitu} cancels $v$.
By \Cref{assumption:MNL_general_projection_identification}, it implies that $v=0$. Consequently,
\begin{align}
\nabla_\theta^2 \rho(\theta)\big|_{\theta^\star} \succ 0 ,
\end{align}
and \Cref{assumption:missing_info} holds.

We now consider the right projection. From the expression in \eqref{equation:eroiueroiuwe}, the terms in the right-projection objective that depend on $w$ are
\begin{align}
-\sum_{k=1}^K\left\{\sum_{l=1}^{L^N}\sum_{m=1}^2r(k,l,m)\right\}\log w_k.
\end{align}
Since $r$ is a probability vector, for any $w \in \Delta_K$, we have that
\begin{align}
\sum_{k=1}^K
\left\{\sum_{l=1}^{L^N} \sum_{m=1}^2 r(k,l,m)\right\}
\log \frac{\sum_{l=1}^{L^N}\sum_{m=1}^2r(k,l,m)}{w_k} = \KL(\sum_{l=1}^{L^N} \sum_{m=1}^2 r(\cdot,l,m)\| w) \geq 0 ,
\end{align}
with equality if and only if
\begin{align}\label{equation:roiwjerlijdv}
w_k = \sum_{l=1}^{L^N} \sum_{m=1}^2r(k,l,m)
\end{align}
for any $k \in [K]$. Hence $(w_k)_{k \in [K]}$ in \eqref{equation:roiwjerlijdv} uniquely minimize the right-projection objective. The formula holds for any $r\in\Rcal^{N,K}(P^N,\psi)$; and since $r^\star$ has every entry strictly positive, there exists a neighborhood $U_r$ of $r^\star$ such that for any $r \in U_r, w_k >0$ for any $k \in [K]$, with $w, r$ given in \eqref{equation:roiwjerlijdv}.

For any $k\in[K]$, the part of the right-projection objective that depends on $\beta_k$ is
\begin{align}\label{equation:erpiwepriweut}
-\sum_{l=1}^{L^N}\sum_{m = 1}^2r^\star(k,l,m)
\left\{\log\bar t_l^N(\beta_k)+\log\gamma(m,\beta_k)\right\} .
\end{align}
Since $r^\star(k,l,m) = w_k^\star e_l^N \bar{t}^N_l (\beta_k^\star) \gamma(m, \beta_k^\star)$, the value of the objective in \eqref{equation:erpiwepriweut} at $\beta_k$ minus its value at $\beta_k^\star$ is
\begin{equation}\label{equation:rweoriuweioru}
\begin{aligned}
& \sum_{l=1}^{L^N}\sum_{m=1}^2 r^\star(k,l,m)
\log \frac{\bar{t}^N_l (\beta_k^\star) \gamma(m, \beta_k^\star)}{\bar{t}^N_l (\beta_k) \gamma(m, \beta_k)} \\
& \quad = w_k^\star \sum_{r=1}^M e^N_r \sum_{y \in \tilde{a}_r} s(\beta_k^\star, \tilde{a}_r, y) \log \frac{s(\beta_k^\star, \tilde{a}_r, y)}{s(\beta_k, \tilde{a}_r, y)} \\
& \quad \quad + w_k^\star g(\beta_k^\star)
\log \frac{g(\beta_k^\star)}{g(\beta_k)}
+ w_k^\star \{1 - g(\beta_k^\star)\} \log \frac{1-g(\beta_k^\star)}{1 - g(\beta_k)}.
\end{aligned}
\end{equation}
For any $\tilde{a}_r \subseteq [J]$, we have that
\begin{align}
\sum_{y\in \tilde{a}} s(\beta_k^\star, \tilde{a}_r, y) \log \frac{s(\beta_k^\star, \tilde{a}_r, y)}{s(\beta_k, \tilde{a}_r, y)} \geq 0,
\end{align}
with equality if and only if $s(\beta_k, \tilde{a}_r, y) = s(\beta_k^\star, \tilde{a}_r, y)$ for any $y \in \tilde{a}$. Similarly, for any $k \in [K]$, we have that
\begin{align}
g(\beta_k^\star) \log \frac{g(\beta_k^\star)}{g(\beta_k)} + \{1 - g(\beta_k^\star) \}
\log \frac{1 - g(\beta_k^\star)}{1 - g(\beta_k)} \geq 0 ,
\end{align}
with equality if and only if $g(\beta_k) = g(\beta_k^\star)$. Since $w_k^\star > 0$, and $e_r^N > 0$, the difference in \eqref{equation:rweoriuweioru} is nonnegative. It is zero only if
\begin{align}
\text{for any } r \in [M], y \in \tilde{a}_r, s(\beta_k, \tilde{a}_r, y) & = s(\beta_k^\star, \tilde{a}_r, y) , \text{ and } g(\beta_k) = g(\beta_k^\star) .
\end{align}
By \Cref{assumption:MNL_general_projection_identification}, these equalities imply that $\beta_k = \beta_k^\star$. Hence $\beta_k^\star$ is the unique minimizer of \eqref{equation:erpiwepriweut} over $\Bcal$.
Finally, differentiating twice the objective
\begin{align}
    \beta \mapsto - \sum_{l=1}^{L^N} \sum_{m=1}^2 r^\star(k,l,m) \left\{\log \bar{t}^N_l (\beta) + \log \gamma(m, \beta) \right\}
\end{align}
from \eqref{equation:erpiwepriweut} at $\beta  = \beta_k^\star$ (which is allowed since \Cref{assumption:smoothness_certify} holds by \Cref{lemma:mnl_general_complete_curvature}) gives that
\begin{align}
- \sum_{l=1}^{L^N} \sum_{m=1}^2 r^\star(k,l,m)
\nabla_\beta^2 \left\{\log\bar{t}^N_l(\beta) + \log\gamma(m,\beta)\right\}
\bigg|_{\beta_k^\star} & = w_k^\star \left\{
\Ical_{\mathrm{obs}}(\beta_k^\star) + \Ical_g (\beta_k^\star) \right\} \\
& = w_k^\star \Ical_{\mathrm{com}}(\beta_k^\star) \succ 0 ,
\end{align}
where the last inequality follows from
\Cref{assumption:MNL_general_richness}. Since $\beta_k^\star \in \mathrm{int}(\Bcal)$ and
\begin{align}
-\sum_{l=1}^{L^N} \sum_{m=1}^2 r^\star(k,l,m)
\nabla_\beta^2 \left\{\log\bar{t}^N_l(\beta) + \log \gamma(m,\beta) \right\} \bigg|_{\beta_k^\star} \succ 0 ,
\end{align}
continuity implies that there exists an open convex neighborhood $V_k$ of $\beta_k^\star$, with $\overline{V_k} \subseteq \mathrm{int}(\Bcal)$, such that, for every $r$ sufficiently close to $r^\star$ and every $\beta\in V_k$,
\begin{align}
-\sum_{l=1}^{L^N}\sum_{m=1}^2r(k,l,m)
\nabla_\beta^2
\left\{ \log \bar{t}^N_l(\beta) + \log \gamma(m,\beta) \right\} \succ 0 .
\end{align}
At $r = r^\star$, we have already shown that $\beta_k^\star$ is the unique minimizer of \eqref{equation:erpiwepriweut} over $\Bcal$. Therefore, by continuity and compactness of $\Bcal \setminus V_k$ (we chose $V_k$ open), we have
\begin{align}\label{equation:epriweirotwu}
\inf_{\beta\in\Bcal\setminus V_k}
\sum_{l=1}^{L^N} \sum_{m=1}^2 r^\star(k,l,m)
\log \frac{\bar{t}^N_l (\beta_k^\star) \gamma(m, \beta_k^\star)}{\bar{t}^N_l (\beta )\gamma(m, \beta)} > 0 .
\end{align}
The mappings $\log \bar{t}^N_l$ and $\log \gamma(m, \cdot)$ are continuous and bounded on the compact set $\Bcal$. Hence, for $r$ sufficiently close to $r^\star$, the strict inequality in \eqref{equation:epriweirotwu} continues to hold (possibly with a smaller positive lower bound). Consequently, there exists a neighborhood $V_r$ of $r^\star$ such that for any $r \in V_r$, we have that the minimizer of the problem
\begin{align}
\inf_{\beta \in \Bcal} - \sum_{l=1}^{L^N} \sum_{m = 1}^2 r(k,l,m) \left\{\log\bar t_l^N(\beta) + \log \gamma(m, \beta)\right\} .
\end{align}
belongs to $V_k$. Let $r \in V_r$. On $V_k$, the Hessian of the mapping
\begin{align}
\beta \mapsto - \sum_{l=1}^{L^N} \sum_{m=1}^2 r(k,l,m) \left\{ \log \bar{t}^N_l (\beta) + \log \gamma(m, \beta) \right\}
\end{align}
is positive definite, hence this mapping is strictly convex on $V_k$. It therefore has a unique minimizer in $V_k$. Since $V_k \subseteq \mathrm{int}(\Bcal)$, this minimizer satisfies
\begin{align}
\sum_{l=1}^{L^N} \sum_{m=1}^2 r(k,l,m)
\nabla_\beta \left\{\log\bar{t}^N_l (\beta) + \log \gamma(m,\beta) \right\} = 0 .
\end{align}
The derivative of the left-hand side with respect to $\beta$ is nonsingular on $V_k$. The implicit function theorem therefore implies that there exists a neighborhood $W_0$ such that the unique minimizer is continuously differentiable in $r$, for any $r \in W_0$. Since
\begin{align}
w_k = \sum_{l=1}^{L^N} \sum_{m=1}^2r(k,l,m)
\end{align}
is also unique and continuously differentiable in $r$, the right projection is unique and continuously differentiable on a neighborhood $W_0$ of $r^\star$. Hence \Cref{assumption:local_selection} holds.
\end{proof}

\end{document}